\documentclass[a4paper,fleqn]{cas-sc}

\hypersetup{ colorlinks=false }
\usepackage[numbers]{natbib}
\usepackage[labelsep=period]{caption}
\usepackage{subcaption}
\usepackage{placeins}
\usepackage{flafter}

\usepackage[contents={Preprint \today}, color=gray!70, scale=5, angle=45]{background}

\ExplSyntaxOn \cs_gset:Npn \__first_footerline: { } \cs_gset:Npn \__first_footerline_left: { } \cs_gset:Npn
\__first_footerline_right: { } \ExplSyntaxOff

\newtheorem{theorem}{Theorem}

\newtheorem{proposition}[theorem]{Proposition}
\newproof{proof}{Proof}

\newdefinition{remark}{Remark}

\def\tsc#1{\csdef{#1}{\textsc{\lowercase{#1}}\xspace}}
\tsc{WGM}
\tsc{QE}
\tsc{EP}
\tsc{PMS}
\tsc{BEC}
\tsc{DE}

\begin{document}

\let\WriteBookmarks\relax \def\floatpagepagefraction{1} \def\textpagefraction{.001}
\shorttitle{Submitted to Journal}
\shortauthors{X. XXX et~al.}

\title [mode = title]{A Fourth-order Conservative Adaptive Multiresolution Wavelet Upwind Scheme for Compressible Flows}                      

\author[1,2]{Bing Yang}
\credit{Conceptualization, Methodology, Investigation, Formal analysis, Software, Writing-original draft}

\affiliation[1]{organization={School of Civil Engineering and Mechanics, Lanzhou University}, city={LanZhou},
postcode={730000}, state={Gansu}, country={China}}

\affiliation[2]{organization={Key Laboratory of Mechanics on Disaster and Environment in Western China (Lanzhou
                University), the Ministry of Education, Lanzhou University}, city={Lanzhou}, postcode={730000},
                country={China}}

\author[1,2]{Xiaojing Liu}
\credit{Methodology, Software, Formal analysis}

\author[1,2]{Youhe Zhou}
\credit{Conceptualization, Methodology}

\author[3]{Feng Xiao}
\credit{Conceptualization, Methodology, Supervision, Writing-review and editing}
\affiliation[3]{organization={Department of Mechanical Engineering, Institute of Science Tokyo}, addressline={2-12-1
Ookayama, Meguro-ku}, city={Tokyo}, postcode={152-8550}, country={Japan}}

\author[1,2]{Jizeng Wang}
\cormark[1]
\cortext[cor1]{Corresponding author.}
\ead{jzwang@lzu.edu.cn}
\credit{Conceptualization, Methodology, Supervision, Funding acquisition, Writing-review and editing}

\begin{abstract}
A fourth-order conservative adaptive multiresolution average-interpolating wavelet upwind scheme is proposed for
compressible flows governed by hyperbolic conservation laws. A family of asymmetric average-interpolating wavelets with
upwind properties is constructed for conservative finite volume discretization, while symmetric average-interpolating
wavelets are employed for multiresolution decomposition and reconstruction of physical variables in the adaptive
procedure. This unified cell-average wavelet framework exploits the upwind property of asymmetric wavelets and the
data-compression efficiency of symmetric wavelets. Since both the conservative discretization and the adaptive
multiresolution representation are constructed from cell-average quantities, the proposed scheme preserves strict
conservation during both numerical evolution and adaptive cell redistribution. Unlike hybrid adaptive wavelet methods
that use wavelets mainly for data compression and mesh adaptation, the present adaptive wavelet upwind scheme utilizes
average-interpolating wavelet multiresolution approximation to reconstruct the interface values directly for numerical
flux evaluation, thereby avoiding additional ghost-cell marking and reconstruction near coarse--fine mesh interfaces.
The boundary variation diminishing (BVD) reconstruction is incorporated at the finest resolution level to achieve
non-oscillatory shock-capturing capability. Numerical tests demonstrate that the proposed scheme achieves the expected
fourth-order accuracy, maintains conservation errors close to machine precision, and controls numerical errors around
the prescribed threshold. The proposed method also sharply captures shock waves and contact discontinuities without
spurious oscillations and resolves multiscale smooth structures through a sparse adaptive representation. These results
indicate that the proposed scheme provides an efficient, conservative, and reliable approach for high-resolution
simulations of compressible flows.

\end{abstract}




\maketitle

\section{Introduction}
Compressible flows governed by hyperbolic conservation laws often involve multiscale smooth structures, such as acoustic
waves and vortices, as well as various discontinuities, including shock waves and contact
discontinuities~\cite{Shu2020}. The coexistence of these complex flow features poses a continuing challenge to the
development of accurate, robust, and efficient numerical schemes for compressible flows. When such problems are solved
on fixed meshes, a fine resolution must be employed over the entire computational domain to resolve multiscale smooth
structures and capture localized discontinuities. However, fine resolution is typically required only in limited regions
that contain discontinuities, steep gradients, or small-scale flow structures. Numerical methods on the fixed meshes
therefore introduce many unnecessary degrees of freedom in smooth large-scale regions, resulting in excessive
computational cost and memory consumption. In contrast, adaptive numerical methods provide an effective way to address
this difficulty by dynamically concentrating fine cells in regions with localized flow features while retaining coarse
cells in smooth regions. In this manner, they can significantly improve computational efficiency while maintaining the
required resolution and accuracy and therefore have become an important approach for high-resolution simulations of
large-scale compressible flows~\cite{Wang2013}. 

For compressible flows, adaptive numerical methods are generally expected to have several fundamental properties,
including conservativity, high-order accuracy, high-resolution shock capture capability, and reliable adaptive error
estimation. Conservativity is essential for ensuring that discontinuities propagate at the correct speeds and are
captured at the accurate locations~\cite{LeVeque1992}. High-order accuracy, together with high-resolution
shock-capturing capability, allows flow structures over a broad range of frequency components to be accurately resolved
using a limited number of degrees of freedom. In addition, reliable adaptive error estimation is required to assess and
control the accuracy of the computed solution~\cite{Tang2023}. 

Among various adaptive strategies, adaptive mesh refinement (AMR) is currently one of the most widely used techniques in
computational fluid dynamics. Therefore, we discuss the above properties of AMR-based adaptive methods. The construction
of conservative adaptive numerical methods requires two key ingredients: a conservative discretization on general
nonuniform meshes and a conservative reconstruction procedure during adaptive mesh redistribution. Conservative methods
formulated in terms of cell averages, such as the finite volume method (FVM), discontinuous Galerkin (DG)
method~\cite{Cockburn1989}, and multi-moment constrained finite volume method~\cite{Ii2009}, can naturally preserve
conservation on general nonuniform meshes. By contrast, conservative finite difference schemes based on nodal values,
i.e., the Shu--Osher conservative finite difference formulation~\cite{Shu1989}, guarantee conservation only on uniform
meshes and on certain nonuniform meshes generated by quadratic or exponential mappings, but are generally
non-conservative on arbitrary nonuniform meshes~\cite{Merriman2003}. In addition to the underlying discretization, the
reconstruction of physical variables during mesh adaptation must also be conservative. Sebastian and
Shu~\cite{Sebastian2003} noted that interpolation procedures not based on cell averages are non-conservative. Therefore,
cell-average-based conservative interpolation or reconstruction is required to maintain conservation during adaptive
mesh updates. Consequently, in the AMR framework, adaptive finite volume methods~\cite{Alauzet2016} and adaptive DG
methods~\cite{Remacle2003} generally preserve conservation, whereas adaptive finite difference methods usually fail to
achieve strict conservation on general adaptive meshes~\cite{Merriman2003,Chen2016,Li2015}.

Second-order AMR-based adaptive numerical methods have been well
developed~\cite{Alauzet2016,Tang2018AAAS,Alauzet2016Anisotropic}. To further improve their accuracy, resolution, and
computational efficiency, considerable efforts have been devoted to high-order adaptive numerical methods. Chen et
al.~\cite{Chen2016} combined high-order finite difference weighted essentially non-oscillatory (WENO) schemes with AMR
techniques and developed an adaptive method on curvilinear meshes. They noted that the resulting method is not strictly
conservative, with conservation errors remaining at the same order as the truncation errors of the scheme. Schaal et
al.~\cite{Schaal2015} constructed adaptive high-order DG methods and demonstrated their efficiency advantage over
low-order adaptive methods. Zhang and Groth~\cite{Zhang2011} developed an anisotropic block-based AMR technique combined
with a fourth-order finite volume central essentially non-oscillatory (ENO) scheme. Qi et al.~\cite{Qi2024} further
integrated a third-order discontinuous Galerkin method with the immersed boundary method to efficiently simulate
compressible flows around complex geometries on Cartesian meshes.

Error estimation in AMR techniques aims to identify flow regions that significantly affect the resolution of flow
structures or the prediction of aerodynamic quantities of interest. By evaluating local cell-wise errors, adaptive
indicators are devised to determine whether the corresponding cells should be refined or coarsened~\cite{Tang2023}.
Therefore, error estimation serves as a core ingredient of AMR techniques and is essential for ensuring the accuracy,
efficiency, and reliability of adaptive computations. Existing error-estimation strategies for AMR techniques can be
broadly classified into three categories: flow-feature-based methods~\cite{Tang2020flowfeature}, adjoint-based
methods~\cite{Park2004}, and Hessian-based methods~\cite{Habashi2000}. Flow-feature-based methods refine regions where
prominent structures, such as shocks, shear layers, and separated vortices, are expected to make dominant contributions
to the numerical error. Typical examples include indicators based on shock sensors~\cite{Tang2020flowfeature} and
vortex-identification criteria~\cite{Kamkar2011}. These methods are simple to implement and effective in improving the
resolution of targeted flow features. However, they are often problem-dependent and do not guarantee a reduction in the
global error~\cite{Roy2009}. Adjoint-based methods are goal-oriented approaches in which mesh refinement is guided by
the sensitivity of a prescribed quantity. By solving the corresponding adjoint problem, adjoint-weighted error
indicators can be constructed to identify regions that contribute most to the error in the target
functional~\cite{Park2004}. These methods can generate cost-effective meshes for practical simulations; however, the
construction and robust solution of suitable adjoint problems remain nontrivial, especially for nonlinear compressible
flows with discontinuities, thereby hindering their practical applications~\cite{Fidkowski2011,Alauzet2012Adjoint}.
Hessian-based criteria construct anisotropic meshes by equidistributing a local interpolation-error
estimate~\cite{Frey2005}. They can generate highly stretched and directionally aligned elements, thereby substantially
reducing the number of cells required to achieve a prescribed accuracy. However, the interpolation-error criterion does
not necessarily provide a direct estimate of the actual discretization error of the governing partial differential
equations (PDEs)~\cite{Alauzet2016Anisotropic}. In addition, its extension to high-order discretizations is nontrivial.
The standard Hessian approach is associated with linear interpolation error, whereas the interpolation error of a
$p$th-order approximation is governed by $(p+1)$th-order derivatives. Therefore, the construction of suitable high-order
metric tensors further increases the algorithmic complexity, especially in three dimensions~\cite{Fidkowski2011}.

Although conservative and high-order adaptive numerical methods have been developed within the AMR framework, there
remain several limitations. Some high-order adaptive schemes do not strictly preserve the conservation property on
general adaptive meshes. Existing indicators are often problem-dependent, and reliable error estimation remains
difficult for nonlinear compressible flows with discontinuities. Addressing these issues requires a cell-average-based
numerical tool capable of identifying, localizing and capturing flow features across different scales, while
quantitatively estimating their scale-dependent magnitudes. Wavelet theory naturally provides such a framework due to
its inherent multiresolution analysis, localization in both physical and scale spaces, and flexible choice of basis
functions ~\cite{Yang2024Advances}. These features make wavelets particularly attractive for developing adaptive
numerical methods for compressible flows.

The basic principle of wavelet-based adaptive methods is to use the magnitudes of wavelet coefficients as indicators of
local flow features. Mesh refinement is performed in regions where the wavelet coefficients exceed a prescribed
threshold, typically corresponding to local steep gradients, discontinuities, or high-frequency smooth
structures~\cite{Harten1995CPAM,Cohen2003,Vasilyev2000}. Accordingly, wavelet components whose coefficients are below
the threshold are discarded, resulting in a sparse multiresolution representation of the function. Wavelet-based
adaptive multiresolution methods therefore possess inherent advantages in adaptive indication, adaptive error
estimation, and sparse data representation. For the adaptive indication, wavelet multiresolution analysis and
localization in physical and scale spaces provide a natural way to identify, localize, and capture local solution
features. The decay behavior of wavelet coefficients can further be used to distinguish strong discontinuities from
high-frequency smooth structures. Together with wavelet-threshold filtering, these properties make adaptive mesh
updating straightforward~\cite{Vasilyev1996,Schneider2010}. From the viewpoint of error estimation, wavelet
multiresolution analysis can quantitatively characterize the magnitudes of solution features at different scales,
thereby enabling accurate error estimation for numerical solutions. Moreover, effective control of the global error can
be achieved by specifying appropriate thresholding parameters for wavelet coefficients at different resolution
levels~\cite{Harten1993}. From the perspective of function approximation, adaptive nonlinear wavelet approximation is
highly efficient for representing functions whose Besov regularity is stronger than their Sobolev regularity, such as
functions with localized steep gradients and boundary-layer structures. Consequently, high data compression ratios can
be obtained while preserving the essential localized features of the solution~\cite{DeVore1998,Deiterding2016}.

Wavelet-based adaptive numerical methods for compressible flows can be mainly classified into two categories according
to how wavelet multiresolution analysis is incorporated into the numerical framework. In the first category, scaling and
wavelet functions are applied directly as basis functions to construct the discretization scheme through wavelet
multiresolution approximations. Meanwhile, the multiresolution framework is also employed to design the adaptive
algorithm. These methods are referred to as pure wavelet numerical methods, mainly including wavelet Galerkin methods
and wavelet collocation methods~\cite{Schneider2010,Yang2023CF}. Minbashian et al.~\cite{Minbashian2017} proposed an
adaptive wavelet Streamline-Upwind Petrov-Galerkin (SUPG) method by borrowing the idea of the finite element SUPG
method. Pereira et al.~\cite{Pereira2023} showed that the dynamic adaptation procedure of adaptive methods intrinsically
introduces a non-smooth discrete operator associated with energy dissipation, which contributes to the stable solution
of hyperbolic conservation laws. In contrast, wavelet collocation methods are simpler to implement and computationally
more efficient. Early studies on wavelet collocation methods mainly employed symmetric interpolating wavelets as basis
functions and successfully applied them to small-viscosity convection--diffusion
problems~\cite{Vasilyev1996,Bertoluzza1996}. However, symmetric interpolating wavelet collocation schemes behave
similarly to central finite difference schemes for convective-term discretization, since their modified wave-number
analysis indicates zero numerical dissipation over the wave-number range. As a result, numerical oscillations may arise
in simulations of low-viscosity or inviscid compressible flows, thereby limiting their practical
applications~\cite{Yang2024Advances}. Inspired by classical upwind schemes in finite difference and finite volume
methods~\cite{Harten1983}, we have recently constructed asymmetric interpolating wavelet basis functions with
upwind-biased properties and proposed high-order adaptive multiresolution wavelet collocation upwind
schemes~\cite{Yang2023CF}. In addition to their aforementioned advantages in adaptive indication, adaptive error
estimation, and data compression, pure wavelet-based adaptive numerical methods can achieve improved spectral resolution
compared with conventional polynomial-based upwind schemes~\cite{Yang2024EML}. Moreover, their inherent low-pass
filtering property can effectively suppress high-frequency numerical oscillations, leading to improved accuracy and
stability in flow problems under extreme conditions~\cite{Yang2025AMS}. However, the wavelet basis functions used in the
above adaptive pure wavelet numerical methods, such as Daubechies wavelets~\cite{Restrepo1995}, semi-orthogonal
wavelets~\cite{Minbashian2017}, and interpolating wavelets~\cite{Vasilyev1996,Yang2023CF}, are generally not constructed
based on cell averages of the function. Consequently, these methods can at most preserve conservation in an approximate
sense, but cannot guarantee strict conservation. 

In addition, Alpert's multiwavelets~\cite{Alpert1993}, which provide $L^2$-basis functions through a multiresolution
analysis on partitioned elements, have also been employed in pure wavelet-based adaptive numerical methods for
conservation laws. In the adaptive multiresolution DG framework developed by M\"uller and
co-workers~\cite{Hovhannisyan2014,Gerhard2016,Gerhard2015}, the DG approximation space is represented by element-wise
scaling and wavelet functions of multiwavelets, and the governing equations are discretized in the corresponding
multiwavelet space. In such multiwavelet-based multiresolution DG methods, the conservation property is inherited from
the underlying DG weak formulation, together with the preservation of the cell-average coefficients in the multiwavelet
representation. Since multiwavelets may have jumps inside an element, the internal numerical flux does not cancel in the
evolution equations for the detail coefficients, and the corresponding integrals must be split across the internal
discontinuity. This additional treatment, together with multiscale transformations, thresholding, and adaptive grid
reconstruction, further increases the algorithmic complexity of multiwavelet-based adaptive
methods~\cite{Hovhannisyan2014}.

The second category refers to hybrid wavelet-based methods, where wavelet multiresolution analysis is only used in the
adaptive mesh-update procedure. In this type of method, wavelets are employed to detect discontinuities, estimate local
multiscale errors, guide mesh refinement or coarsening, and reconstruct data during adaptive mesh updates, while the
governing equations are discretized by classical numerical schemes. Harten~\cite{Harten1993} introduced symmetric
average-interpolating wavelets for constructing conservative adaptive multiresolution finite volume schemes. It should
be noted that, although these wavelets are constructed from cell averages, they retain many favorable properties similar
to those of interpolating wavelets within the multiresolution framework, including compact support, shape adjustability,
smoothness, polynomial reconstruction, and refinability~\cite{Harten1995CPAM,Harten1994JCP,Harten1997}. Most existing
hybrid wavelet-based adaptive methods are therefore constructed by coupling average-interpolating wavelets with
classical conservative discretization schemes, and thus generally preserve conservation. 

Harten and co-workers~\cite{Harten1995CPAM,Harten1994JCP,Harten1996SIAM} developed conservative adaptive finite volume
shock-capturing schemes based on symmetric average-interpolating wavelets. In these methods, the wavelet multiresolution
decomposition of physical variables is performed on a uniform mesh to locate discontinuities. Then, ENO reconstruction
is applied only in the vicinity of discontinuities, while simpler linear schemes are used in smooth regions. This
strategy reduces the computational cost associated with ENO reconstruction while retaining the shock-capturing
capability of the scheme. However, both the spatial discretization and time evolution are still carried out on the
finest uniform mesh. Consequently, the reduction in overall computational cost is only partially achieved. Cohen et
al.~\cite{Cohen2003} extended Harten's approach to dynamically adaptive meshes and proposed fully adaptive
multiresolution finite volume schemes for conservation laws, in which second-order ENO schemes were combined with
average-interpolating wavelets. Roussel et al.~\cite{Roussel2004} subsequently extended the fully adaptive
multiresolution finite volume scheme to parabolic problems. Castro et al.~\cite{Castro2016} developed high-order
adaptive multiresolution finite volume methods by combining high-order WENO schemes with average-interpolating wavelets.
Their results showed that high-order adaptive methods require fewer computational resources and achieve higher
efficiency than low-order ones. Since hybrid adaptive wavelet methods discretize the governing equations using classical
numerical schemes, additional ghost cells must be marked near coarse--fine mesh interfaces for flux evaluation. The
required physical variables are reconstructed through wavelet multiresolution approximation before the numerical fluxes
are computed by the underlying scheme. This extra reconstruction procedure increases algorithmic complexity and may
reduce computational efficiency.

These observations suggest a promising alternative: using average-interpolating wavelets both as basis functions for the
conservative discretization and as the foundation of the adaptive strategy. Such a formulation is expected to inherit
many advantages of existing pure wavelet methods, including high-order accuracy, high resolution, robustness, reliable
adaptive error estimation. At the same time, wavelet basis functions constructed from cell averages provide a natural
mechanism for enforcing strict conservation. Moreover, because the discretization and the adaptive representation are
constructed within the same wavelet multiresolution framework, additional ghost-cell marking and reconstruction near
coarse--fine cell interfaces are no longer required for flux evaluation. All interface fluxes can instead be computed
directly and efficiently through the wavelet multiresolution approximation. In this study, we develop a conservative
adaptive multiresolution average-interpolating wavelet upwind scheme that goes beyond using wavelet multiresolution
analysis merely for adaptive procedures. The proposed conservative adaptive wavelet upwind scheme fully exploits the
upwind stability of asymmetric average-interpolating wavelets for conservative discretization and the data-compression
efficiency of symmetric wavelets for adaptive representation within the unified average-interpolating wavelet framework,
thereby providing a novel and efficient approach for high-resolution simulations of compressible flows.

The rest of this paper is organized as follows. Section~\ref{Sec:Theory_and_methods} introduces the multiresolution
analysis of average-interpolating wavelets and presents the detailed construction of general average-interpolating
wavelets. The main ingredients of the proposed wavelet upwind scheme on uniform cells and the conservative adaptive
multiresolution wavelet upwind scheme are then described. In Section~\ref{Sec:Numerical_test}, the numerical performance
of the proposed adaptive wavelet upwind schemes is examined through a series of benchmark tests. The main conclusions
are summarized in Section~\ref{Sec:Conclusion}.
\section{Fundamental theory and numerical methods}
\label{Sec:Theory_and_methods}
As discussed above, constructing wavelet bases from cell averages is essential for developing a conservative adaptive
wavelet upwind scheme. In this section, we start with the multiresolution analysis of average interpolating wavelets,
and consider the following one-dimensional scalar conservation law as the basic model:
\begin{equation}
u_t + f(u)_x = 0.
\label{eq:model}
\end{equation}
We then elaborate on the fundamental ingredients of the proposed conservative adaptive wavelet upwind scheme.
\subsection{Multiresolution analysis of average-interpolating wavelets} 
Average-interpolating wavelets were constructed in the 1990s, driven by the needs of diverse
applications~\cite{Harten1996SIAM, Donoho1994}. In computational fluid dynamics community, Harten's pioneering work on
the average-interpolating wavelet originated with the development of adaptive multiresolution numerical methods for
hyperbolic conservation laws~\cite{Harten1995CPAM}. The average-interpolating wavelets proposed by Harten mainly
concentrate on the multiresolution representation of functions based on cell averages. Since the proposed adaptive
wavelet upwind scheme exploits the wavelet multiresolution approximation of functions to construct the upwind
discretization scheme, we present the fundamental theory of the average-interpolating wavelets from the perspective of
function approximation. 
\subsubsection{Biorthogonal multiresolution analysis}
Before introduce the biorthogonal multiresolution analysis, we first give some primary definitions. A Hilbert space
denoted by $L^2({\Omega})$ is the collection of square integrable functions over $\Omega$ that satisfy $\int_{\Omega}
|f(x)|^2 \mathrm{d}x < \infty$, where ${\Omega}$ is any open subset of $\mathbb{R}$ ~\cite{Ciarlet2013}. The resulting
$L^2({\Omega})$-norm of the function is defined by
\begin{equation}
\|f\|_{L^2(\Omega)} = \left( \int_{\Omega} |f(x)|^2\, \mathrm{d}x \right)^{1/2}.
\label{eq:l2norm}
\end{equation}
The space $L^2({\Omega})$ is naturally equipped with the following inner product:
\begin{equation}
(f,g)= \int_{\Omega} f(x)g(x)\,\mathrm{d}x, \mathrm{whenever} f,g \in L^2({\Omega}).
\label{eq:innerp}
\end{equation}
In addition, the $L^{\infty}({\Omega})$-norm is defined by 
 \begin{equation}
\|f\|_{L^{\infty}(\Omega)} = \sup_{x \in \Omega}{\{|f(x)|\}}.
\label{eq:lfnorm}
\end{equation}

A multiresolution analysis defined in $L^2(\mathbb{R})$ space consists of a sequence of closed subspaces $V_J$,
satisfying the following relations:
\begin{equation}
\cdots V_{-1} \subset V_{0} \subset V_1 \cdots \subset V_J \subset V_{J+1} \cdots,
\label{eq:nestw}
\end{equation}

\begin{equation}
\overline{\bigcup_{J \in \mathbb{Z}} V_J} = L^2(\mathbb{R}),
\end{equation}

\begin{equation}
\bigcap_{J \in \mathbb{Z}} V_J = \{0\},
\end{equation}

\begin{equation}
f \in V_J \iff f(2^{-J}\cdot) \in V_0,
\end{equation}

\begin{equation}
f \in V_0 \iff f(\cdot - k) \in V_0.
\end{equation}
In addition, there exists an auxiliary function named as a scaling function $\phi$ such that
\begin{equation}
\{\phi(x-k)|k \in \mathbb{Z}\} \text{ is a Riesz basis of } V_0.
\label{eq:basisw}
\end{equation}

The basic principle of the multiresolution analysis is that whenever a collection of closed spaces satisfies Eqs.
~\eqref{eq:nestw}-\eqref{eq:basisw}, then there exists a wavelet basis $\{{\psi}_{J,k}(x)| J, k \in \mathbb{Z}\}$,
$\psi_{J,k}=\psi(2^Jx-k)$ of the complement space $W_J$ of $V_J$ in $V_{J+1}$. Due to $V_J \subset V_{J+1}$ and $W_J
\subset V_{J+1}$, the scaling function and wavelet function both satisfy the following refinement relation:
\begin{equation}
\phi(x)=\sum_{k}{h_k\phi(2x-k)},\,\psi(x)=\sum_{k}{g_k\phi(2x-k)},
\label{eq:refinew}
\end{equation}
where $h_k$ and $g_k$ are filter coefficients for the scaling function and wavelet function, respectively.

The biorthogonal multiresolution analysis is composed of a primal multiresolution analysis and a dual multiresolution
analysis. Here, the multiresolution analysis introduced above is regarded as the primal multiresolution analysis. The
dual multiresolution analysis consists of a sequence of closed subspaces $\tilde{V}_J$. The corresponding dual scaling
function and wavelet function are denoted by $\tilde{\phi}(x)$ and $\tilde{\psi}(x)$, which also follow the refinement
relations with filter coefficients $\tilde{h}_k$ and $\tilde{g}_k$, respectively. The biorthogonality implies that the
dual scaling and wavelet functions are orthogonal to their primal counterparts in the sense of inner product, namely
\begin{equation}
\begin{aligned}
\left( \phi(x), \tilde{\psi}(x-l) \right)
&= \left( \psi(x), \tilde{\phi}(x-l) \right) = 0, \\
\left( \phi(x), \tilde{\phi}(x-l) \right)
&= \left( \psi(x), \tilde{\psi}(x-l) \right) = \delta_{0,l},
\end{aligned}
\label{eq:bior}
\end{equation}
where $\delta_{i,j}$ is the Kronecker function, $\delta_{i,j}=1$, if $i=j$; otherwise, $\delta_{i,j}=0$.  

Then, a projection operator $\mathcal{P}_J f: L^2 \to V_J$ can be found in the biorthogonal setting as 
\begin{equation}
\mathcal{P}_Jf(x)=\sum_{k}{c_{J,k}\phi_{J,k}},
\label{eq:pj}
\end{equation}
where $c_{J,k}=(f,\tilde{\phi}_{J,k}), \, \tilde{\phi}_{J,k}=2^J\tilde{\phi}(2^Jx-k), \, \phi_{J,k}=\phi(2^Jx-k)$. Since
$V_{J+1}=V_J \bigoplus W_J$, the projection operator $\mathcal{P}_{J+1}$ can be decomposed accordingly as
\begin{equation}
\mathcal{P}_{J+1}f(x)=\sum_{k}{c_{J,k}\phi_{J,k}}+\sum_{m}{d_{J,m}\psi_{J,m}},
\label{eq:pj1}
\end{equation}
where $d_{J,m}=(f,\tilde{\psi}_{J,m}), \, \tilde{\psi}_{J,m}=2^J\tilde{\psi}(2^Jx-m)$. The multiresolution analysis is
said to be of $N$th order if the scaling function can exactly reproduce polynomials of degree up to $N-1$:
\begin{equation}
\mathcal{P}_{J}x^{p}=x^p,\quad 0 \le {p} < N.
\label{eq:prep}
\end{equation}

\subsubsection{Construction of general average-interpolating wavelets}
\label{subsec:Constructwavelet}
We have summarized the properties of the scaling function for numerical computation and analysis in our previous study
\cite{Yang2023CF}, including compact support, interpolation, polynomial reproduction, smoothness, refinability and shape
adjustability. For solving elliptic PDEs, symmetric wavelets are advantageous, as their linear phase property enhances
data compression rates~\cite{Daubechies1992}, while the symmetry of the discretization operators ensures improved
numerical stability. In contrast, for compressible flows governed by hyperbolic conservation laws, the upwind property
of the scheme necessitates the use of asymmetric wavelet bases~\cite{Yang2023CF}. From the perspective of adaption
process, symmetric wavelets with the linear phase property yield higher data compression rates~\cite{Yang2024EML}.
Furthermore, to guarantee conservation property, wavelet functions must possess the cell-average interpolation property
in both the discretization and adaptation processes. These requirements suggest that a conservative adaptive wavelet
upwind scheme is preferably constructed based on both symmetric and asymmetric average-interpolating wavelets. However,
existing studies have only reported the construction of symmetric average-interpolating wavelets~\cite{Harten1995CPAM,
Sweldens2005}, whereas the construction of asymmetric counterparts, although conceptually similar, has not been
systematically formulated in terms of explicit expressions and detailed construction procedures. 

Next, we present a general construction framework for average-interpolating wavelets using an integral-average
interpolation technique together with the refinement relation in Eq.~\eqref{eq:refinew}. First, we describe the detailed
procedure for constructing the scaling functions as follows:   
\begin{enumerate}[(1)]
\item Choose the type of the required wavelet and determine its multiresolution analysis order $N$: (a) for symmetric
wavelets, $N \in \text{odd}$ is the unique choice; (b) for asymmetric wavelets, $N \in \text{even}$ or $N \in
\text{odd}$ is optional. 

\item Select a cell stencil. Specify $N$ cells in $V_0$ space as the base and choose two interpolation cells in the
$V_1$ space. We define the difference between the number of cells in the $V_0$ space on the left side of the
interpolation cells in $V_1$ and that on the right side as a stencil bias parameter $b$. Obviously, $b=0$ for symmetric
wavelets, and $0 < | b | < N$ for asymmetric wavelets. Here, we show stencils for $N=5$ symmetric wavelet and $N=4,b=1$
asymmetric wavelet as examples in Fig.~\ref{fig:waveletstencil}. 
\begin{figure}
    \centering

    \begin{subfigure}[b]{0.45\textwidth}
        \centering
        \includegraphics[width=\linewidth]{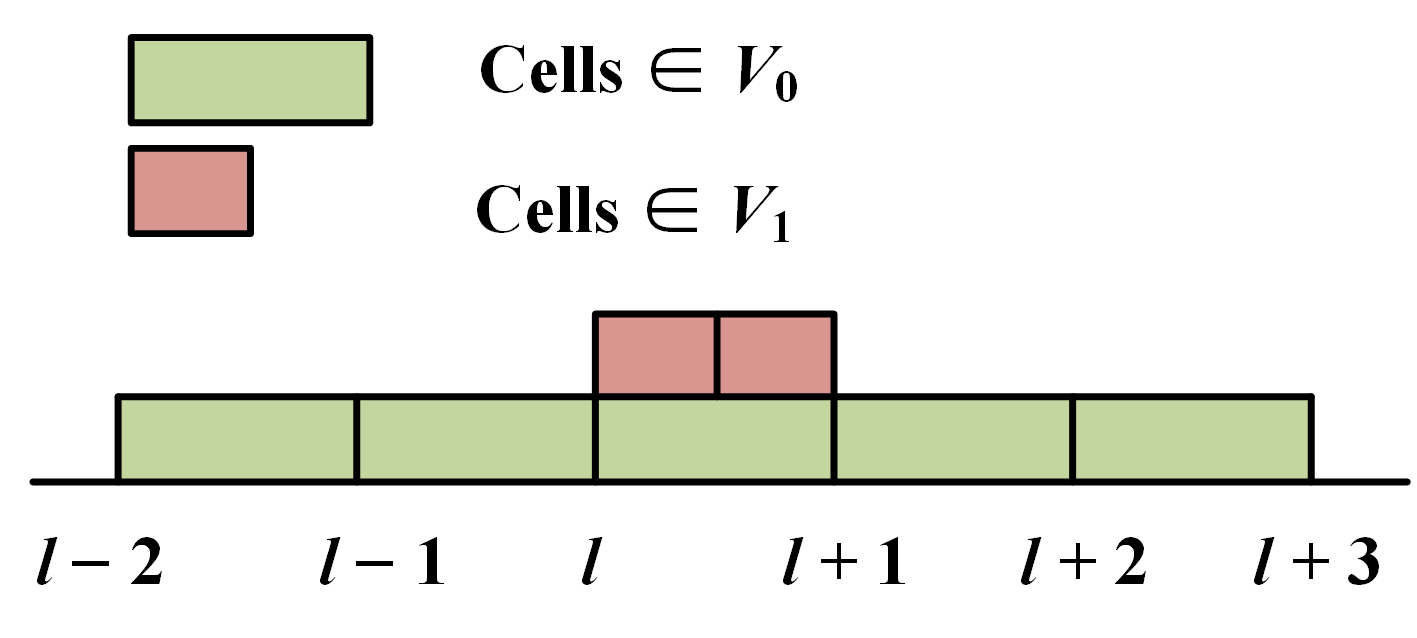}
        \caption{$N=5$}
        \label{fig:N5b0stencil}
    \end{subfigure}
    \hfill
    \begin{subfigure}[b]{0.40\textwidth}
        \centering
        \includegraphics[width=\linewidth]{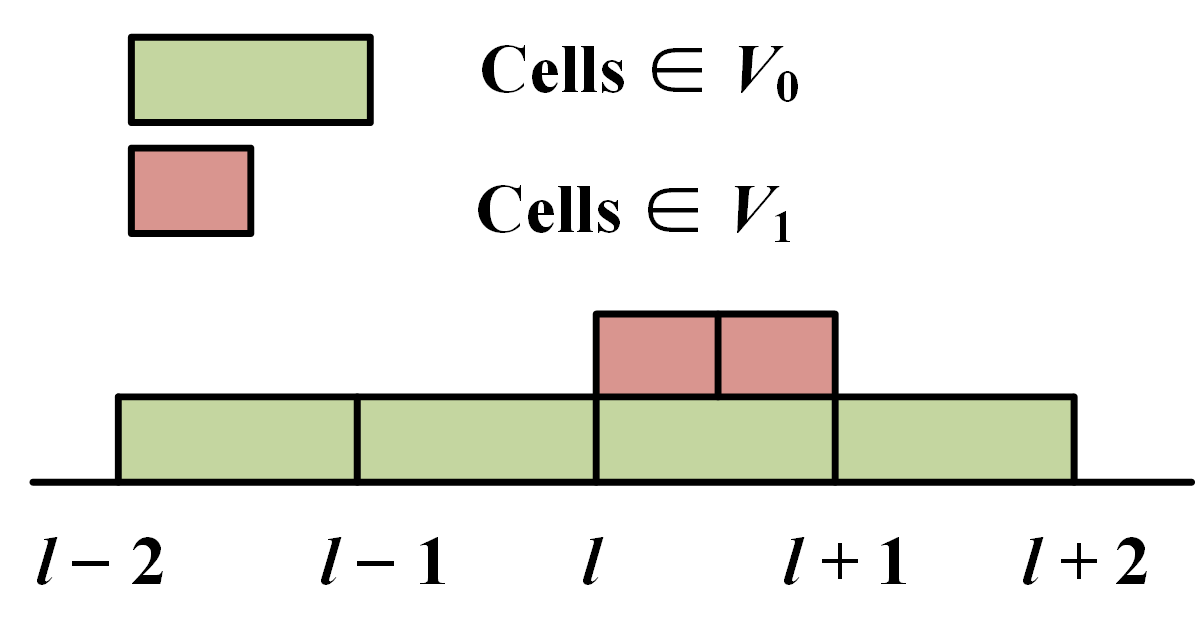}
        \caption{$N=4$, $b=1$}
        \label{fig:N4b1stencil}
    \end{subfigure}

    \caption{Stencils for symmetric and asymmetric wavelets.}
    \label{fig:waveletstencil}
\end{figure}

\item Calculate filter coefficients of the scaling function. The scaling function of the average-interpolating wavelet
satisfies the following integral average relation over the cell $[k,k+1]$ and refinement relation respectively:
\begin{equation}
\bar{\phi}_{k}=\int_{k}^{k+1}\phi(x)\, \mathrm{d}x=\delta_{0,k}.
\label{eq:inta}
\end{equation}

\begin{equation}
\phi(x)=\sum_{l = 1 - N + b}^{N + b}{h_l\phi(2x-l)}.
\label{eq:refinescaling}
\end{equation}
Substituting Eq.~\eqref{eq:refinescaling} into Eq.~\eqref{eq:inta}, the following equations can be derived:
\begin{equation}
h_0 + h_1 = 2.
\label{eq:h01}
\end{equation}

\begin{equation}
h_{2 l}+h_{2 l + 1}=0,\quad \frac{1 - N + b}{2}  \le  l \le \frac{N - 1+ b}{2} , \, l \ne 0.
\label{eq:h2l}
\end{equation}
Eqs.~\eqref{eq:h01} and~\eqref{eq:h2l} provide a system of $N$ linear equations. To uniquely determine the filter
coefficients, an additional set of $N$ linear equations is required. These additional constraints are derived from
average-interpolating polynomials constructed based on the scaling function $\phi(x)$. Let $p_l(x)$ be a polynomial of
degree $N-1$ defined over cells in $[x_L, x_R]$, which preserves the same cell averages as the scaling function
$\phi(x)$ over the intervals $[k, k+1]$. Then, we have
\begin{equation}
\int_{k}^{k+1} p_{l}(x)\, \mathrm{d}x
= \int_{k}^{k+1} \phi(x)\, \mathrm{d}x
= \bar{\phi}_k,
\quad
k \in \left[x_L,\;x_R\right],
\label{eq:pl1}
\end{equation}
where ~$x_L = l + \dfrac{1 - N - b}{2}$, ~$x_R = l + \dfrac{N + 1 - b}{2}$. Then, the polynomial $p_l(x)$ can be
reconstructed by enforcing the integral average condition $\bar{\phi}_k = \delta_{0,k}$. Accordingly, the scaling
coefficient $h_{2l}$ can be readily obtained by integrating $p_l(x)$ over the half cell $[l, l + \frac{1}{2}]$.
\begin{equation}
h_{2l} = 2 \int_{l}^{l + \frac{1}{2}} p_{l}(x)\, \mathrm{d}x.
\label{eq:h2l1}
\end{equation}
At this point, the determination of the filter coefficients has been reduced to evaluating the integral of the
reconstruction polynomials of degree $N-1$ over the half cell. The reconstruction methods developed in the FVM can be
naturally adopted here. A straightforward approach is to enforce the integral average condition as shown in
Eq.~\eqref{eq:pl1} to determine the coefficients of the polynomial $p_l(x)$ and subsequently compute $h_{2l}$ via
integration over the half cell. However, this approach requires solving a system of linear equations and is therefore
computationally implicit. Here, we propose an explicit formulation that directly evaluates $h_{2l}$ using the primitive
function of the polynomial $p_l(x)$. Let $P_l(x)$ be the primitive function of $p_l(x)$, we have
\begin{equation}
\int_{x_L}^{x} P_{l}^{\prime}(x)\, \mathrm{d}x=\int_{x_L}^{x} p_{l}(x)\, \mathrm{d}x=P_l(x)-P_l(x_L).
\label{eq:Prp}
\end{equation}
Assume that $P_{l}(x_L)=0$. Then, the values of \(P_l(x)\) at integer points are given by
\begin{equation}
P(k) = \sum_{i = x_L}^{k-1} \bar{\phi}_i =
\begin{cases}
0, & k < 1, \\
1, & k \ge 1.
\end{cases}
\label{eq:Pint}
\end{equation}
After constructing the primitive function $P_l(x)$ via a Lagrange interpolating polynomial determined by its values at
the integer points $k=x_L,x_L+1,\ldots,x_R$, the explicit formula for the filter coefficients $h_{2l}$ can be derived as
follows:
\begin{equation}
h_{2l} =
\left\{
\begin{aligned}
&2 \sum_{m=x_L}^{x_R} P(m)\,
\prod_{\substack{i=x_L \\ i\ne m}}^{x_R}
\frac{l + \frac{1}{2} - i}{m - i}, \qquad &l \le 0\\[6pt]
&2 \left[\left(
\sum_{m=x_L}^{x_R} P(m)\,
\prod_{\substack{i=x_L \\ i\ne m}}^{x_R}
\frac{l + \frac{1}{2} - i}{m - i}\right)
- 1
\right]. \qquad &l>0
\end{aligned}
\right.
\label{eq:h2lex}
\end{equation}
By solving Eqs.~\eqref{eq:h01},~\eqref{eq:h2l}, and ~\eqref{eq:h2lex}, we can uniquely determine the filter coefficients
of the scaling function.

\item Compute the values and $n$th order derivatives of the scaling function. First, these values at integer points are
obtained from the refinement relation:
\begin{equation}
\phi^{(n)}(x)=
2^n\sum_{k=NL}^{NR}h_k\phi^{(n)}(2x-k),\quad
N_L= 1 - N + b,\,
N_R= N + b,
\label{eq:fint}
\end{equation}
where $N_L$ and $N_R$ denote the left and right endpoints of the compact support of the scaling function. Since linear
equations derived by Eq.~\eqref{eq:fint} are singular, an additional constraint must be imposed by the polynomial
reproduction property as follows~\cite{Sweldens2005}:
\begin{equation}
\sum_{k} \frac{1}{n+1} \left( (k+1)^{n+1} - k^{n+1} \right) \phi^{(n)}(x - k)
= n!, \quad 0 \le n < N.
\label{eq:fints}
\end{equation}
By replacing the last linear equation in the linear system obtained from Eq.~\eqref{eq:fint} by Eq.~\eqref{eq:fints}
evaluated at $x=0$, the values at integer points of the scaling function can be uniquely calculated. These values at
arbitrary dyadic points can then be computed using Eq.~\eqref{eq:fint} via a cascade algorithm
\cite{Daubechies1992,Liu2019}. 
\end{enumerate}

Due to the average-interpolating property, scaling coefficients $c_{J,k}$ is the integral average of the function $f$
over the cell $[\frac{k}{2^J},\frac{k+1}{2^J}]$ denoted by $\bar{f}_{J,k}$. It is natural to choose the characteristic
function as the dual scaling function~\cite{Sweldens2005}:
\begin{equation}
\tilde{\phi}(x)=\chi_{[0,1]}(x)=
\begin{cases}
1, & x \in [0, 1], \\
0, & x \notin [0,1].
\end{cases}
\label{eq:dualscaling}
\end{equation}
In the dyadic framework, the scaling coefficients in $V_J$ space can be determined by the arithmetic average of the
scaling coefficients in $V_{J+1}$ space:
\begin{equation}
\bar{f}_{J,k}=\frac{1}{2}\left(\bar{f}_{J + 1,2k}+\bar{f}_{J + 1,2k + 1}\right).
\label{eq:scalingcoeff}
\end{equation}
To facilitate the computation of wavelet coefficients through two successive projections, a natural choice of the
wavelet function in the average-interpolating framework is the difference of two scaling functions~\cite{Sweldens2005}:
\begin{equation}
\psi(x)=\phi(2x)-\phi(2x-1).
\label{eq:waveletf}
\end{equation}

As shown in Eq.~\eqref{eq:pj1}, the wavelet coefficients $d_{J,m}$ can be obtained by the inner product between the
function and the corresponding dual wavelet function. However, the direct construction of the dual wavelet function is
generally nontrivial. Alternatively, the coefficients $d_{J,m}$ can be derived by substituting the refinement relation
\eqref{eq:refinew} into projection operator $\mathcal{P}_J$ and calculating the difference
$\mathcal{P}_{J+1}-\mathcal{P}_J$. The wavelet coefficients can then be given by:
\begin{equation}
d_{J,m}=\bar{f}_{J+1,2m}-\sum_{n}{h_{2n}{\bar{f}}_{J,m-n}}.
\label{eq:djm}
\end{equation}

Thus far, all the essential procedures for constructing general average-interpolating wavelets have been presented. By
following the above framework, both symmetric and asymmetric wavelets with desirable properties can be constructed. In
what follows, the average-interpolating wavelet approximation of functions will be discussed.
\subsubsection{Average-interpolating wavelet approximation of functions}
\label{subsec:waveletapproximation}
Based on the scaling and wavelet functions, as well as the corresponding expansion coefficients introduced in
subsection~\ref{subsec:Constructwavelet}, both single-resolution and multiresolution average-interpolating wavelet
approximations of a function in $L^2({\mathbb{R}})$ can be formulated. However, the physical problems are generally
defined on bounded domains. Functions limited in a finite domain $\Omega\subset\mathbb R$ can be approximated by the
wavelet approximation at a single resolution level as follows:
\begin{equation}
\mathcal{P}_J f(x)=\sum_{k \in \mathcal{K}_J}{\bar{f}_{J,k}\phi_{J,k}(x)},\quad 
\mathcal{K}_J=\left\{k\in\mathbb Z:\operatorname{supp}\phi_{J,k}\cap\Omega\ne\emptyset
\right\}.
\label{eq:sapprox}
\end{equation}
The corresponding approximation error estimate for sufficiently smooth functions is given in
Proposition~\ref{prop:LinfL2err}.
\begin{proposition}
\label{prop:LinfL2err}
Assume that $\mathcal P_J$ is generated by an average-interpolating wavelet multiresolution analysis of order $N$ on
$\mathbb R$. In particular, $\mathcal P_J$ exactly reproduces polynomials of degree at most $N-1$ from their cell
averages. Let $\Omega\subset\mathbb R$ be a bounded computational interval, and let $\widetilde\Omega$ be an extended
interval containing all supports $\operatorname{supp}\phi_{J,k}$ involved in the evaluation of $\mathcal P_J f$ on
$\Omega$.

Suppose that $f\in C^r(\widetilde\Omega)$ with integer $r\ge 1$, and that the exterior values required near the boundary
of $\Omega$ are supplied by the same function $f$ on $\widetilde\Omega$. Define $q=\min\{N,r\}$. Then
\begin{equation}
\|f-\mathcal P_J f\|_{L^\infty(\Omega)}
\le
C_\infty h_J^q |f|_{C^q(\widetilde\Omega)},
\label{eq:Linf-error}
\end{equation}
and consequently,
\begin{equation}
\|f-\mathcal P_J f\|_{L^2(\Omega)}
\le
C_2 h_J^q |f|_{C^q(\widetilde\Omega)},
\label{eq:L2-error}
\end{equation}
where $C_\infty$ is independent of $J$ and $h_J=\frac{1}{2^J}$,\, $C_2=|\Omega|^{1/2}C_\infty$.
\end{proposition}

\begin{proof}
Let $q=\min\{N,r\}$. For any $x\in\Omega$, define
\[
\mathcal K_J(x)
=
\left\{
k\in\mathcal K_J:
x\in\operatorname{supp}\phi_{J,k}
\right\}.
\]
Then only the scaling functions with indices in $\mathcal K_J(x)$ contribute to $\mathcal P_J f(x)$. Let $\widetilde
I_x\subset\widetilde\Omega$ be a local interval containing the supports of all $\phi_{J,k}$ with $k\in\mathcal K_J(x)$.
Since the scaling functions are compactly supported and have finite overlap, there exists a constant $C_s>0$, depending
only on the support size of the scaling functions, such that
\[
\operatorname{diam}(\widetilde I_x)\le C_s h_J,
\]
where $C_s$ is independent of $J$ and $h_J=\frac{1}{2^J}$.

Choose a point $x_0\in\widetilde I_x$ and let $p_{q-1}$ be the Taylor polynomial of $f$ of degree $q-1$ about $x_0$.
Since the wavelet has multiresolution order $N$ on $\mathbb R$, $\mathcal P_J$ exactly reproduces polynomials of degree
at most $N-1$ from their cell averages. Since $q\le N$, we have
\[
\mathcal P_J p_{q-1}=p_{q-1}.
\]
Therefore,
\[
f-\mathcal P_J f
=
(f-p_{q-1})-\mathcal P_J(f-p_{q-1}).
\]
Evaluating at $x$ and using the triangle inequality gives
\[
|f(x)-\mathcal P_J f(x)|
\le
|f(x)-p_{q-1}(x)|
+
|\mathcal P_J(f-p_{q-1})(x)|.
\]

The compact support and finite overlap of the scaling functions imply the local boundedness of $\mathcal P_J$ in the
maximum norm. Thus, there exists a constant $C_P>0$, independent of $J$ and $h_J$, such that
\[
|\mathcal P_J g(x)|
\le
C_P\|g\|_{L^\infty(\widetilde I_x)} .
\]
Hence,
\[
|f(x)-\mathcal P_J f(x)|
\le
(1+C_P)\|f-p_{q-1}\|_{L^\infty(\widetilde I_x)} .
\]

Since $\operatorname{diam}(\widetilde I_x)\le C_s h_J$, Taylor's theorem gives
\[
\|f-p_{q-1}\|_{L^\infty(\widetilde I_x)}
\le
C_T h_J^q |f|_{C^q(\widetilde I_x)},
\]
where $C_T$ depends only on $q$ and the relative support size of the scaling functions, but is independent of $J$ and
$h_J$. Therefore,
\[
|f(x)-\mathcal P_J f(x)|
\le
(1+C_P)C_T h_J^q |f|_{C^q(\widetilde I_x)}.
\]
Since $\widetilde I_x\subset\widetilde\Omega$, it follows that
\[
|f(x)-\mathcal P_J f(x)|
\le
C_\infty h_J^q |f|_{C^q(\widetilde\Omega)},
\qquad
C_\infty=(1+C_P)C_T.
\]
Taking the supremum over $x\in\Omega$ yields
\[
\|f-\mathcal P_J f\|_{L^\infty(\Omega)}
\le
C_\infty h_J^q |f|_{C^q(\widetilde\Omega)}.
\]

Since $\Omega$ is bounded, we further have
\[
\|f-\mathcal P_J f\|_{L^2(\Omega)}
\le
|\Omega|^{1/2}\|f-\mathcal P_J f\|_{L^\infty(\Omega)}.
\]
Consequently,
\[
\|f-\mathcal P_J f\|_{L^2(\Omega)}
\le
C_2 h_J^q |f|_{C^q(\widetilde\Omega)},
\qquad
C_2=|\Omega|^{1/2}C_\infty .
\]
\end{proof}

The average-interpolating wavelet multiresolution approximation of the function defined on $\Omega$ can be written as 
\begin{equation}
\begin{aligned}
&\mathcal{P}_{J_0}^{J_{\text{max}}}f(x)
=\sum_{k \in \mathcal{K}_{J_0}}{\bar{f}_{J_0,k}\phi_{J_0,k}}(x)+
\sum_{J=J_0}^{J_{\text{max}}-1}\sum_{m \in \Lambda_{J}}{d_{J,m}\psi_{J,m}}(x), \\
&\mathcal{K}_{J_0}=\left\{k\in\mathbb Z:\operatorname{supp}\phi_{J_0,k}\cap\Omega\ne\emptyset
\right\},\,
\Lambda_J=\left\{m\in\mathbb Z:\operatorname{supp}\psi_{J,m}\cap\Omega\ne\emptyset
\right\}.
\end{aligned}
\label{eq:mapprox}
\end{equation}
where $J_0$ is the coarsest resolution level, and $J_\text{max}$ is the finest resolution level. The wavelet
coefficients serve as local smoothness indicators, and their magnitude reflects the degree of local regularity, as
quantified by Proposition~\ref{prop:djkvalue}.
\begin{proposition}
Let $d_{J,k}$ be the wavelet coefficient defined by the average-interpolating multiresolution decomposition. Suppose
that $f\in C^r(\widetilde\Omega_{J,k})$ with integer $r\ge 1$, where $\widetilde\Omega_{J,k}$ is a local extended
interval containing the support region associated with the wavelet coefficient $d_{J,k}$. Define $q=\min\{N,r\}$. Then
\begin{equation}
|d_{J,k}|
\le
C_d h_J^q |f|_{C^q(\widetilde\Omega_{J,k})},
\label{eq:detail-decay}
\end{equation}
where $C_d$ is independent of $J$ and $h_J=\frac{1}{2^J}$.
\label{prop:djkvalue}
\end{proposition}

\begin{proof}
Let $q=\min\{N,r\}$. For two consecutive approximation spaces $V_J$ and $V_{J+1}$, the detail component between these
two levels is defined by
\[
\mathcal Q_J f
=
\mathcal P_{J+1}f-\mathcal P_J f .
\]
Thus,
\[
\mathcal Q_J f
=
(\mathcal P_{J+1}f-f)+(f-\mathcal P_J f).
\]

Let $\widetilde\Omega_{J,k}$ be a local extended interval associated with the wavelet coefficient $d_{J,k}$. Applying
Proposition~\ref{prop:LinfL2err} to the two levels $J$ and $J+1$ on $\widetilde\Omega_{J,k}$ gives
\[
\|f-\mathcal P_J f\|_{L^\infty(\widetilde\Omega_{J,k})}
\le
C_\infty^{(0)} h_J^q |f|_{C^q(\widetilde\Omega_{J,k})},
\]
and
\[
\|f-\mathcal P_{J+1}f\|_{L^\infty(\widetilde\Omega_{J,k})}
\le
C_\infty^{(1)} h_{J+1}^q |f|_{C^q(\widetilde\Omega_{J,k})},
\]
where $C_\infty^{(0)}$ and $C_\infty^{(1)}$ are independent of $J$ and $h_J=\frac{1}{2^J}$. Since $h_{J+1}=h_J/2$, it
follows that $h_{J+1}^q=2^{-q}h_J^q$. Therefore,
\[
\|\mathcal Q_J f\|_{L^\infty(\widetilde\Omega_{J,k})}
\le
C_Q h_J^q |f|_{C^q(\widetilde\Omega_{J,k})},\quad C_Q=C_\infty^{(0)}+2^{-q}C_\infty^{(1)}.
\]

Since the average-interpolating wavelet coefficient is obtained from a finite linear combination of local cell averages
of the detail component, there exist fixed coefficients $\{a_m\}_{m\in S}$, with a finite index set $S$, such that
\[
d_{J,k}
=
\sum_{m\in S} a_m \overline{\mathcal Q_J f}_{m},
\]
where $\overline{\mathcal Q_J f}_{m}$ denotes the cell average of $\mathcal Q_J f$ over a local cell contained in
$\widetilde\Omega_{J,k}$. Hence,
\[
|d_{J,k}|
\le
\sum_{m\in S}|a_m|\,|\overline{\mathcal Q_J f}_{m}|.
\]
Since
\[
|\overline{\mathcal Q_J f}_{m}|
\le
\|\mathcal Q_J f\|_{L^\infty(\widetilde\Omega_{J,k})},
\]
we obtain
\[
|d_{J,k}|
\le
C_\psi
\|\mathcal Q_J f\|_{L^\infty(\widetilde\Omega_{J,k})},
\qquad
C_\psi=\sum_{m\in S}|a_m|.
\]
Here $C_\psi$ is independent of $J$ and $h_J$ because the number of involved cell averages and the coefficients are
fixed for a given average-interpolating wavelet. Combining the above estimates yields
\[
|d_{J,k}|
\le
C_d h_J^q |f|_{C^q(\widetilde\Omega_{J,k})}, \quad C_d=C_\psi C_Q.
\]
\end{proof}

Finally, the fast wavelet transform (FWT) and its inverse are presented. The FWT is a linear algorithm to calculate the
wavelet coefficients $d_{J,m}$ at resolution levels $J \ge J_0$ and the cell averages $\bar{f}_{J_0,k}$ at the coarsest
level $J_0$ from the given cell averages $\bar{f}_{J,k}$. The inverse FWT reconstructs the cell averages from its
multiresolution representation. The details of the FWT and inverse FWT are described as follows:
\begin{equation}
\label{eq:FWT}
\begin{aligned}
&\textbf{Fast Wavelet Transform}: \\
&\textbf{For}\;J = J_{\max}-1, J_{\max}-2, \dots, J_0 \\
&\quad \text{For each index $k$ at the level}\,J \\
&\qquad \text{Compute scaling coefficients:}
\quad \bar{f}_{J,k}
= \frac{1}{2}
\left(\bar{f}_{J+1,2k}+\bar{f}_{J+1,2k+1}\right). \\
&\qquad \text{Compute wavelet coefficients:}
\quad d_{J,k}
= \bar{f}_{J+1,2k}
-\sum_{n} h_{2n}\bar{f}_{J,k-n}.
\end{aligned}
\end{equation}

\begin{equation}
\label{eq:IFWT}
\begin{aligned}
&\textbf{Inverse Fast Wavelet Transform}: \\
&\textbf{For}\;J = J_0, J_0+1, \dots, J_{\max}-1 \\
&\quad \text{For each index $k$ at the level}\,J \\
&\qquad \text{Reconstruct even cell:}
\quad \bar{f}_{J+1,2k}
= d_{J,k}
+\sum_{n} h_{2n}\bar{f}_{J,k-n}. \\
&\qquad \text{Reconstruct odd cell:}
\quad \bar{f}_{J+1,2k+1}
= 2\bar{f}_{J,k}
-\bar{f}_{J+1,2k}.
\end{aligned}
\end{equation}

\noindent\textbf{Remark 1.} In Proposition~\ref{prop:LinfL2err}, we assume that the exterior values required near the
boundary of $\Omega$ in the average-interpolating wavelet approximation are provided by the same function $f$. For
practical computations, these values should be supplied through appropriate boundary conditions. \\
\noindent\textbf{Remark 2.} It should be emphasized that the wavelet coefficients $d_{J,m}$ computed by
Eq.~\eqref{eq:djm} enable the reconstruction of the function $f$ at arbitrary dyadic points through the wavelet
multiresolution approximation in Eq.~\eqref{eq:mapprox}. In contrast, the average-interpolating wavelets in Harten's
framework mainly focus on the discretized multiresolution representation of $f$ over computational cells
\cite{Harten1995CPAM}. This distinction indicates that the present formulation provides a function-space multiresolution
representation rather than a purely data-based decomposition. \\
\noindent\textbf{Remark 3.} Proposition~\ref{prop:djkvalue} shows that the wavelet coefficients provide a natural
indicator of the local regularity of the function, thereby forming the foundation of adaptive multiresolution
wavelet-based numerical methods. \\
\noindent\textbf{Remark 4.} Both the wavelet multiresolution decomposition and reconstruction are linear operators,
since the wavelet coefficients and reconstructed cell averages are obtained through finite linear combinations of
neighboring cell averages. The FWT and inverse FWT provide the efficient implementation of the wavelet transform.

\subsection{Conservative wavelet upwind schemes on uniform cells}
\subsubsection{Finite volume method}
In the finite volume method, the computational domain $\Omega$ is divided into $N_1$ non-overlapping cells,
$\mathcal{I}_l: x \in [x_{l-1/2}, x_{l+1/2}], \, l = 1, 2 ,\dots, N_1$, with a uniform mesh size $\Delta x = x_{l+1/2}-
x_{l-1/2}$. The cell-average of the variable $u$ in the cell  $\mathcal{I}_l$ is defined as follows:
\begin{equation}
\bar{u}_l(t)=\frac{1}{\Delta x }\int_{x_{l-1/2}}^{x_{l+1/2}} u(x,t)\, \mathrm {d}x.
\label{eq:uaverage}
\end{equation}
By integrating Eq.~\eqref{eq:model} over the cell $\mathcal{I}_l$ and dividing by the mesh size $\Delta x$, the
following semi-discrete ordinary differential equation (ODE) is obtained:
\begin{equation}
\frac{\mathrm {d} \bar{u}_l(t)}{\mathrm {d} t}=-\frac{1}{\Delta x }\left(\hat f_{l+1/2}-\hat f_{l-1/2}\right),
\label{eq:ODEs}
\end{equation}
where $\hat f_{l+1/2}$ is the numerical flux at the cell interface $x_{l+1/2}$, which is generally determined by a
Riemann solver:
\begin{equation}
\hat f_{l+1/2}=f_{l+1/2}^{Riemann}\left(u_{l+1/2}^L, u_{l+1/2}^R\right).
\label{eq:RiemannSolver}
\end{equation}
where $u_{l+1/2}^L$ and $u_{l+1/2}^R$ are the left and right interface values at $x_{l+1/2}$, reconstructed from the
cell averages $\bar{u}_k$ using the corresponding upwind-biased reconstructions. 
\subsubsection{Average-interpolating wavelet upwind scheme}
The construction of the conservative wavelet upwind scheme relies on cell-average-based wavelet basis functions with
upwind properties. Using the general construction method for average-interpolating wavelets presented in
subsection~\ref{subsec:Constructwavelet}, a pair of asymmetric average-interpolating wavelets can be constructed that
are mirror-symmetric with respect to $x=1/2$ and enable upwind-biased wavelet reconstructions of the interface values
$u_{l+1/2}^L$ and $u_{l+1/2}^R$.

The upwind property is controlled by the stencil bias parameter $b$. An average-interpolating wavelet with $b > 0$ is
biased in the positive direction, whereas one with $b < 0$ is biased in the negative direction. Accordingly, the
wavelets with $b > 0$ and $b < 0$ are referred to as the positive- and negative-direction upwind wavelets, respectively.
It should be noted that not all asymmetric average-interpolating wavelets are suitable for constructing stable upwind
schemes. The stability of a wavelet upwind scheme can be assessed by examining the non-negativity of the dissipation
coefficients over the entire wave-number range through the Fourier analysis~\cite{Yang2023,Vichnevetsky1982}. Owing to
the mirror symmetry between the positive- and negative-direction upwind wavelets, it is sufficient to analyze the
positive-direction case with $b>0$. In the present work, the average-interpolating wavelets with $N=4$ and $b=\pm 1$ are
adopted to construct the proposed scheme. The filter coefficients of the corresponding scaling functions are listed in
Table.~\ref{tab:filter_coefficients}, and the associated scaling and wavelet functions are illustrated in
Fig.~\ref{fig:N4basis}. According to Proposition~\ref{prop:LinfL2err}, the selected wavelet bases lead to a fourth-order
average-interpolating wavelet upwind scheme, referred to as AIWU4. The spectral analysis in
subsection~\ref{subsec:Fourieranalysis} demonstrates that the proposed AIWU4 scheme has non-negative dissipation
coefficients over the entire wave-number range, indicating the stability of the selected wavelet discretization
operator. 
\begin{table}
    \centering
    \caption{Filter coefficients of the scaling functions of the positive- and negative-direction upwind wavelets.}
    \label{tab:filter_coefficients}
    \begin{tabular*}{0.90\textwidth}{@{\extracolsep{\fill}}cccc}
        \toprule
        \multicolumn{2}{c}{Positive direction, $N=4$, $b=1$} & \multicolumn{2}{c}{Negative direction, $N=4$, $b=-1$} \\
        \cmidrule(lr){1-2} \cmidrule(lr){3-4} Filter coefficient & Value & Filter coefficient & Value \\
        \midrule
        $h_{-2}$ & $-0.078125$ & $h_{-4}$ & $0.046875$ \\
        $h_{-1}$ & $0.078125$  & $h_{-3}$ & $-0.046875$ \\
        $h_{0}$  & $0.859375$  & $h_{-2}$ & $-0.265625$ \\
        $h_{1}$  & $1.140625$  & $h_{-1}$ & $0.265625$ \\
        $h_{2}$  & $0.265625$  & $h_{0}$  & $1.140625$ \\
        $h_{3}$  & $-0.265625$ & $h_{1}$  & $0.859375$ \\
        $h_{4}$  & $-0.046875$ & $h_{2}$  & $0.078125$ \\
        $h_{5}$  & $0.046875$  & $h_{3}$  & $-0.078125$ \\
        \bottomrule
    \end{tabular*}
\end{table}

\begin{figure}
    \centering

    \begin{subfigure}[b]{0.45\textwidth}
        \centering
        \includegraphics[width=\linewidth]{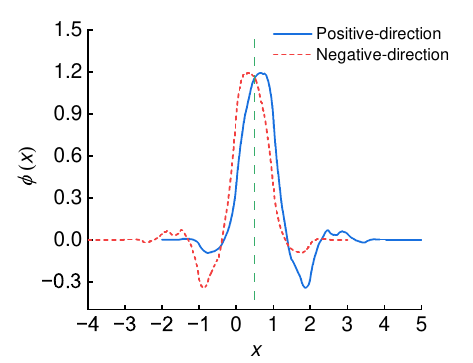}
        \caption{Scaling functions}
        \label{fig:N4scaling}
    \end{subfigure}
    \hfill
    \begin{subfigure}[b]{0.45\textwidth}
        \centering
        \includegraphics[width=\linewidth]{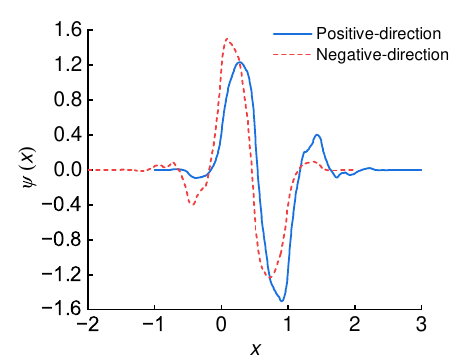}
        \caption{Wavelet functions}
        \label{fig:N4bwavelet}
    \end{subfigure}

    \caption{Scaling and wavelet functions of the $N=4$, $b=\pm 1$ average-interpolating wavelets.}
    \label{fig:N4basis}
\end{figure}
\FloatBarrier

When using wavelet approximation at a single resolution level, the corresponding cells are distributed uniformly. Since
the average-interpolating wavelet is constructed in the dyadic framework, the mesh size only depends on the resolution
level. Hence, we utilize asymmetric average-interpolating wavelet approximation at the resolution level $J$ to construct
the wavelet upwind scheme on uniform cells. First, we restate cell distributions used in the dyadic framework. The
non-overlapping cells are generated as $\mathcal{I}_l=[\dfrac{l}{2^J}, \dfrac{l+1}{2^J}], \, l = l_L, l_L+1 ,\dots,
l_R$, with a uniform mesh size $\Delta x = \frac{1}{2^J}$. $\bar u_{J,l}$ denotes the cell average of the variable $u$
over the cell $\mathcal{I}_l$. Then, the interface values $u_{l+1/2}^L$ and $u_{l+1/2}^R$ can be reconstructed by the
corresponding upwind average-interpolating wavelet approximations as follows:
\begin{equation}
u_{l+1/2}^L=\sum_{k=l + 2 - (N + b)}^{l -(1 - N + b)}\bar{u}_{J,k} \phi_{J,k}^{+}(x_{l+1/2}), \quad b = 1, 
\label{eq:leftwavelet}
\end{equation}
\begin{equation}
u_{l+1/2}^R=\sum_{k=l + 2 - (N + b)}^{l -(1 - N + b)}\bar{u}_{J,k} \phi_{J,k}^{-}(x_{l+1/2}), \quad b = -1,
\label{eq:rightwavelet}
\end{equation}
where $\phi_{J,k}^{+}(x)$ and $\phi_{J,k}^{-}(x)$ are scaling functions of the positive- and negative-direction upwind
wavelets, respectively. Once obtaining these interface values by the wavelet approximation, numerical fluxes can be
computed by the selected Riemann solver. Then, the semi-discrete ODEs as shown in Eq.~\eqref{eq:ODEs} can be resolved by
the  third-order strong-stability preserving Runge-Kutta (SSP-RK3) method \cite{Gottlieb2001}.

Although the upwind average-interpolating wavelet approximation has been introduced above in the finite-volume-type
formulation, the same approximation can also be incorporated into the Shu--Osher's conservative finite-difference form
at a single resolution level $J$~\cite{Shu1989,Merriman2003}. This provides a conservative finite-difference-type
formulation at uniform cells. However, when this formulation is directly extended to an adaptive multiresolution
setting, the conservativity of the resulting adaptive scheme is generally not preserved. This issue will be discussed in
subsection~\ref{subsec:analysisconservation}. Therefore, the proposed finite-volume-type conservative wavelet upwind
formulation based on cell averages is essential for preserving conservation property in the adaptive framework. 
\subsubsection{Boundary Variation Diminishing (BVD) reconstrutciton for discontinuties}
\label{subsec:BVDreconstruction}
To develop high-resolution shock-capturing schemes, Xiao and the co-authors~\cite{Sun2016, Deng2019} proposed a novel
BVD reconstruction method in the FVM framework. The philosophy of the BVD reconstruction is that the minimization of
jumps of the reconstruction values at cell interfaces will effectively reduce the numerical dissipation in the Riemann
solvers. By utilizing the BVD principle with a tangent of hyperbola for interface capturing (THINC) function as one of
the reconstruction candidates, various polynomial-based high-resolution shock-capturing schemes have been proposed both
in FVM and FDM frameworks~\cite{Sun2016,Deng2019,Deng2020, Deng2022, Wakimura2022, Pan2024}. We also have proposed a
fourth-order conservative wavelet upwind scheme with the BVD reconstruction in the FDM framework~\cite{Yang2025AMS}. Due
to the inherent low-pass filtering property of the wavelet bases, the wavelet-based numerical scheme naturally filters
out high-frequency numerical oscillations and improves the ability to capture discontinuities.

In this subsection, we will present the fourth-order average-interpolating wavelet upwind scheme with the BVD
reconstruction for capturing discontinuities. By leveraging the monotonicity and differentiability of the hyperbolic
tangent function, Xiao and the co-authors \cite{Sun2016, Deng2019} proposed the THINC reconstruction to depict
discontinuities in compressible flows without spurious oscillations. The THINC reconstruction $\tilde u(x)^T$ for the
variable $u$ is given by~\cite{Sun2016, Deng2019}
\begin{equation}
\tilde u(x)^T=\bar u_{\min} + \dfrac{\bar u_{\max}}{2}\left[1 + \theta \, \tanh\left(\beta 
\left(\frac{x-x_{l-1/2}}{x_{l+1/2}-x_{l-1/2}}-\tilde x_l\right)\right)\right],
\label{eq:THINCF}
\end{equation}
where $\bar u_{\min}= {\min}\left\{\bar u_{l-1}, \bar u_{l+1}\right\}$, \, $\bar u_{\max}= \max\left\{\bar u_{l-1}, \,
\bar u_{l+1}\right\}-\bar u_{\min}$, $\theta = \textrm{sgn}\left(\bar u_{l+1}-\bar u_{l-1}\right)$. The parameter
$\beta$ controls the jump thickness. A small value of $\beta$ produces a smooth profile, while a large $\beta$ leads to
a sharp transition. The unknown $\tilde x_l$ which denotes the location of the jump center, is determined by the
following cell average constraint condition:
\begin{equation}
\bar u_l=\frac{1}{\Delta x} \int_{x_{l-1/2}}^{x_{l+1/2}} \tilde u(\xi)^T \, \mathrm{d}\xi. 
\label{eq:jumpcenter}
\end{equation}
Since the THINC reconstruction has no solution for local extreme, it degrades to a piecewise constant in such cases. The
explicit formulations of the THINC reconstruction for interface values $u_{l+1/2}^{L, T}$ and $u_{l-1/2}^{R, T}$ are
given by~\cite{Deng2019} 
\begin{equation}
u_{l+1/2}^{L,T}=
\begin{cases}
\bar u_{\min}+\dfrac{\bar u_{\max}}{2}\left(1+\theta 
\dfrac{\tanh(\beta) + A }{1 + A\tanh(\beta)}\right) \quad 
&\mathrm{if} \quad (\bar u_{l+1}- \bar u_l)(\bar u_l - \bar u_{l-1}) > 0,\\
\bar u_l
&\mathrm{otherwise}.
\end{cases}
\label{eq:LeftTHINC}
\end{equation}
\begin{equation}
u_{l-1/2}^{R,T}=
\begin{cases}
\bar u_{\min}+\dfrac{\bar u_{\max}}{2}\left(1+ \theta A\right) \quad 
&\mathrm{if} \quad (\bar u_{l+1}- \bar u_l)(\bar u_l - \bar u_{l-1}) > 0,\\
\bar u_l
&\mathrm{otherwise}.
\end{cases}
\label{eq:RigthTHINC}
\end{equation}
where $A=\dfrac{B/\cosh(\beta)-1}{\tanh(\beta)}$, $B=\exp[\theta\beta(2C - 1)]$, and $C = \dfrac{\bar u_l-\bar
u_{\min}+\epsilon}{\bar u_{\max}+\epsilon}$. The parameter $\epsilon = 1.0 \times 10^{-20}$ is used to avoid arithmetic
failure. 

The choice of $\beta$ is critical for achieving satisfactory performance of the THINC scheme. He et al.~\cite{He2022BVD}
proved that the THINC scheme with $\beta=\ln 3$ achieves second-order accuracy and satisfies the TVD property. In
addition, the spectral properties of THINC schemes with different $\beta$ showed that a larger value of $\beta$ is more
suitable for capturing sharp discontinuities due to its anti-diffusion effect, with $\beta=1.6$--$2.0$ being
recommended~\cite{Deng2020}. Therefore, in the present work, the THINC scheme with $\beta_s=\ln 3$ is used to
reconstruct relatively smooth interface values, denoted by $u_{l+1/2}^{L,Ts}$ and $u_{l-1/2}^{R,Ts}$, whereas the THINC
scheme with $\beta_h=1.6$ is adopted to enhance the resolution of sharp discontinuities, with the corresponding
interface values denoted by $u_{l+1/2}^{L,Th}$ and $u_{l-1/2}^{R,Th}$.

In the present study, the AIWU4 scheme is used as the high-order candidate, whereas the aforementioned THINC schemes
serve as the low-order candidates. The two-stage BVD reconstruction proposed by Deng et al.~\cite{Deng2019} is employed
to select the desirable reconstructed interface values. The resulting method is referred to as the average-interpolating
wavelet of fourth-order and THINC function of two-stage reconstruction based on the BVD algorithm
($\mathrm{AIW_4T_2\mbox{-}BVD}$). Before presenting the BVD reconstruction, we recall the definition of the total
boundary variation (TBV)~\cite{Deng2019} of a given cell $\mathcal{I}_l$ as
\begin{equation}
\text{TBV}_l=|u_{l+1/2}^L-u_{l+1/2}^R|+|u_{l-1/2}^L-u_{l-1/2}^R|.
\label{eq:TBV}
\end{equation}
Then, the interface values of the variable $u$ can be determined by the following two-stage reconstruction process:\\
\textbf{Stage 1} The first BVD reconstruction using the AIWU4 scheme and the  $\beta_s=\ln3$ THINC scheme.
\begin{enumerate}[(1)]
\item Initialize $\text{flag}(l)=0$, and apply the AIWU4 scheme to compute the corresponding interface values
$u_{l+1/2}^{L,W}$ and $u_{l+1/2}^{R,W}$ at the cell $\mathcal{I}_l$. Calculate the $\mathrm{TBV}_l^W $
\begin{equation}
\text{TBV}_l ^W=|u_{l+1/2}^{L, W}-u_{l+1/2}^{R,W}|+|u_{l-1/2}^{L,W}-u_{l-1/2}^{R,W}|.
\label{eq:TBVwavelet}
\end{equation}

\item Compute the interface values $u_{l+1/2}^{L,Ts}$ and $u_{l+1/2}^{R,Ts}$ by the THINC scheme with $\beta_s=\ln3$,
and obtain the associated $\mathrm{TBV}_l^{Ts}$
\begin{equation}
\text{TBV}_l ^{Ts}=|u_{l+1/2}^{L, Ts}-u_{l+1/2}^{R,Ts}|+|u_{l-1/2}^{L,Ts}-u_{l-1/2}^{R,Ts}|.
\label{eq:TBVTs}
\end{equation}

\item Compare the $\mathrm{TBV}_l^W$ and $\mathrm{TBV}_l^{Ts}$ at each cell, and mark the cell
\begin{equation}
\text{flag}(k)=1, \quad k=l-1,l,l+1, \, \text{if} \; \mathrm{TBV}_l^{Ts} < \mathrm{TBV}_l^W.
\label{eq:MarkcellTs}
\end{equation}

\item Modify the interface values by
\begin{equation}
u_{l+1/2}^L =
\begin{cases}
u_{l+1/2}^{L,Ts}, & \text{if } \mathrm{flag}(l)=1,\\
u_{l+1/2}^{L,W},  & \text{otherwise}.
\end{cases}
\label{eq:Select1L}
\end{equation}
\begin{equation}
u_{l-1/2}^R =
\begin{cases}
u_{l-1/2}^{R,Ts}, & \text{if } \mathrm{flag}(l)=1,\\
u_{l-1/2}^{R,W},  & \text{otherwise}.
\end{cases}
\label{eq:Select1R}
\end{equation}
\end{enumerate}  
\textbf{Stage 2} The second BVD construction with the obtained interface values in Stage 1 and the $\beta_h=1.6$ THINC
scheme.
\begin{enumerate}[(1)]
\item Calculate the $\mathrm{TBV}_l$ using the interface values obtained in Stage 1
\begin{equation}
\mathrm{TBV}_l =|u_{l+1/2}^{L}-u_{l+1/2}^{R}|+|u_{l-1/2}^{L}-u_{l-1/2}^{R}|.
\label{eq:TBVSelect1}
\end{equation}

\item Compute the interface values $u_{l+1/2}^{L,Th}$ and $u_{l+1/2}^{R,Th}$ by the THINC scheme with  $\beta_h=1.6$ and
obtain the associated $\mathrm{TBV}_l^{Th}$
\begin{equation}
\mathrm{TBV}_l ^{Th}=|u_{l+1/2}^{L, Th}-u_{l+1/2}^{R,Th}|+|u_{l-1/2}^{L,Th}-u_{l-1/2}^{R,Th}|.
\label{eq:TBVTh}
\end{equation}

\item Mark the cell by comparing $\mathrm{TBV}_l$ and $\mathrm{TBV}_l^{Th}$
\begin{equation}
\mathrm{flag}(l)=2, \quad \text{if} \; \text{TBV}_l^{Th} < \text{TBV}_l.
\label{eq:MarkcellTh}
\end{equation}

\item Calculate the final interface values by
\begin{equation}
u_{l+1/2}^L =
\begin{cases}
u_{l+1/2}^{L,Th}, & \text{if } \mathrm{flag}(l)=2,\\
u_{l+1/2}^{L},  & \text{otherwise}.
\end{cases}
\label{eq:Select2L}
\end{equation}
\begin{equation}
u_{l-1/2}^R =
\begin{cases}
u_{l-1/2}^{R,Th}, & \text{if } \mathrm{flag}(l)=2,\\
u_{l-1/2}^{R},  & \text{otherwise}.
\end{cases}
\label{eq:Select2R}
\end{equation}
\end{enumerate}  

\noindent\textbf{Remark 1.} The adjacent cells $l-1$ and $l+1$ are marked for reconstruction using the THINC scheme with
$\beta_s=\ln3$ if $\mathrm{TBV}_l^{Ts} < \mathrm{TBV}_l^ W$, thereby capturing the discontinuities without spurious
oscillations. \\
\noindent\textbf{Remark 2.} For the Euler systems, both the computation of the interface values and the BVD
reconstruction are performed in the local characteristic fields. The Roe average is employed for the local
characteristic decomposition~\cite{Roe1997}.

\subsubsection{Spectral properties of average-interpolating wavelet upwind schemes}
\label{subsec:Fourieranalysis}
For the linear AIWU4 scheme, Fourier analysis~\cite{Vichnevetsky1982} is performed to assess its dissipation and
dispersion properties. Since the standard Fourier analysis is not directly applicable to nonlinear schemes, the
approximate dispersion relation (ADR) proposed by Pirozzoli~\cite{Pirozzoli2006} is employed for the nonlinear
$\mathrm{AIW_4T_2\mbox{-}BVD}$ scheme to examine the spectral properties. The dissipation and dispersion properties of
the AIWU4 and $\mathrm{AIW_4T_2\mbox{-}BVD}$ schemes are compared with those of the linear fifth-order upwind scheme and
WENO5-Z scheme, as shown in  Fig.~\ref{fig:krki}. The dissipation coefficients of the AIWU4 scheme remain non-negative
over the entire wave-number range, confirming its numerical stability as the basic high-order candidate. The
$\mathrm{AIW_4T_2\mbox{-}BVD}$ scheme can reproduce the dissipation and dispersion properties of the AIWU4 scheme for
wave numbers up to an intermediate wave-number ($\alpha \approx 1.86$) .

The numerical dissipation of the wavelet schemes is comparable to that of the WENO5-Z scheme in the low and intermediate
wave-number regions. In contrast, in the high wave-number region, the AIWU4 and $\mathrm{AIW_4T_2\mbox{-}BVD}$ schemes
show substantially stronger dissipation, which is beneficial for suppressing high-frequency oscillations. The linear
AIWU4 scheme exhibits better spectral resolution than the linear $5$th order polynomial-based upwind scheme. For the
nonlinear schemes, the $\mathrm{AIW_4T_2\mbox{-}BVD}$ scheme provides significantly better spectral resolution than the
WENO5-Z scheme in intermediate and high wave-number regions. 

Interestingly, by directly comparing the discretization coefficients, we find that the upwind discretization operator
for $\partial u/\partial x$ constructed from the $N=4$, $b=1$ average-interpolating wavelets is identical to that of the
conservative wavelet upwind scheme based on the $N=5$, $b=1$ interpolating wavelet. This coefficient-level equivalence
indicates that the proposed wavelet upwind scheme inherits the numerical merits of the conservative wavelet upwind
scheme based on interpolating wavelets. Moreover, it suggests that the relation between the two discretizations may not
be accidental, but may instead reflect a more general correspondence between the upwind discretization operators for
$\partial u/\partial x$ constructed from average-interpolating wavelets of order $N$ and interpolating wavelets of order
$N+1$. A rigorous proof of this correspondence for arbitrary order is beyond the scope of the present work and remains a
topic for future investigation.

\begin{figure}
    \centering

    \begin{subfigure}[b]{0.45\textwidth}
        \centering
        \includegraphics[width=\linewidth]{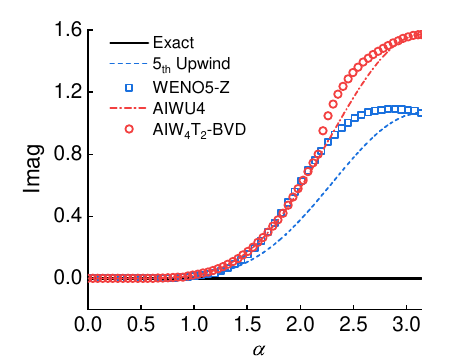}
        \caption{Dissipation.}
        \label{fig:kr}
    \end{subfigure}
    \hfill
    \begin{subfigure}[b]{0.45\textwidth}
        \centering
        \includegraphics[width=\linewidth]{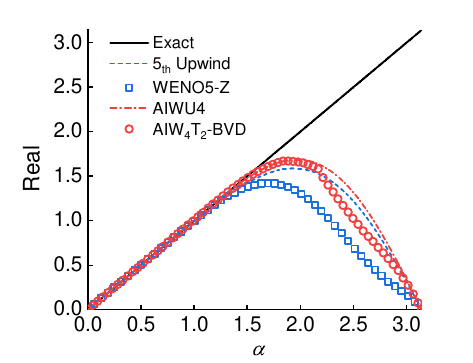}
        \caption{Dispersion.}
        \label{fig:ki}
    \end{subfigure}

    \caption{Comparison of the spectral properties of different schemes.}
    \label{fig:krki}
\end{figure}
\FloatBarrier

\subsection{Conservative adaptive wavelet upwind schemes on multiresolution non-uniform cells}
\subsubsection{Multiresolution cell adaptation}
First, we introduce the cell structure of the average-interpolating wavelet multiresolution analysis, which corresponds
to a graded tree structure in the dyadic framework~\cite{Cohen2003,Roussel2004}. It should be noted that the
connectivity in the tree structure must be always guaranteed. Following Cohen and the co-authors work~\cite{Cohen2003},
definitions of the tree structure are described as follow:
\begin{enumerate}[(1)]
\itemsep=0pt
\item A cell $\mathcal{I}_{J,k}$ is defined as $\mathcal{I}_{J,k}=[\dfrac{k}{2^J}, \dfrac{k+1}{2^J}]$.

\item The root cells \(\mathcal{I}_{J_0,k}\) form the basis of the tree and correspond to the cells at the coarsest
level.

\item For a cell \(\mathcal{I}_{J,k}\), its two child cells at the next finer level are denoted by
\(\mathcal{I}_{J+1,2k}\) and \(\mathcal{I}_{J+1,2k+1}\), respectively.

\item The child cells sharing the same parent cell are called sibling cells.

\item An active cell with no active child cells is referred to as a leaf cell denoted by $\mathcal
I_{J,k}^{\mathrm{leaf}}$, which constitutes the final adaptive cell distribution for the physical variables in the
finite-volume discretization.

\item An active cell with two active child cells is referred to as a wavelet-basis cell denoted by $\mathcal
I_{J,k}^{w}$, which is associated to the wavelet basis function $\psi_{J,k}(x)$ required for the wavelet multiresolution
approximation.

\end{enumerate}

Figure~\ref{fig:gradedtree} shows a schematic illustration of the graded-tree structure for the one-dimensional
average-interpolating wavelet. The cells at different resolution levels, from $J=1$ to $J=5$, are displayed from coarse
to fine. The standard cells are indicated by thin black lines, while the leaf cells and wavelet-basis cells are marked
in orange and blue, respectively. In the construction of conservative adaptive multiresolution wavelet upwind schemes,
the tree structure should satisfy the following conditions:
\begin{enumerate}[(1)]
\itemsep=0pt
\item The root cells are always activated, since they correspond to the scaling functions at the coarsest level and
ensure the baseline accuracy of the wavelet multiresolution approximation.

\item The two child cells of an active parent cell are activated simultaneously, so that hanging cells are avoided and
the multiresolution decomposition between two successive levels remains well defined.

\item  The stencil cells required for computing each wavelet coefficient must be activated, so that the multiresolution
decomposition and reconstruction can be performed consistently by the FWT and inverse FWT algorithms.
\end{enumerate}
\begin{figure}
    \centering
    \includegraphics[width=0.90\textwidth]{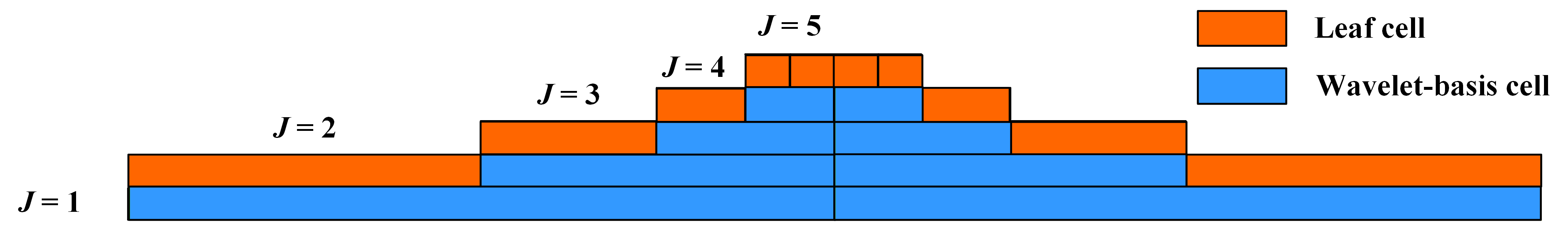}
    \caption{Schematic illustration of the graded-tree structure used in the adaptive average-interpolating wavelet method.}
    \label{fig:gradedtree}
\end{figure}
\FloatBarrier

As presented in subsection.~\ref{subsec:waveletapproximation}, the physical variable $u$ defined on a domain $\Omega$
can be decomposed into the following average-interpolating wavelet multiresolution approximation:
\begin{equation}
\begin{aligned}
&\mathcal{P}_{J_0}^{J_{\text{max}}}u(x)
=\sum_{k \in \mathcal{K}_{J_0}}{\bar{u}_{J_0,k}\phi_{J_0,k}}(x)+
\sum_{J=J_0}^{J_{\text{max}}-1}\sum_{m \in \Lambda_{J}}{d_{J,m}\psi_{J,m}}(x), \\
&\mathcal{K}_{J_0}=\left\{k\in\mathbb Z:\operatorname{supp}\phi_{J_0,k}\cap\Omega\ne\emptyset
\right\},\,
\Lambda_J=\left\{m\in\mathbb Z:\operatorname{supp}\psi_{J,m}\cap\Omega\ne\emptyset
\right\}.
\end{aligned}
\label{eq:umapprox}
\end{equation}
The magnitudes and local decay rates of the wavelet coefficients given in Proposition~\ref{prop:djkvalue} can be used to
evaluate the local regularity of the functions. For the solutions of compressible flows that contain multiscale smooth
structures and discontinuities, the wavelet coefficients remain small magnitudes in large-scale smooth regions, but
becomes significant near small-scale high-frequency structures and discontinuities. Therefore, the wavelet
multiresolution analysis of functions provides a natural discontinuity detector. By discarding wavelet coefficients
$d_{J,m}$ that fall below a prescribed tolerance, we can achieve a sparse representation of the physical variable $u$
with a specified accuracy as follows:
\begin{equation}
\begin{aligned}
&\mathcal{P}_{J_0}^{J_{\text{max}}}u(x)
=\sum_{k \in \mathcal{K}_{J_0}}{\bar{u}_{J_0,k}\phi_{J_0,k}}(x)+
\sum_{J=J_0}^{J_{\text{max}}-1}\sum_{m \in \mathcal{A}_J}{d_{J,m}\psi_{J,m}}(x), \\
&\mathcal{K}_{J_0}=\left\{k\in\mathbb Z:\operatorname{supp}\phi_{J_0,k}\cap\Omega\ne\emptyset
\right\},\,
\mathcal{A}_J=\left\{m\in\mathbb Z:\operatorname{supp}\psi_{J,m}\cap\Omega\ne\emptyset,\, |d_{J,m}| \ge \epsilon_J
\right\},
\end{aligned}
\label{eq:umthreshold}
\end{equation}
where $\epsilon_J$ denotes the level-dependent thresholding parameter, and $\mathcal{A}_J$ represents the collection of
the index of the wavelet-basis cells at the level $J$. Here, a wavelet-basis cell that $|d_{J,m}| \ge \epsilon_J$ is
called a significant cell. For the average-interpolating wavelet, Harten~\cite{Harten1993} showed that, by setting
$\epsilon_J=\dfrac{\epsilon_0}{2^{J_{\max}-1-J}}$, the $L^{\infty}$-norm error of the wavelet multiresolution
approximation in the smooth region $\Omega_s$, away from discontinuities, can be bounded by a prescribed tolerance
$\epsilon_0$ as follows
\begin{equation}
\|u-\mathcal P_{J_0}^{J_{\max}} u\|_{L^\infty(\Omega_s)} < C\,\epsilon_0.
\label{eq:errorcontrol}
\end{equation}
This result suggests that wavelet multiresolution analysis provides a rigorous error-control mechanism, thereby enabling
the development of adaptive numerical methods with automatic, threshold-based error control.

Based on the aforementioned tree structure and theoretical foundation, we next present our multiresolution cell
adaptation strategy. It should be noted that Eq.~\eqref{eq:umthreshold} is sufficient for approximating a function with
localized features within the prescribed tolerance $\epsilon_0$. However, when this criterion is directly used for
time-dependent compressible flows with discontinuities, stability issues may arise. To address this difficulty, Liandrat
and Tchamitchian~\cite{Liandrat1992} introduced an adjacent zone, which consists of wavelet-basis cells whose
coefficients are significant or may become significant during one time-integration step, while the adaptive cell
distribution is kept fixed over that step. M{\"u}ller and his collaborators \cite{Hovhannisyan2014,Gerhard2016} also
emphasized that the prediction set of cells should be constructed such that all detail coefficients not included in this
set remain non-significant after the update and evolution procedures. The proposed multiresolution cell adaptation
strategy also follows this idea. 

For compressible flows governed by hyperbolic conservation laws, discontinuities may develop even from sufficiently
smooth initial conditions. Therefore, the adaptive strategy should include a refinement procedure capable of detecting
newly emerging discontinuities based on the cell averages of the physical variables at the coarsest level. Accordingly,
the proposed multiresolution cell adaptation strategy consists of two components: adaptive cell generation at the
coarsest level $J_0$ and adaptive cell generation at the level $J$, denoted by $\mathrm{ACG}\text{-}J_0$ and
$\mathrm{ACG}\text{-}J$, respectively. The set of adaptive multiresolution cells at $t_n$ is denoted by $\mathcal{M}^n$.
Next, we present how to obtain the multiresolution cells $\mathcal{M}^{n+1}$ from the solutions computed on
$\mathcal{M}^n$.

In the $\mathrm{ACG}\text{-}J_0$ strategy, newly emerging discontinuities are detected using the smoothness indicator
proposed by Jiang and Shu~\cite{JiangShu1996}:
\begin{equation}
\begin{aligned}
IS_{J_0,l}
&=\dfrac{13}{12}\left(\bar u_{J_0,l-1}-2 \bar u_{J_0,l}+\bar u_{J_0,l+1}\right)^2+
\dfrac{1}{4}\left(\bar u_{J_0,l-1}-\bar u_{J_0,l+1}\right)^2 \\
&=\left(u^{\prime}\Delta x\right)^2\left(1+\mathcal{O}\left(\Delta x^2\right)\right),\quad \Delta x= \dfrac{1}{2^{J_0}}.
\end{aligned}
\label{eq:ISJ0}
\end{equation}
Eq.~\eqref{eq:ISJ0} shows that $IS_{J_0,l}$ is mainly governed by $(u')^2(\Delta x)^2$. For each cell at the coarsest
level, if $IS_{J_0,l} > M_0(\Delta x)^2$, where $M_0$ is a user-defined parameter for identifying newly developed steep
gradients, the cell and its neighboring cells are refined by activating their child cells at the next finer level. The
detailed procedure is summarized as follows:
\begin{equation}
\begin{aligned}
&\textbf{ACG-}J_0: \\
&\text{For each cell } l \text{ at the coarsest level } J_0  \\
&\quad
\mathcal{I}_{J_0,l} \text{ are activated }, \\
&\quad
IS_{J_0,l}
=
\frac{13}{12}
\left(
\bar{u}_{J_0,l-1}
-2\bar{u}_{J_0,l}
+\bar{u}_{J_0,l+1}
\right)^2
+
\frac{1}{4}
\left(
\bar{u}_{J_0,l-1}
-\bar{u}_{J_0,l+1}
\right)^2,\\
&\quad
\text{if } IS_{J_0,l}\ge M_0(\Delta x)^2
\quad\Longrightarrow\quad
\left\{
\mathcal{I}_{J_0+1,2k},
\mathcal{I}_{J_0+1,2k+1}
\right\}_{k\in\{l-1,\,l,\,l+1\}}
\text{ are activated.}
\end{aligned}
\label{eq:ACGJ0}
\end{equation}

As for the $\mathrm{ACG}\text{-}J$, significant cells are identified by evaluating the wavelet coefficient $d_{J,k}$ of
the physical variables. In the adaptive process, symmetric wavelets are adopted for function decomposition and
reconstruction. Daubechies~\cite{Daubechies1992} noted that symmetric errors are generally more tolerable than
asymmetric ones in compression applications, suggesting that symmetry is advantageous for obtaining sparse
representations while preserving the essential shape of the reconstructed function. To maintain consistency with the
accuracy of the fourth-order wavelet upwind scheme, the $N=5$ symmetric average-interpolating wavelet is adopted for
decomposition and reconstruction in the adaptive procedure. There are three ingredients in the $\mathrm{ACG}\text{-}J$:
\begin{enumerate}[(a)]
\itemsep=0pt
\item If a significant cell satisfies $|d_{J,k}|\ge 2^{N-1}\epsilon_J$ with $J+2\le J_{\max}-1$, it is refined by two
levels. The cells $I_{J+2,4k},\ldots,I_{J+2,4k+3}$ are first activated. Then, the cell and the corresponding cells in
its adjacent zone are activated. The adjacent zone is determined by
\begin{equation}
J'=J + 1,\quad |k'-2k| \le 2 n_{r0}, 
\label{eq:ACG-JA}
\end{equation}
where $n_{r0}$ is the width of the adjacent zone. In all numerical tests, $n_{r0}=2$ is used unless otherwise specified.

\item If a significant cell satisfies $\epsilon_J \le |d_{J,k}| < 2^{N-1}\epsilon_J$, the cell and the associated cells
within its adjacent zone, are activated. The adjacent zone is specified by
\begin{equation}
J'=J + 1,\quad |k'-2k| \le 2 n_{r}, 
\label{eq:ACG-JA1}
\end{equation}
Here, $n_r=2$ is used in all numerical tests.

\item For each cell $I_{J+1,k'}$ activated in (a) and (b), the stencil cells associated with its parent cell at level
$J$ are also activated to ensure the completeness of the FWT and inverse FWT stencils. Specifically, $I_{J,\lfloor
k'/2\rfloor+s}$ is activated for $s=-n_b,\ldots,n_b$, where $n_b$ denotes the stencil width determined by the
multiresolution order $N$.
\end{enumerate}

The corresponding procedures of the $\mathrm{ACG}\text{-}J$ are presented below:
\begin{equation}
\begin{aligned}
&\textbf{ACG-}J: \\
&\text{For } J=J_{\max}-1,\ldots,J_0,\qquad
\epsilon_J=\frac{\epsilon_0}{2^{J_{\max}-1-J}},\,n_b=(N-1+|b|)/2,\\
&\text{For each wavelet-basis cell} \\
&\quad
\text{if } |d_{J,k}|\ge \epsilon_J,\text{ then}\\
&\qquad\quad
\text{if } |d_{J,k}|\ge 2^{N-1}\epsilon_J,\ J+2\le J_{\max}\\
&\qquad\qquad\quad
\mathcal I_{J+2,4k},\ldots,\mathcal I_{J+2,4k+3}
\text{ are activated,} \\
&\qquad\qquad\quad
\mathcal I_{J,k},\ 
\mathcal I_{J+1,2(k+s)},\ 
\mathcal I_{J+1,2(k+s)+1}
\text{ are activated for } s=-n_{r0},\ldots,n_{r0},\\
&\qquad\quad
\text{else} \\
&\qquad\qquad\quad
\mathcal I_{J,k},\ 
\mathcal I_{J+1,2(k+s)},\ 
\mathcal I_{J+1,2(k+s)+1}
\text{ are activated for } s=-n_r,\ldots,n_r,\\
&\quad
\text{if }I_{J+1,k'} \text{ is activated}, \text{ then}\\
&\qquad\quad
\mathcal I_{J,\lfloor k'/2\rfloor+s} \text{ is activated for } s=-n_b,\ldots,n_b.
\end{aligned}
\label{eq:ACGJ}
\end{equation}

After the $\mathrm{ACG}\text{-}J_0$ and $\mathrm{ACG}\text{-}J$ procedures are performed, the refined multiresolution
cells are obtained. However, although the stencil cells associated with all significant cells and their adjacent cells
have been marked in the $\mathrm{ACG}\text{-}J$ procedure, the resulting marked-cell set may still be insufficient to
guarantee a complete tree structure. By introducing a recursive stencil recheck procedure, a complete set of
multiresolution cells $\mathcal{M}^{n+1}$ can be obtained for advancing the solution from $t_n$ to $t_{n+1}$.
\subsubsection{Adaptive multiresolution wavelet upwind scheme}
In this subsection, we present the main ingredients of the adaptive multiresolution wavelet upwind scheme. Once the
complete set of multiresolution cells $\mathcal{M}^{n+1}$ has been generated, the wavelet coefficients associated with
the wavelet-basis cells $\mathcal I_{J,k}^{w}$ in $\mathcal{M}^{n+1}$ are updated according to
\begin{equation}
d_{J,k}^{n+1}=
\begin{cases}
0, & |d_{J,k}^n|\le \epsilon_J \ \text{and}\ \mathcal{I}_{J,k}^{w}\notin \mathcal{M}^{n+1},\\
d_{J,k}^n, & \text{otherwise}.
\end{cases}
\label{eq:djk_updating}
\end{equation}
It should be noted that the order of the thresholding process is important for achieving desirable accuracy. In many
adaptive multiresolution algorithms, thresholding is performed during the cell or node refinement
procedure~\cite{Harten1995CPAM,Harten1996SIAM, Holmstrom1999}. With such a treatment, some wavelet coefficients that may
still be required in the updated multiresolution cell set $\mathcal{M}^{n+1}$ could be discarded before the final cell
distribution is obtained. In the present strategy, the thresholding process is therefore performed after the complete
multiresolution cell set $\mathcal{M}^{n+1}$ has been generated, in order to retain more physical information from the
current solution.

Then, the cell averages of the physical variable $u$ on the leaf cells $\mathcal I_{J,l}^{\mathrm{leaf}}$ in
$\mathcal{M}^{n+1}$ are computed using the inverse FWT algorithm presented in
subsection~\ref{subsec:waveletapproximation}. Next, the FWT algorithm is applied to the positive- and negative-direction
upwind wavelets to compute the corresponding wavelet coefficients of the variable $u$. The interface values for all
active leaf cells on are then reconstructed by the wavelet multiresolution approximation:
\begin{equation}
\begin{aligned}
&u_{l+1/2}^{L,MW}
=\sum_{k \in \mathcal{K}_{J_0}^{+}}{\bar{u}_{J_0,k}\phi_{J_0,k}^{+}}(x_{l+1/2})+
\sum_{J=J_0}^{J_{\text{max}}-1}\sum_{m \in \mathcal{A}_J^{+}}{d_{J,m}^{+}\psi_{J,m}^{+}}(x_{l+1/2}), \\
&\mathcal{K}_{J_0}^{+}=\left\{k\in\mathbb Z:\operatorname{supp}\phi_{J_0,k}^{+}\cap\Omega\ne\emptyset
\right\},\,
\mathcal{A}_J^{+}=\left\{m\in\mathbb Z:\operatorname{supp}\psi_{J,m}^{+}\cap\Omega\ne\emptyset,\, 
\mathcal{I}_{J,m}^{w} \in \mathcal{M}^{n+1}
\right\},
\end{aligned}
\label{eq:ulAM}
\end{equation}

\begin{equation}
\begin{aligned}
&u_{l+1/2}^{R,MW}
=\sum_{k \in \mathcal{K}_{J_0}^{-}}{\bar{u}_{J_0,k}\phi_{J_0,k}^{-}}(x_{l+1/2})+
\sum_{J=J_0}^{J_{\text{max}}-1}\sum_{m \in \mathcal{A}_J^{-}}{d_{J,m}^{-}\psi_{J,m}^{-}}(x_{l+1/2}), \\
&\mathcal{K}_{J_0}^{-}
=\left\{k\in\mathbb Z:
\operatorname{supp}\phi_{J_0,k}^{-}\cap\Omega\ne\emptyset
\right\},\,
\mathcal{A}_J^{-}
=\left\{m\in\mathbb Z:
\operatorname{supp}\psi_{J,m}^{-}\cap\Omega\ne\emptyset,\,
\mathcal{I}_{J,m}^{w} \in \mathcal{M}^{n+1}
\right\},
\end{aligned}
\label{eq:urAM}
\end{equation}
where the superscripts $(+)$ and $(-)$ correspond to the positive- and negative-direction upwind wavelets, respectively.
For compressible flows with smooth solutions, the interface values obtained by the wavelet multiresolution approximation
on the multiresolution cell set $\mathcal{M}^{n+1}$ lead to high-order accurate results. However, spurious oscillations
may arise in the presence of discontinuities, such as shock waves and contact discontinuities. To suppress these
oscillations, a limiting projection is incorporated into the multiresolution average-interpolating wavelet upwind scheme
on the adaptive multiresolution cell set $\mathcal{M}^{n+1}$. Since the sparse representation of the physical variable
$u$ obtained by the wavelet multiresolution analysis reflects its local regularity, the resulting leaf cell distribution
in the $\mathcal{M}^{n+1}$ can also serve as a natural indicator of local nonsmoothness. In particular, discontinuities
and local steep gradients are preliminarily detected and localized by the finest-level leaf cells, and the limiting
projection is subsequently applied within these nonsmooth regions.

In the present work, the two-stage BVD reconstruction presented in subsection~\ref{subsec:BVDreconstruction} is
performed on the finest-level leaf cells whose stencil cells required by the THINC scheme are all located at the finest
resolution level. Accordingly, the BVD reconstruction in the adaptive multiresolution wavelet upwind scheme is expressed
as
\begin{equation}
\begin{aligned}
&\textbf{BVD reconstruction in the set } \mathcal{M}^{n+1}: \\
&\text{For each finest-level leaf cell } \mathcal{I}_{J,l}^{\mathrm{leaf}} \in \mathcal{M}^{n+1},\\
&\quad
\text{if the stencil cells required by the THINC scheme are all located at the finest level}, \\ 
&\qquad
\text{perform the BVD reconstruction to obtain }
u_{l+1/2}^{L},\,u_{l+1/2}^{R},\\
&\quad
\text{otherwise},\\
&\qquad
u_{l+1/2}^{L}=u_{l+1/2}^{L,MW},\quad
u_{l+1/2}^{R}=u_{l+1/2}^{R,MW}.
\end{aligned}
\label{eq:BVDadaptive}
\end{equation}
The interface values obtained on the multiresolution cell set $\mathcal{M}^{n+1}$ are then used to evaluate the
numerical fluxes through the selected Riemann solver. The solution is subsequently advanced to $t_{n+1}$ using the
SSP-RK3 scheme for time marching.

With the above ingredients presented, the proposed adaptive multiresolution wavelet upwind scheme is referred to as the
adaptive multiresolution average-interpolating wavelet scheme of fourth order with two-stage THINC reconstruction based
on the BVD algorithm, denoted by $\mathrm{AMAIW_4T_2\mbox{-}BVD}$. The overall procedure of the proposed adaptive scheme
is summarized below:
\begin{enumerate}[(1)]
\itemsep=0pt

\item Select a characteristic variable that reflects the localized nonsmoothness of the physical variables. Evaluate the
wavelet coefficients on the wavelet-basis cells $\mathcal I_{J,k}^{w}$ in the $\mathcal{M}^{n}$ using the FWT algorithm
in Eq.~\eqref{eq:FWT} of the $N=5$ symmetric average-interpolating wavelet.

\item Obtain the complete set of multiresolution cells $\mathcal{M}^{n+1}$ through the $\mathrm{ACG}\text{-}J_0$
procedure in Eq.~\eqref{eq:ACGJ0}, the $\mathrm{ACG}\text{-}J$ procedure in Eq.~\eqref{eq:ACGJ}, and the recursive
stencil recheck.

\item Update the wavelet coefficients on the wavelet-basis cells $\mathcal I_{J,k}^{w}$ in $\mathcal{M}^{n+1}$ according
to Eq.~\eqref{eq:djk_updating}, and then compute the physical variables on the leaf cells $\mathcal
I_{J,l}^{\mathrm{leaf}}$ in $\mathcal{M}^{n+1}$ using the inverse FWT algorithm in Eq.~\eqref{eq:IFWT} of the $N=5$
symmetric average-interpolating wavelet.

\item Evaluate the wavelet coefficients on the wavelet-basis cells $\mathcal I_{J,k}^{w}$ in $\mathcal{M}^{n+1}$ using
the FWT algorithms in Eq.~\eqref{eq:FWT} of the $N=4$, $b=1$ and $b=-1$ asymmetric upwind average-interpolating
wavelets. Then, compute the left and right interface values at the leaf cells using the corresponding upwind wavelet
multiresolution approximations given in Eqs.~\eqref{eq:ulAM} and~\eqref{eq:urAM}, respectively. Perform the BVD
reconstruction in Eq.~\eqref{eq:BVDadaptive} to obtain the final interface values, and advance the solution from $t_n$
to $t_{n+1}$.

\item If $t_{n+1}=t_{\mathrm{end}}$, the computation is terminated; otherwise, repeat procedures (1)--(4).

\end{enumerate}

\noindent\textbf{Remark 1.} For the Euler systems, the BVD reconstruction is also performed in the local characteristic
fields. Since the wavelet multiresolution approximation on the given multiresolution cell set $\mathcal{M}^{n+1}$ is
linear, the corresponding interface values can be computed entirely in the physical space. Only those used in the BVD
reconstruction need to be transformed into the local characteristic fields. Thus, the proposed adaptive multiresolution
wavelet upwind scheme with the BVD reconstruction avoids repeated characteristic projections over different wavelet
reconstruction stencils, which simplifies the reconstruction procedure and reduces the overhead for systems of
conservation laws. In contrast, the interface values of the THINC scheme must be reconstructed in the local
characteristic fields because the THINC reconstruction is nonlinear. \\
\noindent\textbf{Remark 2.} The conventional conservative adaptive multiresolution methods hybridized with classical
finite volume schemes~\cite{Cohen2003, Roussel2004, Tang2008} usually require additional marking of ghost cells within
the stencil of the discretization scheme near the interface between coarse- and fine-level cells. The corresponding cell
averages must be reconstructed by the inverse FWT algorithm before computing the interface values. This increases the
complexity of the algorithmic structure and may reduce computational efficiency. In contrast, the proposed adaptive
multiresolution wavelet upwind scheme does not require this additional procedure. The interface values are computed
directly by the wavelet multiresolution approximation, making the algorithm simpler and more efficient. \\
\noindent\textbf{Remark 3.} In the proposed adaptive wavelet upwind scheme, the BVD reconstruction is restricted to the
finest-level leaf cells. This localized implementation substantially reduces the associated computational cost and
provides a novel framework for developing adaptive BVD-based shock-capturing methods. \\
\noindent\textbf{Remark 4.} It should be noted that the proposed conservative adaptive multiresolution wavelet upwind
scheme is not limited to the fourth-order accurate formulation. Higher-order conservative adaptive wavelet upwind
schemes can also be constructed by using asymmetric average-interpolating wavelets with higher multiresolution orders.
\subsubsection{Analysis of the conservation property of numerical schemes}
\label{subsec:analysisconservation}
In this subsection, we analyze the conservativity of the proposed wavelet upwind scheme and other main numerical schemes
on both fixed and adaptive meshes. Sebastian and Shu~\cite{Sebastian2003} pointed out that interpolation, regardless of
its specific form, is generally nonconservative unless it is based on cell averages. For the Shu--Osher conservative
finite difference form, Merriman~\cite{Merriman2003} showed that its strict conservation property is essentially tied to
uniform computational meshes. Although conservativity holds for certain specially stretched nonuniform meshes in the
physical domain, such as quadratically or exponentially stretched meshes, this does not imply that the Shu--Osher form
provides a directly conservative discretization on arbitrary nonuniform physical meshes. Consequently, adaptive finite
difference schemes based on the Shu--Osher form, in which the discretization is directly carried out on adaptive meshes
in the physical domain, are generally nonconservative~\cite{Merriman2003}. 

Merriman~\cite{Merriman2003} suggested that it might be possible to develop a fully adaptive version of the Shu--Osher
finite difference scheme, based on Harten's adaptive multiresolution framework using average-interpolating wavelets.
Motivated by this conjecture, we examine whether the Shu--Osher conservative finite difference form can be directly
extended to an adaptive multiresolution average-interpolating wavelet setting. Through an observation of the proposed
adaptive multiresolution wavelet upwind scheme, we provide a counterexample to such a direct extension, while not
excluding the possibility of achieving conservativity with additional correction mechanisms. Considering a flux function
$f(u)$ satisfying $f'(u)\ge 0$, the one-dimensional scalar conservation law in Eq.~\eqref{eq:model} is discretized by
the following upwind wavelet multiresolution approximation:
\begin{equation}
\begin{aligned}
\frac{\mathrm{d} u_l(t)}{\mathrm{d} t}
&= -\frac{1}{\Delta x_l}
\left(\hat f_{l+1/2}^{L,MW}-\hat f_{l-1/2}^{L,MW}\right), \\
\hat f_{l+1/2}^{L,MW}
&=
\sum_{k \in \mathcal{K}_{J_0}^{+}}
f_{J_0,k}\phi_{J_0,k}^{+}(x_{l+1/2})
+
\sum_{J=J_0}^{J_{\max}-1}
\sum_{m \in \mathcal{A}_J^{+}}
d_{J,m}^{+,f}\psi_{J,m}^{+}(x_{l+1/2}).
\end{aligned}
\label{eq:ODEsFDM}
\end{equation} 
If the above discretization were conservative in the sense of the Shu--Osher form, the flux values used in the
multiresolution approximation would have to satisfy a consistency relation between two successive resolution levels.
Specifically, the flux value at a coarse-level node should be equal to the arithmetic average of the corresponding two
fine-level flux values, namely
\begin{equation}
f_{J,k}=\dfrac{1}{2}\left(f_{J+1,2k}+f_{J+1,2k+1}\right).
\end{equation} 
However, this relation is a consistency requirement of the average-interpolating wavelet for averaged quantities, and it
is not generally satisfied for a nonlinear flux function $f(u)$ on adaptive multiresolution nodes. Therefore, when the
adaptive multiresolution wavelet upwind approximation is directly cast into the conservative Shu--Osher finite
difference form, the resulting scheme is generally nonconservative on adaptive meshes.

It is worth noting that the present nonconservative formulation may still preserve a discrete weighted sum of the nodal
values. Specifically, the quantity $\sum_l u_l \Delta x_l$ may remain unchanged during the computation. However, this
property should not be confused with strict conservation of cell-average quantities, because $u_l$ represents a nodal or
pointwise degree of freedom rather than the cell average over $I_l$. It should be noted that this observation does not
contradict the conservation property of the Shu--Osher finite difference formulation on uniform meshes. This is because,
on uniform meshes, the nodal values can be consistently interpreted as cell averages due to the commutation property
between the cell-average operator and the finite-difference operator~\cite{Merriman2003}. Consequently, the preservation
of the nodal sum admits a conservative interpretation in terms of cell averages. In contrast, for adaptive
multiresolution meshes, this commutation property is generally lost, and the adaptive Shu--Osher finite-difference
formulation therefore no longer preserves conservation in the cell-average sense. Nevertheless, for low-order numerical
schemes, the discrepancy between nodal values and cell averages may be comparable to the leading truncation error.
Therefore, preservation of the weighted nodal sum may be regarded as an approximate conservation property within the
accuracy of the scheme.

For the finite volume scheme on fixed uniform or nonuniform meshes given in Eq.~\eqref{eq:ODEs}, the conservation
property is preserved both locally at the cell level and globally over the physical domain. This is because the scheme
evolves cell averages by numerical fluxes across cell interfaces, so that the flux contribution at each internal
interface is shared by two neighboring cells with opposite signs. This mechanism is also embedded in compact high-order
schemes such as the DG scheme~\cite{Cockburn1989} and MCV scheme~\cite{Ii2009}. Although these schemes use multiple
degrees of freedom to represent the solution within each cell, strict conservation is ensured through the evolution of
the cell averages. In DG schemes, the constant mode corresponds to the element average and satisfies a flux-difference
update when the test function is taken as unity. In MCV schemes, the cell average moment is explicitly constrained and
updated by numerical fluxes across cell interfaces. Moreover, we note that although flux reconstruction methods evolve
nodal values similar to finite difference schemes, cell-average conservation is inherently incorporated through the
construction of the correction flux function~\cite{Huynh2007}. The above discussion indicates that strict conservation
in modern conservative schemes is essentially associated to the evolution of cell averages.

For existing pure wavelet-based adaptive schemes, the associated wavelet basis functions, such as Daubechies wavelets
and interpolating wavelets, are generally not constructed from cell averages. As a result, strict conservation is not
inherently embedded in the wavelet representation itself, and the conservation property should be examined from both the
spatial discretization and the adaptive redistribution process. From the perspective of PDE discretization, wavelet
Galerkin methods within the classical Galerkin framework may preserve global conservation since the governing
conservation law is enforced in a weak sense. However, they generally do not directly yield local conservation in the
finite-volume sense because the resulting discrete equations do not explicitly provide a cell-wise flux balance. As
discussed by Abgrall~\cite{Abgrall2023}, Galerkin-type methods may be recast into a conservative residual-distribution
or finite volume form if the element residuals satisfy a suitable local conservation relation. However, this relies on
an additional residual-distribution or flux-reconstruction interpretation. For wavelet collocation methods, a
conservative form of the wavelet collocation upwind scheme on uniform cells was derived in our previous
study~\cite{Yang2025AMS} by rewriting the conventional wavelet approximation into an interface flux-difference form. It
should be noted that the conservative form of the wavelet collocation upwind scheme essentially falls within the
Shu--Osher conservative finite-difference framework. It is thus not strictly conservative on adaptive nonuniform cells
in the physical domain. For function reconstruction during the adaptive process, global conservation may be retained
only if the compact supports of the discarded wavelet functions are fully contained in the computational domain, owing
to their zero-mean property. Otherwise, truncation or redistribution of wavelet coefficients may alter the global
integral and thus introduce non-conservative errors. 

The above analysis indicates that existing pure wavelet-based adaptive schemes whose wavelet bases are not constructed
from cell averages generally encounter inherent difficulties in achieving strict conservation during both the spatial
discretization and the adaptive redistribution of degrees of freedom. This observation motivates the present
construction of a cell-average-based adaptive multiresolution wavelet scheme. In the present study, the proposed
adaptive multiresolution wavelet upwind scheme exploits the advantages of average-interpolating wavelets in two aspects:
the conservative discretization of the governing equations and the conservative redistribution of physical variables
during cell adaptation. Specifically, the upwind discretization is constructed using a pair of average-interpolating
wavelets with upwind properties in the positive and negative directions, respectively. Meanwhile, the decomposition and
reconstruction of physical variables during adaptation are performed using symmetric average-interpolating wavelets,
which preserve the cell-average representation while retaining the essential solution profiles and improving the
compression capability. Consequently, the proposed scheme maintains strict conservation through both the flux-difference
evolution of the governing equations and the conservative redistribution of physical quantities on adaptive
multiresolution cells.
\subsection{Computational complexity analysis}
At the end of this section, we discuss the computational complexity of the proposed $\mathrm{AIW_4T_2\mbox{-}BVD}$ and
$\mathrm{AMAIW_4T_2\mbox{-}BVD}$ schemes. We mainly focus on the computational costs of the wavelet multiresolution
approximation, FWT, and inverse FWT, which constitute the primary computational components of the proposed schemes.

For the $\mathrm{AIW_4T_2\mbox{-}BVD}$ scheme on uniform cells, the reconstruction stencil of the wavelet approximation
for computing the interface values involves six cells. This stencil size is larger than that of the WENO5-Z scheme but
smaller than that of the WENO7-Z scheme. Therefore, the computational cost of the interface reconstruction in the
$\mathrm{AIW_4T_2\mbox{-}BVD}$ scheme remains comparable to those of classical high-order shock-capturing schemes.

For the $\mathrm{AMAIW_4T_2\mbox{-}BVD}$ scheme, the wavelet multiresolution approximation operator is linear and
sparse. For each active leaf cell, the number of cells involved in the reconstruction stencil is bounded by
$\left(J_{\max}-J_0+1\right)2(N-1)$. Let $N_1$ denote the total number of active leaf cells. Then, the computational
cost of the wavelet multiresolution approximation is bounded by
$\mathcal{O}\left(\left(J_{\max}-J_0+1\right)2(N-1)N_1\right)$. In practical computations, the number of resolution
levels is prescribed to be comparatively small; for example, $J_{\max}-J_0\leq 8$ in the present study. Therefore, for a
fixed multiresolution order $N$, this bound reduces to $\mathcal{O}(N_1)$.

For a full multiresolution representation, the FWT and inverse FWT have linear complexity because the computational
costs over successive resolution levels form a geometric series. In the present adaptive implementation, the
multiresolution representation is further sparsified through the thresholding process, and the transforms are performed
only on the active leaf cells and retained wavelet coefficients. Hence, their computational costs scale linearly with
the number of involved active degrees of freedom. Therefore, the proposed adaptive multiresolution wavelet upwind scheme
has a linear or nearly linear computational cost in practical computations, while retaining the advantages of adaptive
data compression and conservative high-order shock-capturing discretization. This makes it an efficient framework for
solving compressible flows.
\section{Numerical results}
\label{Sec:Numerical_test}
In this section, we present numerical results for several benchmark tests to verify the order of accuracy,
conservativity, computational efficiency, and robustness of the proposed conservative adaptive multiresolution wavelet
upwind scheme. The SSP-RK3 scheme~\cite{Gottlieb2001} is employed for time marching, with a CFL number of 0.4 unless
otherwise specified. The time step for the adaptive wavelet upwind scheme is determined by the mesh size at the finest
level. For the one-dimensional Euler systems, the HLLC Riemann solver~\cite{Toro2013} is employed to compute the
numerical fluxes. To measure the numerical errors, the following $L^1$- and $L^{\infty}$-norm errors are defined:
\begin{equation}
\|E\|_{L^1}
=
\sum_{l}
\left|\bar u_l^s-\bar u_l^e\right|\Delta x_l,
\label{eq:errorl1}
\end{equation}
\begin{equation}
\|E\|_{L^{\infty}}
=
\max\limits_{l}
\left|\bar u_l^s-\bar u_l^e\right|,
\label{eq:errorlinf}
\end{equation}
where $\bar u_l^e$ and $\bar u_l^s$ denote the exact and numerical solutions at the $l$th leaf cell, respectively. To
quantify the conservativity of the proposed $\mathrm{AMAIW_4T_2\mbox{-}BVD}$ scheme, the conservation error is defined
as
\begin{equation}
E_{\mathrm{cons}}^n
=
\left|
\sum_{\mathcal I_l \in \mathcal{M}^{n}} \bar{u}_l^{\,n}\Delta x_l
-
\sum_{\mathcal I_l \in \mathcal{M}^{0}} \bar{u}_l^{\,0}\Delta x_l
\right|.
\label{eq:errorconser}
\end{equation}

\subsection{Linear scalar equation in one dimension}
In this subsection, numerical tests are conducted to the following one-dimensional linear scalar equation with the
periodic boundary condition by the proposed wavelet upwind schemes:
\begin{equation}
u_t+u_x=0, \quad x \in [-1,1].
\label{eq:1DlinearEquation}
\end{equation}
\subsubsection{Accuracy and error control tests of sine wave advection}
We first consider the advection of a sine wave to evaluate the order of accuracy of the proposed AIWU4 and
$\mathrm{AIW_4T_2\mbox{-}BVD}$ schemes with the initial condition:
\begin{equation}
u(x,0)=\sin(\pi x).
\label{eq:sinwave}
\end{equation}

A CFL number of $1/16$ is used in this test to minimize the influence of the time integration scheme on the accuracy of
the numerical results. The solutions at $t=2.0$ are computed through refinement tests, and the corresponding errors and
orders of accuracy in different norms are listed in Table~\ref{tab:accuracy_linear}, where $N_1$ denotes the total
number of leaf cells. It can be observed that the orders of accuracy of the proposed schemes agree well with the
theoretical order. Moreover, the numerical errors and convergence orders of the linear $\mathrm{AIWU4}$ scheme are
identical to those of the $\mathrm{AIW_4T_2\mbox{-}BVD}$ scheme. This indicates that, for smooth solutions, the BVD
reconstruction selects the high-order average-interpolating wavelet upwind candidate.
\begin{table}[!htbp]
\centering
\caption{Numerical errors and orders of accuracy for the linear scalar equation.}
\label{tab:accuracy_linear}
\begin{tabular*}{0.95\textwidth}{@{\extracolsep{\fill}}cccccc}
\toprule
Schemes & $N_1$ & $L^\infty$ error & $L^\infty$ order & $L^1$ error & $L^1$ order \\
\midrule
\multirow{6}{*}{$\mathrm{AIWU4}$} & 16  & $3.017\times10^{-3}$ & --   & $3.817\times10^{-3}$ & --   \\
& 32  & $1.850\times10^{-4}$ & 4.03 & $2.350\times10^{-4}$ & 4.02 \\
& 64  & $1.149\times10^{-5}$ & 4.01 & $1.462\times10^{-5}$ & 4.01 \\
& 128 & $7.173\times10^{-7}$ & 4.00 & $9.133\times10^{-7}$ & 4.00 \\
& 256 & $4.483\times10^{-8}$ & 4.00 & $5.707\times10^{-8}$ & 4.00 \\
& 512 & $2.803\times10^{-9}$ & 4.00 & $3.569\times10^{-9}$ & 4.00 \\
\multirow{6}{*}{$\mathrm{AIW_4T_2\mbox{-}BVD}$} & 16  & $3.017\times10^{-3}$ & --   & $3.817\times10^{-3}$ & --   \\
& 32  & $1.850\times10^{-4}$ & 4.03 & $2.350\times10^{-4}$ & 4.02 \\
& 64  & $1.149\times10^{-5}$ & 4.01 & $1.462\times10^{-5}$ & 4.01 \\
& 128 & $7.173\times10^{-7}$ & 4.00 & $9.133\times10^{-7}$ & 4.00 \\
& 256 & $4.483\times10^{-8}$ & 4.00 & $5.707\times10^{-8}$ & 4.00 \\
& 512 & $2.803\times10^{-9}$ & 4.00 & $3.569\times10^{-9}$ & 4.00 \\
\bottomrule
\end{tabular*}
\end{table}

Next, we assess the error-control capability of the proposed adaptive multiresolution wavelet upwind scheme. The
numerical errors under different prescribed tolerances $\epsilon_0$ are computed using the proposed
$\mathrm{AMAIW_4T_2\mbox{-}BVD}$ scheme with $J_0=4$ and $J_{\max}=8$, as shown in Table~\ref{tab:accuracy_epsilon}. It
can be seen that the errors in the $L^1$- and $L^\infty$-norm are consistent with the prescribed tolerances. This
confirms the threshold-based error-control capability of the proposed adaptive scheme, in accordance with the
theoretical estimate given in Eq.~\eqref{eq:errorcontrol}.
\begin{table}[!htbp]
\centering
\caption{Error-control tests with $J_0=4$, $J_{\max}=8$ under different thresholding parameters for the sine wave
advection.}
\label{tab:accuracy_epsilon}
\begin{tabular*}{0.95\textwidth}{@{\extracolsep{\fill}}cccc}
\toprule
$\epsilon_0$ & $N_1$ & $L^\infty$ error & $L^1$ error \\
\midrule
$1.000\times10^{-5}$  & 64  & $1.149\times10^{-5}$ & $1.462\times10^{-5}$ \\
$1.000\times10^{-6}$  & 128 & $5.842\times10^{-6}$ & $2.676\times10^{-6}$ \\
$1.000\times10^{-7}$  & 128 & $7.168\times10^{-7}$ & $9.125\times10^{-7}$ \\
$1.000\times10^{-8}$  & 256 & $1.192\times10^{-7}$ & $8.086\times10^{-8}$ \\
$1.000\times10^{-9}$  & 256 & $4.479\times10^{-8}$ & $5.702\times10^{-8}$ \\
$1.000\times10^{-10}$ & 512 & $2.830\times10^{-9}$ & $3.883\times10^{-9}$ \\
\bottomrule
\end{tabular*}
\end{table}

We further evaluate the conservativity of the proposed adaptive scheme. In the following, the numerical result denoted
by level $J$ refers to the solution computed by the $\mathrm{AIW_4T_2\mbox{-}BVD}$ scheme on uniform cells at resolution
level $J$. The conservation errors obtained by the adaptive $\mathrm{AMAIW_4T_2\mbox{-}BVD}$ scheme are compared with
those of the $\mathrm{AIW_4T_2\mbox{-}BVD}$ scheme in Fig.~\ref{fig:sinconserror}. As shown in
Fig.~\ref{fig:sinconserrorJmax}, the conservation errors remain within machine precision for both schemes. Moreover,
Figure~\ref{fig:sinconserrote0} shows that the conservation errors are nearly independent of the prescribed thresholding
tolerances and maintain on the order of machine precision. These results confirm that the proposed adaptive
multiresolution wavelet upwind scheme preserves strict conservation during both the discretized evolution and the
adaptive cell redistribution process.
\begin{figure}[!htbp]
    \centering

    \begin{subfigure}[b]{0.45\textwidth}
        \centering
        \includegraphics[width=\linewidth]{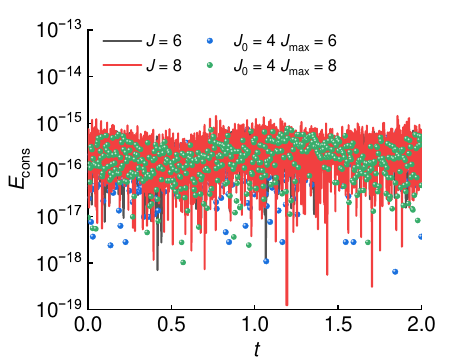}
        \caption{$\epsilon_0=1.0 \times 10^{-6}$.}
        \label{fig:sinconserrorJmax}
    \end{subfigure}
    \hfill
    \begin{subfigure}[b]{0.45\textwidth}
        \centering
        \includegraphics[width=\linewidth]{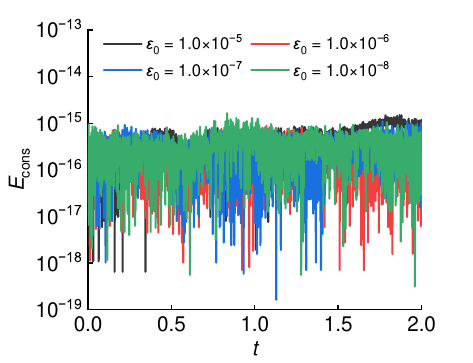}
        \caption{$J_0=4,J_{\max}=8$.}
        \label{fig:sinconserrote0}
    \end{subfigure}

    \caption{Comparison of conservation errors between the adaptive $\mathrm{AMAIW_4T_2\mbox{-}BVD}$ and uniform-cell
    $\mathrm{AIW_4T_2\mbox{-}BVD}$ schemes for the sine wave advection.}
    \label{fig:sinconserror}
\end{figure}

\subsubsection{Advection of complex waves}
To examine the performance of the proposed adaptive multiresolution wavelet upwind scheme in solving profiles with
different local regularity, the propagation of a complex wave devised by Jiang and Shu~\cite{JiangShu1996} is simulated.
The initial condition of this test consists of a smooth but narrow combination of Gaussian, a square wave, a sharp
triangle wave, and a half ellipse expressed as 
\begin{equation}
\begin{aligned}
u(x,0)=
\begin{cases}
\dfrac{1}{6}\left(G(x,\beta,z-\delta)+G(x,\beta,z+\delta)+4 G(x,\beta,z)\right) \quad
& -0.8 \le x \le -0.6,\\
1 \quad
& -0.4 \le x \le -0.2,\\
1-|10(x-0.1)| \quad
& 0 \le x \le 0.2,\\
\dfrac{1}{6}\left(F(x,\alpha,a-\delta)+F(x,\alpha,a+\delta)+4 F(x,\alpha,a)\right) \quad
& 0.4 \le x \le 0.6,\\
0 \quad
& \mathrm{otherwise},
\end{cases}   
\end{aligned}
\label{eq:JiangShu}
\end{equation}
where functions $G$ and $F$ are defined as $G(x,\beta,z)=e^{-\beta(x-z)^2}$, and
$F(x,\alpha,a)=\sqrt{\max{\left(1-\alpha^2\left(x-a\right)^2,0\right)}}$. The constants are given as $a=0.5$, $z=-0.7$,
$\delta=0.005$, $\alpha=10$, $\beta=\ln 2/\left(36\delta^2\right)$.

We first evaluate the capability of the proposed wavelet upwind schemes to resolve solution profiles with different
smoothness, while maintaining high resolution and avoiding spurious oscillations. The numerical results at $t=2.0$ and
$t=500.0$, computed by the $\mathrm{AIW_4T_2\mbox{-}BVD}$ scheme on uniform cells, are compared with those obtained by
the WENO5-Z scheme~\cite{Borges2008,Don2013} in Fig.~\ref{fig:JScompareWENO_uniform}. As shown in
Fig.~\ref{fig:JSt2compareWENO}, the proposed wavelet upwind scheme with the BVD reconstruction produces evidently
steeper jump discontinuities than the WENO5-Z scheme. For the long-time evolution case shown in
Fig.~\ref{fig:JSt500compareWENO}, the $\mathrm{AIW_4T_2\mbox{-}BVD}$ scheme not only achieves higher accuracy in smooth
regions, but also captures significantly sharper discontinuities. These results demonstrate the numerical advantages of
the proposed wavelet upwind scheme with the BVD reconstruction.
\begin{figure}[!htbp]
    \centering

    \begin{subfigure}[b]{0.45\textwidth}
        \centering
        \includegraphics[width=\linewidth]{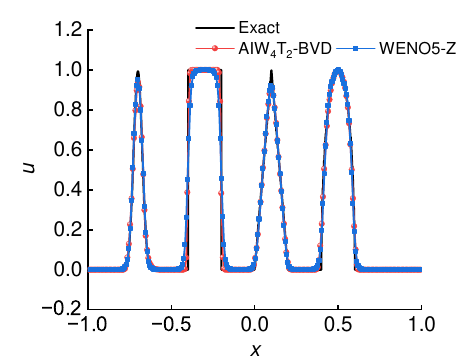}
        \caption{$t=2.0, N_1=256$.}
        \label{fig:JSt2compareWENO}
    \end{subfigure}
    \hfill
    \begin{subfigure}[b]{0.45\textwidth}
        \centering
        \includegraphics[width=\linewidth]{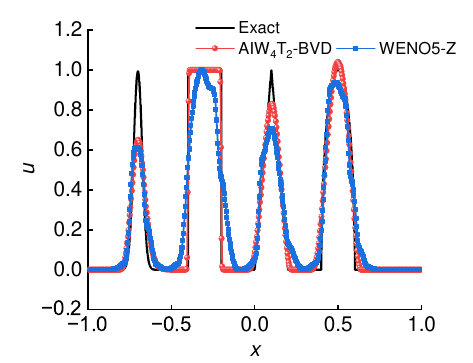}
        \caption{$t=500.0,N_1=512$.}
        \label{fig:JSt500compareWENO}
    \end{subfigure}

    \caption{Numerical results of the Jiang and Shu's problem computed by the $\mathrm{AIW_4T_2\mbox{-}BVD}$ and WENO5-Z
    schemes at $t=2.0$ and $t=500.0$.}
    \label{fig:JScompareWENO_uniform}
\end{figure}

Next, we examine the numerical performance of the proposed adaptive $\mathrm{AMAIW_4T_2\mbox{-}BVD}$ scheme. For this
problem, $n_{r0}$ is set to 6. In Fig.~\ref{fig:JSAdaptivecompareuniform}, the numerical solutions at $t=2.0$ computed
by the adaptive $\mathrm{AMAIW_4T_2\mbox{-}BVD}$ scheme are compared with those obtained by the uniform-cell
$\mathrm{AIW_4T_2\mbox{-}BVD}$ scheme at levels $J=7$ and $J=10$. The uniform-cell solutions at the coarse level $J=7$
are considerably smeared near discontinuities and localized sharp features, whereas the adaptive scheme yields much more
accurate profiles by locally refining cells in these regions. The adaptive results show excellent agreement with the
corresponding finest uniform-cell results at $J=10$, while using far fewer active cells. These results demonstrate that
the proposed adaptive scheme retains the discontinuity and localized-feature capturing capability of the underlying
$\mathrm{AIW_4T_2\mbox{-}BVD}$ scheme at the finest resolution level. They also verify the effectiveness of applying the
BVD reconstruction only on the finest-level cells in the adaptive framework. 
\begin{figure}[!htbp]
    \centering

    \begin{subfigure}[b]{0.45\textwidth}
        \centering
        \includegraphics[width=\linewidth]{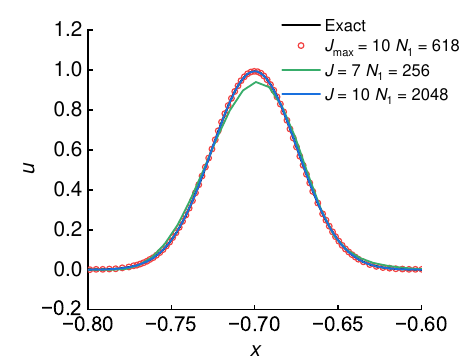}
        \caption{Gaussian profile.}
        \label{fig:JSAdaptivecompareuniformGauss}
    \end{subfigure}
    \hfill
    \begin{subfigure}[b]{0.45\textwidth}
        \centering
        \includegraphics[width=\linewidth]{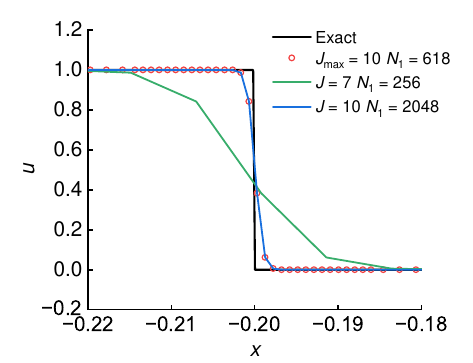}
        \caption{Jump discontinuity.}
        \label{fig:JSAdaptivecompareuniformJump}
    \end{subfigure}

    \vspace{0.3cm}

    \begin{subfigure}[b]{0.45\textwidth}
        \centering
        \includegraphics[width=\linewidth]{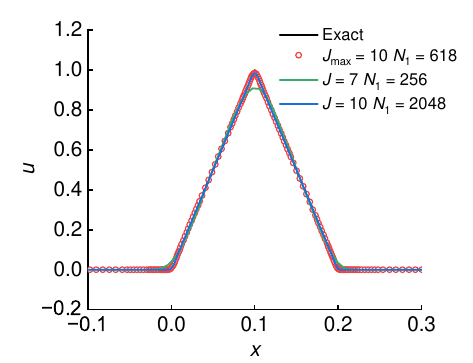}
        \caption{Triangle wave profile.}
        \label{fig:JSAdaptivecompareuniformTri}
    \end{subfigure}
    \hfill
    \begin{subfigure}[b]{0.45\textwidth}
        \centering
        \includegraphics[width=\linewidth]{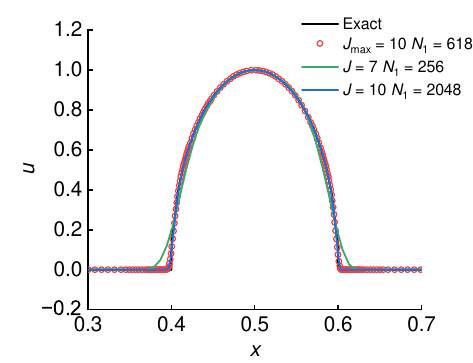}
        \caption{Half ellipse profile.}
        \label{fig:JSAdaptivecompareuniformEllip}
    \end{subfigure}

    \caption{Comparison of the numerical results at $t=2.0$ of the Jiang and Shu's problem obtained by the proposed
    $\mathrm{AMAIW_4T_2\mbox{-}BVD}$ and the $\mathrm{AIW_4T_2\mbox{-}BVD}$ schemes for different solution profiles,
    $J_0=6$, $\epsilon_0=1.0 \times 10^{-5}$.}
    \label{fig:JSAdaptivecompareuniform}
\end{figure}

We further compute the numerical results using the $\mathrm{AMAIW_4T_2\mbox{-}BVD}$ scheme with different finest
resolution levels and thresholding parameters to examine the accuracy, stability, and robustness of the proposed
adaptive multiresolution wavelet upwind scheme. The numerical solutions at $t=2.0$ with a fixed tolerance
$\epsilon_0=1.0 \times 10^{-5}$ and various finest resolution levels are shown in Fig.~\ref{fig:JSJmaxsolution}. As the
finest resolution level increases, the $\mathrm{AMAIW_4T_2\mbox{-}BVD}$ scheme stably resolves sharper discontinuities
and achieves higher accuracy in smooth regions through sparse adaptive representations. 
\begin{figure}[!htbp]
    \centering

    \begin{subfigure}[b]{0.45\textwidth}
        \centering
        \includegraphics[width=\linewidth]{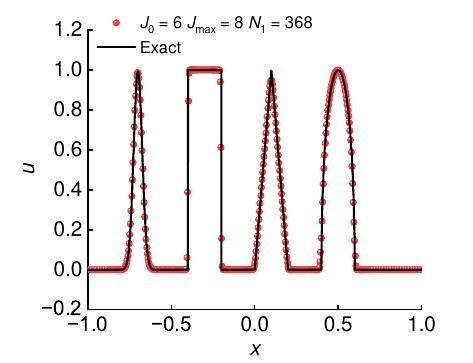}
        \caption{$J_{\max}=8$.}
        \label{fig:JSJmax8solution}
    \end{subfigure}
    \hfill
    \begin{subfigure}[b]{0.45\textwidth}
        \centering
        \includegraphics[width=\linewidth]{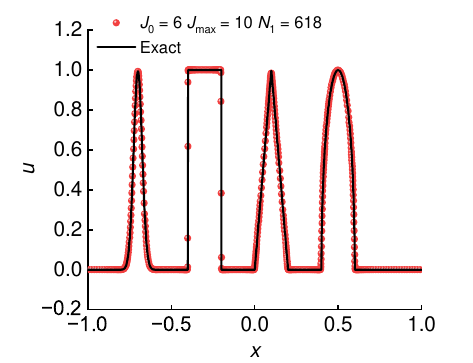}
        \caption{$J_{\max}=10$.}
        \label{fig:JSJmax10solution}
    \end{subfigure}

    \vspace{0.3cm}

    \begin{subfigure}[b]{0.45\textwidth}
        \centering
        \includegraphics[width=\linewidth]{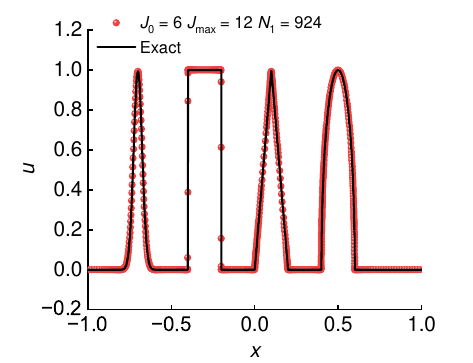}
        \caption{$J_{\max}=12$.}
        \label{fig:JSJmax12solution}
    \end{subfigure}

    \caption{Numerical solutions at $t=2.0$ of the Jiang and Shu's problem obtained by the proposed adaptive
    $\mathrm{AMAIW_4T_2\mbox{-}BVD}$ scheme with different finest resolution levels, $J_0=6$, $\epsilon_0=1.0 \times
    10^{-5}$.}
    \label{fig:JSJmaxsolution}
\end{figure}

To evaluate the error-control capability of the proposed adaptive wavelet upwind scheme for problems involving both
discontinuities and smooth structures, the absolute errors of the numerical results at $t=2.0$ computed by the
$\mathrm{AMAIW_4T_2\mbox{-}BVD}$ scheme with different thresholding tolerances are illustrated in
Fig.~\ref{fig:JSepsilonsolution}. The errors obtained by the adaptive scheme are also compared with those of the
uniform-cell $\mathrm{AIW_4T_2\mbox{-}BVD}$ scheme at the resolution levels $J=J_0$ and $J=J_{\max}$. It can be observed
that the errors of the adaptive scheme generally lie between those of the uniform-cell scheme at the coarsest and finest
resolution levels. This indicates that the nonuniform cell distribution and adaptive cell redistribution do not
deteriorate the basic accuracy guaranteed by the wavelet approximation at the coarsest level. Moreover, as the
prescribed thresholding tolerance decreases, more cells are activated and the overall error magnitude is correspondingly
reduced. Except in the immediate neighborhoods of discontinuities, the numerical errors are effectively controlled
around the prescribed thresholding tolerances. These results verify the threshold-based error-control capability of the
proposed adaptive multiresolution wavelet upwind scheme. 
\begin{figure}[!htbp]
    \centering

    \begin{subfigure}[b]{0.45\textwidth}
        \centering
        \includegraphics[width=\linewidth]{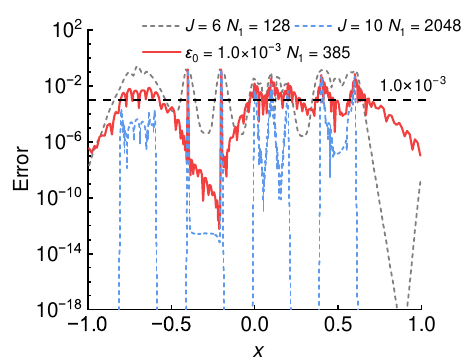}
        \caption{$\epsilon_0=1.0 \times 10^{-3}$.}
        \label{fig:JSepsilonm3}
    \end{subfigure}
    \hfill
    \begin{subfigure}[b]{0.45\textwidth}
        \centering
        \includegraphics[width=\linewidth]{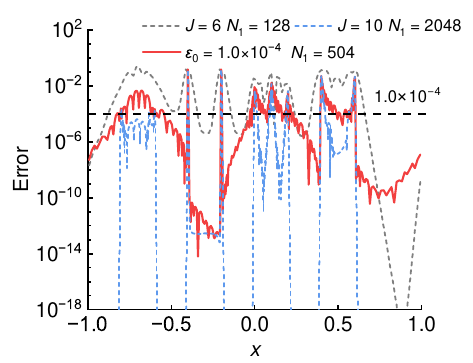}
        \caption{$\epsilon_0=1.0 \times 10^{-4}$.}
        \label{fig:JSepsilonm4}
    \end{subfigure}

    \vspace{0.3cm}

    \begin{subfigure}[b]{0.45\textwidth}
        \centering
        \includegraphics[width=\linewidth]{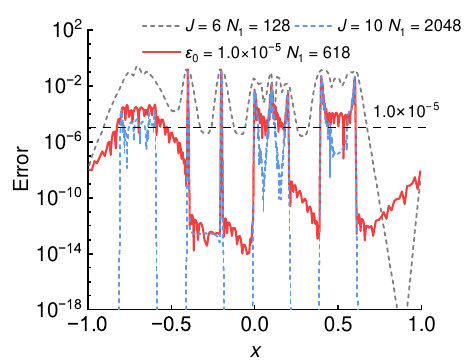}
        \caption{$\epsilon_0=1.0 \times 10^{-5}$.}
        \label{fig:JSepsilonm5}
    \end{subfigure}
    \hfill
    \begin{subfigure}[b]{0.45\textwidth}
        \centering
        \includegraphics[width=\linewidth]{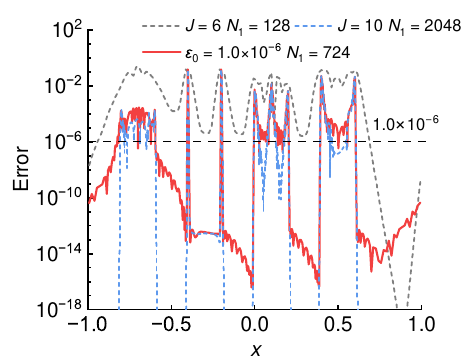}
        \caption{$\epsilon_0=1.0 \times 10^{-6}$.}
        \label{fig:JSepsilonm6}
    \end{subfigure}

    \caption{Numerical solutions at $t=2.0$ of the Jiang and Shu's problem obtained by the proposed adaptive
    $\mathrm{AMAIW_4T_2\mbox{-}BVD}$ scheme with different thresholding parameters, $J_0=6$, and $J_{\max}=10$.}
    \label{fig:JSepsilonsolution}
\end{figure}

We further assess the computational efficiency of the proposed adaptive scheme by analyzing its data compression
performance. An active-cell percentage is defined to measure the degree of data compression as follows:
\begin{equation}
\eta_{\mathrm{act}}=\dfrac{N_{\mathrm{adaptive}}}{N_{\mathrm{finest}}} \times 100\%,
\label{eq:Activepercentage}
\end{equation}
where $N_{\mathrm{adaptive}}$ is the number of adaptive cells, and $N_{\mathrm{finest}}$ is the number of cells at the
finest resolution level. The variation in active-cell percentages of the proposed adaptive scheme with time under
different finest resolution levels and thresholding parameters are shown in Fig.~\ref{fig:JSdiffJmaxPercentage} and
Fig.~\ref{fig:JSdiffepsilonPercentage}, respectively. As shown in Fig.~\ref{fig:JSdiffJmaxPercentage}, the active-cell
percentage  significantly decreases as the finest resolution level increases. In particular, when the finest level is
increased to $J_{\max}=12$, the average active-cell percentage is reduced to about $11.2\%$, indicating that the
proposed adaptive scheme achieves a remarkable data compression capability while retaining the finest-level resolution
near localized features. From Fig.~\ref{fig:JSdiffepsilonPercentage}, the active-cell percentage increases as the
thresholding parameter $\epsilon_0$ decreases. This is reasonable because a smaller threshold imposes a stricter
error-control requirement, and therefore more cells are activated to resolve the localized discontinuities and smooth
wave profiles accurately. Moreover, for each prescribed $J_{\max}$ or $\epsilon_0$, the active-cell percentage remains
nearly constant with only slight oscillations during the time evolution. This behavior is consistent with the physical
property of the linear advection problem, where the solution profile is transported without changing its structure.
These results confirm that the proposed adaptive scheme can effectively capture and track localized solution features
with a sparse representation.
\begin{figure}[!htbp]
    \centering

    \begin{subfigure}[b]{0.45\textwidth}
        \centering
        \includegraphics[width=\linewidth]{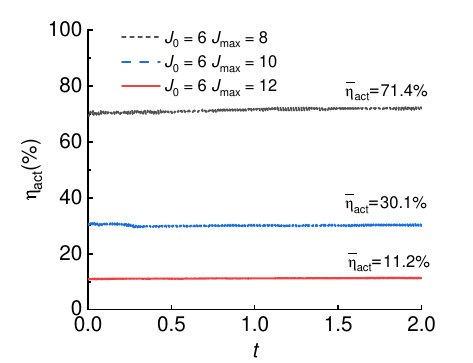}
        \caption{Different finest resolution levels $J_{\max}$, $\epsilon_0=1.0 \times 10^{-5}$.}
        \label{fig:JSdiffJmaxPercentage}
    \end{subfigure}
    \hfill
    \begin{subfigure}[b]{0.45\textwidth}
        \centering
        \includegraphics[width=\linewidth]{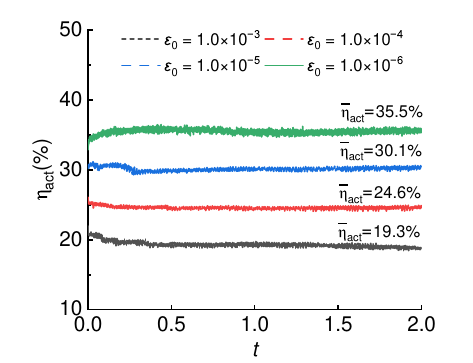}
        \caption{Different thresholding parameters $\epsilon_0$, $J_0=6$, $J_{\max}=10$.}
        \label{fig:JSdiffepsilonPercentage}
    \end{subfigure}

    \caption{Active-cell percentage of the adaptive $\mathrm{AMAIW_4T_2\mbox{-}BVD}$ scheme for the Jiang and Shu's
    problem with different $J_{\max}$ and $\epsilon_0$.}
    \label{fig:JSactivePercentage}
\end{figure}

Finally, we verify the conservation property of the proposed adaptive scheme for solving problems containing
discontinuities. The conservation errors of the adaptive $\mathrm{AMAIW_4T_2\mbox{-}BVD}$ scheme are compared with those
of the corresponding uniform-cell $\mathrm{AIW_4T_2\mbox{-}BVD}$ scheme in Fig.~\ref{fig:JSconservationError}. As shown
in Fig.~\ref{fig:JSdiffJmaxConE}, for different finest resolution levels $J_{\max}$, the conservation errors of the
adaptive scheme agree well with those of the uniform-cell scheme at the corresponding resolution level $J=J_{\max}$.
More importantly, the conservation errors obtained at different resolution levels remain very close to each other and
exhibit nearly the same temporal evolution. The conservation errors of both the adaptive and uniform-cell schemes at
different resolution levels are of the same small magnitude, close to machine precision, indicating that neither the
increase in resolution nor the adaptive cell redistribution introduces noticeable additional conservation errors. As
shown in Fig.~\ref{fig:JSdiffepsilonConE}, the conservation errors obtained with different $\epsilon_0$ show some
visible differences at the initial stage, mainly due to the different adaptive cell distributions generated by different
tolerances. Nevertheless, these differences rapidly diminish as the computation proceeds, and the conservation errors
subsequently remain almost identical to the corresponding uniform-cell result. This indicates that the prescribed
thresholding tolerance has no obvious influence on the long-time conservation behavior, and the threshold-based
adaptation process does not introduce additional conservation errors. Together with the theoretical conservation
analysis presented previously, these numerical results demonstrate the strict conservativity of the proposed adaptive
wavelet upwind scheme for compressible flows involving discontinuities.
\begin{figure}[!htbp]
    \centering

    \begin{subfigure}[b]{0.45\textwidth}
        \centering
        \includegraphics[width=\linewidth]{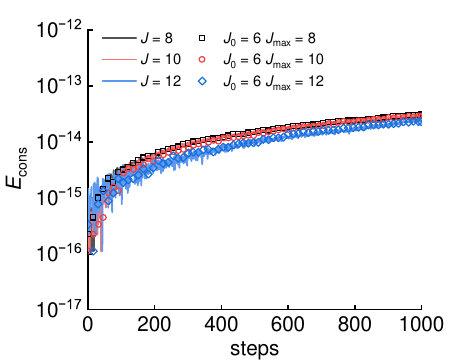}
        \caption{Different finest resolution levels $J_{\max}$, $\epsilon_0=1.0 \times 10^{-5}$.}
        \label{fig:JSdiffJmaxConE}
    \end{subfigure}
    \hfill
    \begin{subfigure}[b]{0.45\textwidth}
        \centering
        \includegraphics[width=\linewidth]{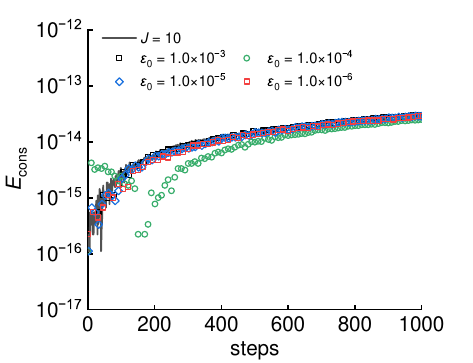}
        \caption{Different thresholding parameters $\epsilon_0$, $J_0=6$, $J_{\max}=10$.}
        \label{fig:JSdiffepsilonConE}
    \end{subfigure}

    \caption{Conservation errors of the $\mathrm{AMAIW_4T_2\mbox{-}BVD}$ scheme for the Jiang and Shu's problem with
    different $J_{\max}$ and $\epsilon_0$.}
    \label{fig:JSconservationError}
\end{figure}
\FloatBarrier

\subsection{One-dimensional nonlinear inviscid Burgers' equation}
Here, the one-dimensional inviscid Burgers' equation is solved with different initial conditions and final times:
\begin{equation}
u_t+\left(\dfrac{u^2}{2}\right)_x=0,\quad 
u(x,0)=u_0(x),
\label{eq:1DBurgersEquation}
\end{equation}
with periodic boundary conditions. The exact solutions of the Burgers' equation is obtained by the semi-analytical
method proposed by Harten et al.~\cite{HARTEN1987}, in which Newton--Raphson iterations are used to solve the
characteristic relation.

\textbf{Case 1.} This case is used to test the order of accuracy of the proposed wavelet upwind scheme for nonlinear
problems. The initial condition is given by
\begin{equation}
u(x,0)=0.5+\sin(\pi x), \quad x \in [0,2].
\label{eq:Case1BurgersInitial}
\end{equation}
The final time is set as $t=0.5/\pi$, which is before the shock formation time $t_s=1/\pi$. 

The numerical results are computed using the uniform-cell AIWU4 and $\mathrm{AIW_4T_2\mbox{-}BVD}$ schemes under a CFL
number of $1/16$. The corresponding numerical errors and orders of accuracy measured in different norms are listed in
Table~\ref{tab:accuracyBurgers}. As the mesh is refined, both the AIWU4 and $\mathrm{AIW_4T_2\mbox{-}BVD}$ schemes
achieve convergence orders close to the theoretical value. This convergence behavior is consistent with that reported
for DG~\cite{Cockburn2003}, MCV~\cite{Ii2009}, and spectral volume~\cite{Wang2002} methods. In addition, the nonlinear
$\mathrm{AIW_4T_2\mbox{-}BVD}$ scheme yields exactly the same numerical errors and orders of accuracy as the linear
AIWU4 scheme. This indicates that, for the smooth solution of the nonlinear Burgers' equation, the BVD reconstruction
also selects the high-order wavelet discretization, thereby preserving the designed accuracy of the underlying linear
scheme.
\begin{table}[!htbp]
\centering
\caption{Numerical errors and orders of accuracy for the one-dimensional Burgers' equation at $t=0.5/\pi$.}
\label{tab:accuracyBurgers}
\begin{tabular*}{0.95\textwidth}{@{\extracolsep{\fill}}cccccc}
\toprule
Schemes & $N_1$ & $L^\infty$ error & $L^\infty$ order & $L^1$ error & $L^1$ order \\
\midrule
\multirow{6}{*}{$\mathrm{AIWU4}$} & 32   & $5.047\times10^{-4}$ & --   & $8.495\times10^{-5}$ & --   \\
& 64   & $4.633\times10^{-5}$ & 3.45 & $6.543\times10^{-6}$ & 3.70 \\
& 128  & $3.467\times10^{-6}$ & 3.74 & $4.881\times10^{-7}$ & 3.74 \\
& 256  & $2.355\times10^{-7}$ & 3.88 & $3.341\times10^{-8}$ & 3.87 \\
& 512  & $1.529\times10^{-8}$ & 3.94 & $2.194\times10^{-9}$ & 3.93 \\
& 1024 & $9.756\times10^{-10}$ & 3.97 & $1.414\times10^{-10}$ & 3.96 \\

\multirow{6}{*}{$\mathrm{AIW_4T_2\mbox{-}BVD}$} & 32   & $5.047\times10^{-4}$ & --   & $8.495\times10^{-5}$ & --   \\
& 64   & $4.633\times10^{-5}$ & 3.45 & $6.543\times10^{-6}$ & 3.70 \\
& 128  & $3.467\times10^{-6}$ & 3.74 & $4.881\times10^{-7}$ & 3.74 \\
& 256  & $2.355\times10^{-7}$ & 3.88 & $3.341\times10^{-8}$ & 3.87 \\
& 512  & $1.529\times10^{-8}$ & 3.94 & $2.194\times10^{-9}$ & 3.93 \\
& 1024 & $9.756\times10^{-10}$ & 3.97 & $1.414\times10^{-10}$ & 3.96 \\
\bottomrule
\end{tabular*}
\end{table}

\textbf{Case 2.} This case was introduced by Sebastian and Shu~\cite{Sebastian2003} to test the capability of numerical
schemes to capture much stronger shocks. Here, we use this problem to further evaluate the discontinuity-capturing
capability, error-control performance, and conservation property of the proposed adaptive scheme. The initial condition
is given by
\begin{equation}
u(x,0)=17.4+13.3\sin(\pi x), \quad x \in [-1,1].
\label{eq:Case2BurgersInitial}
\end{equation}

For this case, the shock formation time is $t_s=1/(13.3\pi)$. The numerical solutions are computed at $t=0.1/\pi$, which
is after shock formation, to examine the shock-capturing capability and adaptive performance of the proposed adaptive
wavelet upwind scheme. In this case, $n_{r0}=6$ is used in the $\mathrm{ACG}\text{-}J$. The numerical results at
$t=0.1/\pi$ obtained by the proposed adaptive $\mathrm{AMAIW_4T_2\mbox{-}BVD}$ scheme with different finest resolution
levels are shown in Fig.~\ref{fig:BurgersdiffJmax}. The proposed adaptive scheme resolves this strong shock problem
stably without spurious oscillations. As the finest resolution level increases, the discontinuity is captured more
sharply and the numerical shock approaches the exact shock location, as further confirmed by the zoomed view near the
shock. Meanwhile, the active-cell number increases from $N_1=90$ for $J_{\max}=6$ to $N_1=198$ for $J_{\max}=10$, which
is much smaller than the number of cells on the corresponding finest uniform mesh. These results demonstrate that the
proposed adaptive multiresolution wavelet upwind scheme can effectively capture strong shocks through sparse adaptive
representations.
\begin{figure}[!htbp]
    \centering

    \begin{subfigure}[b]{0.45\textwidth}
        \centering
        \includegraphics[width=\linewidth]{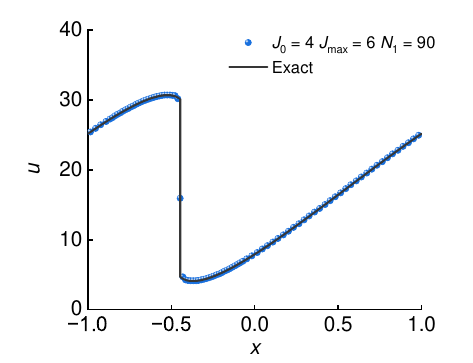}
        \caption{$J_{\max}=6$.}
        \label{fig:BurgersdiffJmax6}
    \end{subfigure}
    \hfill
    \begin{subfigure}[b]{0.45\textwidth}
        \centering
        \includegraphics[width=\linewidth]{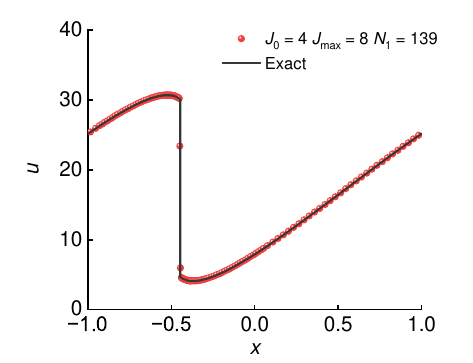}
        \caption{$J_{\max}=8$.}
        \label{fig:BurgersdiffJmax8}
    \end{subfigure}

    \vspace{0.3cm}

    \begin{subfigure}[b]{0.45\textwidth}
        \centering
        \includegraphics[width=\linewidth]{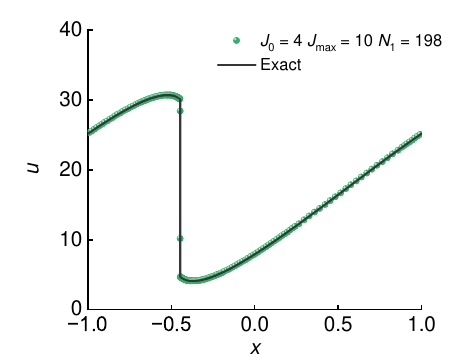}
        \caption{$J_{\max}=10$.}
        \label{fig:BurgersdiffJmax10}
    \end{subfigure}
    \hfill
    \begin{subfigure}[b]{0.45\textwidth}
        \centering
        \includegraphics[width=\linewidth]{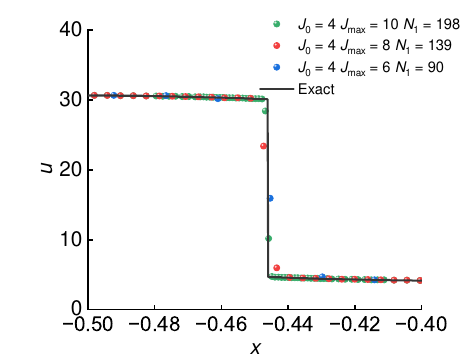}
        \caption{Zoomed view near the shock.}
        \label{fig:BurgersdiffJmaxZoom}
    \end{subfigure}

    \caption{Numerical results at $t=0.1/\pi$ obtained by the proposed adaptive $\mathrm{AMAIW_4T_2\mbox{-}BVD}$ scheme
    for the Burgers' equation with different finest resolution levels, $J_0=4$, and $\epsilon_0=1.0\times10^{-5}$.}
    \label{fig:BurgersdiffJmax}
\end{figure}

The numerical errors at $t=0.1/\pi$ and variations in active-cell percentages with time obtained by the adaptive
$\mathrm{AMAIW_4T_2\mbox{-}BVD}$ scheme under different thresholding tolerances are shown in
Fig.~\ref{fig:Burgerserrorcontrol}. As shown in Fig.~\ref{fig:Burgerserrorcontrol_error}, the errors of the proposed
adaptive scheme generally lie between those of the uniform-cell $\mathrm{AIW_4T_2\mbox{-}BVD}$ schemes at $J=J_0$ and
$J=J_{\max}$. Except in the immediate neighborhood of the shock, the numerical errors are effectively controlled around
the prescribed thresholding tolerances. Figure~\ref{fig:Burgerserrorcontrol_activepercentage} shows that smaller
thresholding tolerances lead to larger active-cell percentages, since more cells are retained in regions with strong
local variations. The active-cell percentage reaches a local peak around the shock formation time, because the rapidly
increasing solution gradient produces large wavelet coefficients over a relatively wide region. After the shock is
formed, the discontinuity becomes highly localized and the solution away from the shock becomes smoother, so the region
requiring finest-level resolution becomes narrower and the active-cell percentage decreases. These results demonstrate
the effectiveness of the threshold-based error-control strategy and the dynamic adaptivity procedure for nonlinear
problems.
\begin{figure}[!htbp]
    \centering

    \begin{subfigure}[b]{0.45\textwidth}
        \centering
        \includegraphics[width=\linewidth]{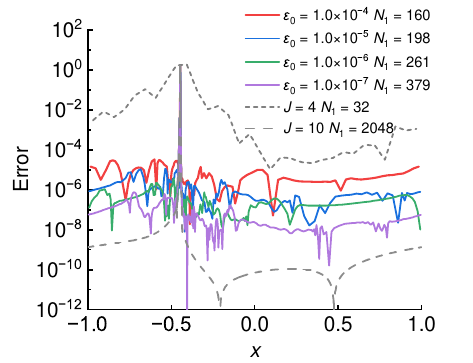}
        \caption{Numerical errors.}
        \label{fig:Burgerserrorcontrol_error}
    \end{subfigure}
    \hfill
    \begin{subfigure}[b]{0.45\textwidth}
        \centering
        \includegraphics[width=\linewidth]{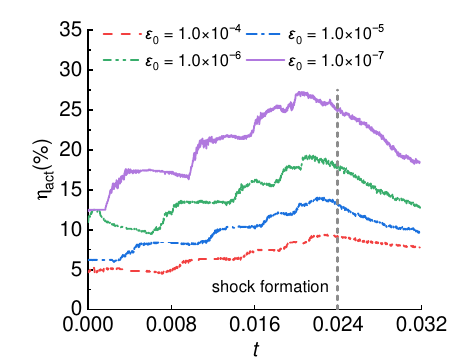}
        \caption{Active-cell percentage.}
        \label{fig:Burgerserrorcontrol_activepercentage}
    \end{subfigure}
    \caption{Numerical errors at $t=0.1/\pi$ and variations in active-cell percentages with time obtained by the
adaptive $\mathrm{AMAIW_4T_2\mbox{-}BVD}$ scheme for the Burgers' equation with different thresholding parameters,
$J_0=4$, and $J_{\max}=10$.}
    \label{fig:Burgerserrorcontrol}
\end{figure}

Next, we conduct a long-time simulation up to $t=24$ to investigate the stability, data-compression capability, and
conservation property of the proposed adaptive $\mathrm{AMAIW_4T_2\mbox{-}BVD}$ scheme. The numerical solutions at
$t=24$ with different finest resolution levels and the corresponding temporal evolutions of the active-cell percentages
are presented in Fig.~\ref{fig:BurgersdiffJmax24}. It can be observed that the proposed adaptive wavelet upwind scheme
can stably solve the nonlinear Burgers' equation with a strong shock during the long-time evolution. The proposed
adaptive scheme accurately captures the shock wave at its correct position without spurious oscillations. As the finest
resolution level increases, the discontinuity is captured more sharply, and the numerical solution obtained with
$J_{\max}=10$ agrees very well with the exact solution. Figure~\ref{fig:Burgersactive24} shows that the active-cell
percentage reaches a local peak around the shock formation time $t_s=1/(13.3\pi)\approx 0.02394$, and then decreases to
a relatively stable value during the subsequent long-time evolution. Before and around shock formation, the solution
gradient increases rapidly over a relatively wide region, leading to large wavelet coefficients and hence more activated
cells. After the shock is formed, the discontinuity becomes localized, while the smooth component away from the shock is
gradually stretched and becomes relatively smoother. During the subsequent long-time evolution, the local regularity of
the smooth component remains nearly unchanged and the shock strength gradually decreases. This leads to the decrease and
subsequent stabilization of the active-cell percentage. We can see that the data-compression advantage becomes more
pronounced at higher finest resolution levels. These results demonstrate that the proposed adaptive multiresolution
wavelet upwind scheme maintains good stability and accuracy for long-time shock evolution, while achieving effective
data compression through adaptive cell redistribution.
\begin{figure}
    \centering

    \begin{subfigure}[b]{0.45\textwidth}
        \centering
        \includegraphics[width=\linewidth]{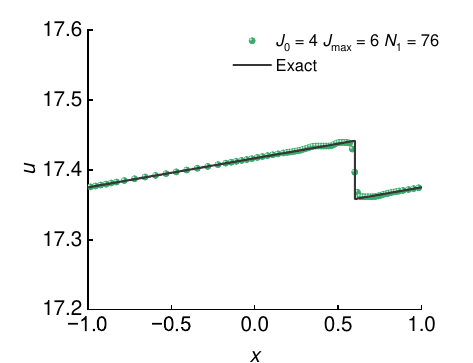}
        \caption{$J_{\max}=6$.}
        \label{fig:BurgersdiffJmax6_24}
    \end{subfigure}
    \hfill
    \begin{subfigure}[b]{0.45\textwidth}
        \centering
        \includegraphics[width=\linewidth]{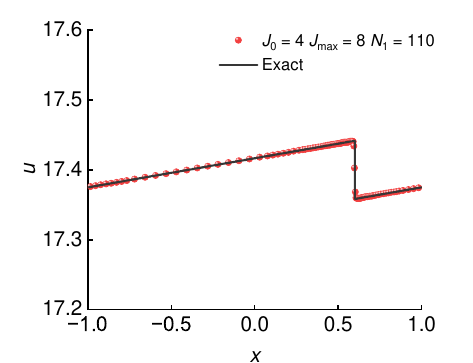}
        \caption{$J_{\max}=8$.}
        \label{fig:BurgersdiffJmax8_24}
    \end{subfigure}

    \vspace{0.3cm}

    \begin{subfigure}[b]{0.45\textwidth}
        \centering
        \includegraphics[width=\linewidth]{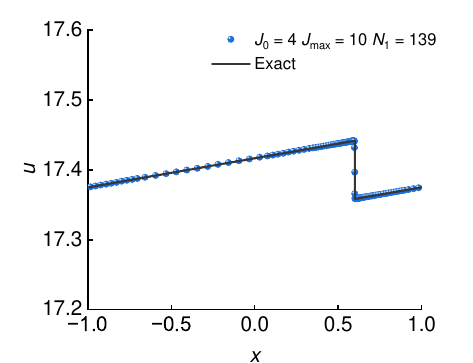}
        \caption{$J_{\max}=10$.}
        \label{fig:BurgersdiffJmax10_24}
    \end{subfigure}
    \hfill
    \begin{subfigure}[b]{0.45\textwidth}
        \centering
        \includegraphics[width=\linewidth]{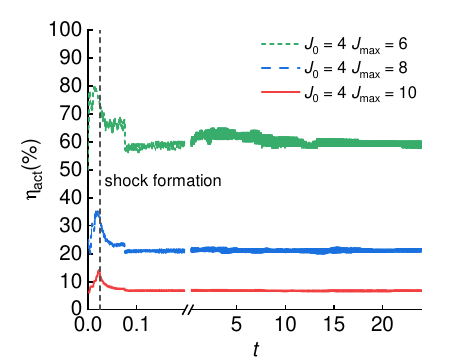}
        \caption{Active-cell percentage.}
        \label{fig:Burgersactive24}
    \end{subfigure}

    \caption{Numerical solutions at $t=24$ for different finest resolution levels and the corresponding temporal
evolutions of the active-cell percentages obtained by the adaptive $\mathrm{AMAIW_4T_2\mbox{-}BVD}$ scheme for the
Burgers' equation with $\epsilon_0=1.0\times10^{-5}$.}
    \label{fig:BurgersdiffJmax24}
\end{figure}

Finally, the conservation errors up to $4.0 \times10^4$ time steps obtained by the proposed adaptive wavelet upwind
scheme are compared with those of the corresponding uniform-cell scheme in Fig.~\ref{fig:Burgersconerror24}. For
different finest resolution levels, the conservation errors of the adaptive scheme almost overlap with those of the
corresponding uniform-cell scheme. Moreover, the conservation errors at different resolution levels are nearly identical
to each other and show the same temporal evolution. This indicates that, in the proposed conservative adaptive scheme,
neither the variation in resolution levels nor the adaptive cell redistribution process introduces noticeable additional
conservation errors during long-time computations. Even after such a large number of time steps, the conservation errors
remain on the order of $10^{-11}$. In contrast, the conservation errors reported for the nonconservative schemes in
Ref.~\cite{Sebastian2003} are several orders of magnitude larger, ranging from about $10^{-4}$ to $10^{-2}$ depending on
the scheme and resolution. These results verify the conservation property of the proposed adaptive scheme for nonlinear
problems with strong discontinuities.
\begin{figure}
    \centering
    \includegraphics[width=0.45\textwidth]{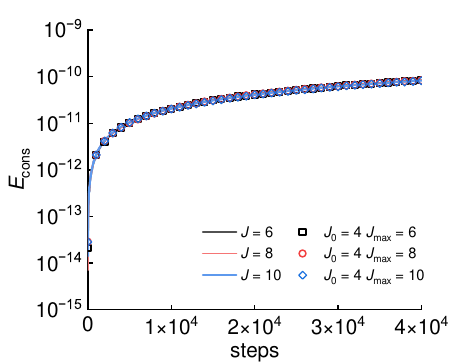}
    \caption{Conservation errors obtained by the adaptive $\mathrm{AMAIW_4T_2\mbox{-}BVD}$ scheme and the corresponding
uniform-cell scheme for the Burgers' equation with different finest resolution levels up to $4.0 \times10^4$ time steps,
$\epsilon_0=1.0\times10^{-5}$.}
    \label{fig:Burgersconerror24}
\end{figure}
\FloatBarrier

\subsection{One-dimensional Euler equations}
The one-dimensional Euler equations of gas dynamics for a polytropic gas in the conservative form are written as 
\begin{equation}
\begin{aligned}
&\mathbf{U}_t+\mathbf{F}(\mathbf{U})_x =0,\\
&\mathbf{U} = (\rho,\rho u,E)^T,\\
&\mathbf{F}(\mathbf{U}) = \left(\rho u,\rho u^2 + p,u(E+p)\right)^T .
\end{aligned}
\end{equation}
where $\rho, u, p$, and $E$ are the density, velocity, pressure and total energy, respectively. For the perfect gas, the
total energy can be determined by the following equation:
\begin{equation}
E=\dfrac{p}{\gamma-1}+\dfrac{1}{2} \rho u^2,
\label{eq:state_equation}
\end{equation}
where $\gamma=1.4$.  
\subsubsection{Advection of density perturbation}
To verify the order of accuracy of the present average-interpolating wavelet upwind schemes, we consider the advection
of density perturbation problem. The initial condition is given by $\rho(x,0)=1+0.2\sin(\pi x)$, $u(x,0)=1$, and
$p(x,0)=1$ on the computational domain $[0,2]$ with periodic boundary conditions.

The numerical solutions at $t=2.0$ are obtained using the uniform-cell AIWU4 and $\mathrm{AIW_4T_2\mbox{-}BVD}$ schemes
with a CFL number of $1/16$. The corresponding numerical errors and convergence orders in different norms are reported
in Table~\ref{tab:accuracy_density_perturbation}. It can be observed that both wavelet schemes achieve the expected
fourth-order accuracy. Moreover, the $\mathrm{AIW_4T_2\mbox{-}BVD}$ scheme gives identical numerical errors and
convergence orders to those of the AIWU4 scheme. This indicates that, for smooth solutions, the BVD algorithm
consistently selects the high-order wavelet candidate, thereby preserving the designed high-order accuracy.

Next, we evaluate the error-control capability of the proposed adaptive multiresolution average-interpolating wavelet
upwind scheme for nonlinear systems. The numerical solutions at $t=2.0$ are obtained using the adaptive
$\mathrm{AMAIW_4T_2\mbox{-}BVD}$ scheme with various thresholding parameters. The corresponding numbers of active leaf
cells and numerical errors are listed in Table~\ref{tab:accuracy_epsilon_euler}. The numerical errors are effectively
controlled in accordance with the prescribed tolerances. These results demonstrate the feasibility of the proposed
adaptive error-control strategy based on wavelet multiresolution analysis for solving nonlinear systems.
\begin{table}[!htbp]
\centering
\caption{Numerical errors and orders of accuracy for the one-dimensional Euler equations.}
\label{tab:accuracy_density_perturbation}
\begin{tabular*}{0.95\textwidth}{@{\extracolsep{\fill}}cccccc}
\toprule
Schemes & $N_1$ & $L^\infty$ error & $L^\infty$ order & $L^1$ error & $L^1$ order \\
\midrule
\multirow{6}{*}{$\mathrm{AIWU4}$} & 16  & $6.030\times10^{-4}$ & --   & $7.631\times10^{-4}$ & --   \\
& 32  & $3.697\times10^{-5}$ & 4.03 & $4.697\times10^{-5}$ & 4.02 \\
& 64  & $2.297\times10^{-6}$ & 4.01 & $2.923\times10^{-6}$ & 4.01 \\
& 128 & $1.434\times10^{-7}$ & 4.00 & $1.825\times10^{-7}$ & 4.00 \\
& 256 & $8.957\times10^{-9}$ & 4.00 & $1.140\times10^{-8}$ & 4.00 \\
& 512 & $5.609\times10^{-10}$ & 4.00 & $7.127\times10^{-10}$ & 4.00 \\
\multirow{6}{*}{$\mathrm{AIW_4T_2\mbox{-}BVD}$} & 16  & $6.030\times10^{-4}$ & --   & $7.631\times10^{-4}$ & --   \\
& 32  & $3.697\times10^{-5}$ & 4.03 & $4.697\times10^{-5}$ & 4.02 \\
& 64  & $2.297\times10^{-6}$ & 4.01 & $2.923\times10^{-6}$ & 4.01 \\
& 128 & $1.434\times10^{-7}$ & 4.00 & $1.825\times10^{-7}$ & 4.00 \\
& 256 & $8.957\times10^{-9}$ & 4.00 & $1.140\times10^{-8}$ & 4.00 \\
& 512 & $5.609\times10^{-10}$ & 4.00 & $7.127\times10^{-10}$ & 4.00 \\
\bottomrule
\end{tabular*}
\end{table}

\begin{table}[!htbp]
\centering
\caption{Error-control tests with $J_0=4$, $J_{\max}=8$ under different thresholding parameters for the advection of
density perturbation of the one-dimensional Euler equations.}
\label{tab:accuracy_epsilon_euler}
\begin{tabular*}{0.95\textwidth}{@{\extracolsep{\fill}}cccc}
\toprule
$\epsilon_0$ & $N_1$ & $L^\infty$ error & $L^1$ error \\
\midrule
$1.000\times10^{-5}$  & 32  & $3.696\times10^{-5}$ & $4.697\times10^{-5}$ \\
$1.000\times10^{-6}$  & 64  & $2.297\times10^{-6}$ & $2.923\times10^{-6}$ \\
$1.000\times10^{-7}$  & 128 & $5.424\times10^{-7}$ & $2.939\times10^{-7}$ \\
$1.000\times10^{-8}$  & 128 & $1.434\times10^{-7}$ & $1.825\times10^{-7}$ \\
$1.000\times10^{-9}$  & 256 & $1.835\times10^{-8}$ & $1.394\times10^{-8}$ \\
$1.000\times10^{-10}$ & 512 & $2.814\times10^{-9}$ & $1.719\times10^{-9}$ \\
\bottomrule
\end{tabular*}
\end{table}

\subsubsection{Sod's problem}
This Riemann problem is a well-known benchmark test proposed by Sod~\cite{Sod1978} to test the performance of numerical
schemes in solving Euler equations, which contains different types of discontinuities, such as shock wave and contact
discontinuity. The initial condition of the Sod's problem is expressed as
\begin{equation}
(\rho,u,p)=
\begin{cases}
(1,0,1), & 0 \leq x \leq 0.5,\\
(0.125,0,0.1), & 0.5 < x \leq 1 .
\end{cases}
\end{equation}
The final simulation time is $t=0.2$. 

We first solve Sod's problem using the adaptive $\mathrm{AMAIW_4T_2\mbox{-}BVD}$ scheme and compare the results with
those obtained by the corresponding uniform-cell $\mathrm{AIW_4T_2\mbox{-}BVD}$ scheme at $J=J_{\max}$. The computed
density, velocity, and pressure are compared in Fig.~\ref{fig:sodcompareJ10}. To further illustrate the resolution
improvement introduced by the adaptive refinement, the coarse uniform-cell result of the $\mathrm{AIW_4T_2\mbox{-}BVD}$
scheme at $J=8$ is also included in the zoomed density profile near the discontinuities. It can be observed that the
numerical solutions obtained by the adaptive scheme are in excellent agreement with those of the corresponding finest
uniform-cell scheme at $J=J_{\max}$. The shock wave, contact discontinuity, and rarefaction wave are sharply resolved
without noticeable spurious oscillations. As shown in Fig.~\ref{fig:sodcompareJ10rhozoomin}, the coarse uniform-cell
result at $J=8$ suffers from visible numerical smearing near the discontinuities, while the adaptive scheme locally
refines the cells in these regions and yields much sharper transitions. The adaptive result almost overlaps with the
finest uniform-cell result at $J=10$, demonstrating that the proposed adaptive scheme can achieve the local resolution
of the finest uniform-cell computation without applying uniform refinement to the entire domain. These results verify
that the proposed adaptive scheme retains the shock-capturing capability of the finest uniform-cell scheme and that the
BVD reconstruction on the finest-level leaf cells is effective.
\begin{figure}
    \centering

    \begin{subfigure}[b]{0.45\textwidth}
        \centering
        \includegraphics[width=\linewidth]{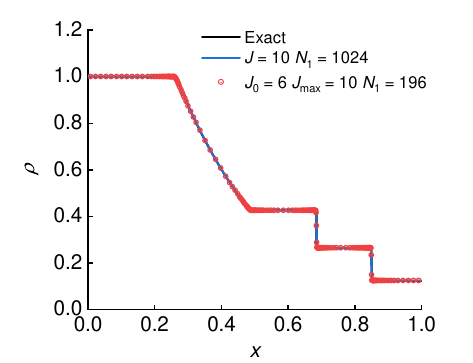}
        \caption{Density.}
        \label{fig:sodcompareJ10rho}
    \end{subfigure}
    \hfill
    \begin{subfigure}[b]{0.45\textwidth}
        \centering
        \includegraphics[width=\linewidth]{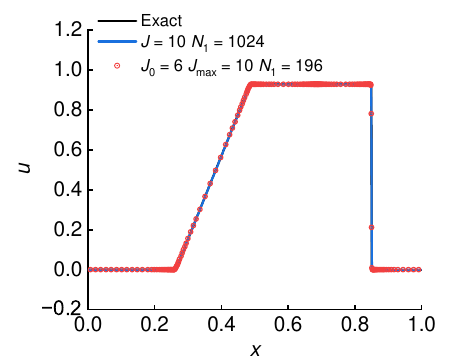}
        \caption{Velocity.}
        \label{fig:sodcompareJ10u}
    \end{subfigure}

    \vspace{0.3cm}

    \begin{subfigure}[b]{0.45\textwidth}
        \centering
        \includegraphics[width=\linewidth]{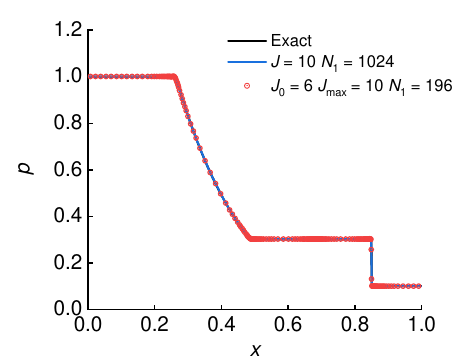}
        \caption{Pressure.}
        \label{fig:sodcompareJ10p}
    \end{subfigure}
    \hfill
    \begin{subfigure}[b]{0.45\textwidth}
        \centering
        \includegraphics[width=\linewidth]{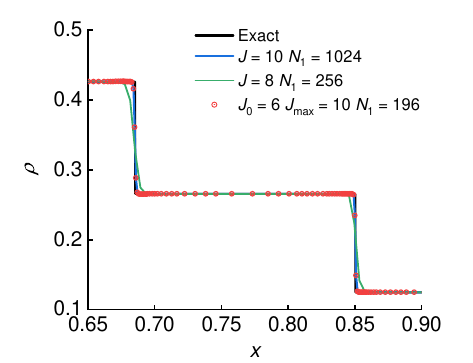}
        \caption{Zoomed density profile near the discontinuities.}
        \label{fig:sodcompareJ10rhozoomin}
    \end{subfigure}
    
    \caption{Comparison of the density, velocity, pressure, and zoomed density profile near the discontinuities at
    $t=0.2$ obtained by the $\mathrm{AMAIW_4T_2\mbox{-}BVD}$ scheme ($\epsilon_0=1.0 \times 10^{-4}$), and the
    $\mathrm{AIW_4T_2\mbox{-}BVD}$ scheme for Sod's problem.}
    \label{fig:sodcompareJ10}
\end{figure}

Next, the numerical results at $t=0.2$ are computed using the proposed adaptive wavelet upwind scheme with different
finest resolution levels. The corresponding density profiles are shown in
Figs.~\ref{fig:soddiffJmax8}--\ref{fig:soddiffJmax12}, and a zoomed-in comparison near the contact discontinuity is
presented in Fig.~\ref{fig:soddiffJmaxZoom}. The proposed adaptive $\mathrm{AMAIW_4T_2\mbox{-}BVD}$ scheme sharply
captures the shock wave and contact discontinuity without visible numerical oscillations for all considered values of
$J_{\max}$. As $J_{\max}$ increases, the computed contact discontinuity gradually approaches the exact solution, as
shown in Fig.~\ref{fig:soddiffJmaxZoom}. These results demonstrate that the proposed adaptive multiresolution wavelet
upwind scheme can accurately and efficiently capture discontinuities in nonlinear Euler systems.
\begin{figure}
    \centering

    \begin{subfigure}[b]{0.45\textwidth}
        \centering
        \includegraphics[width=\linewidth]{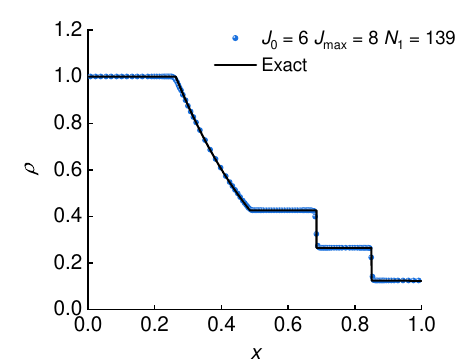}
        \caption{$J_{\max}=8$.}
        \label{fig:soddiffJmax8}
    \end{subfigure}
    \hfill
    \begin{subfigure}[b]{0.45\textwidth}
        \centering
        \includegraphics[width=\linewidth]{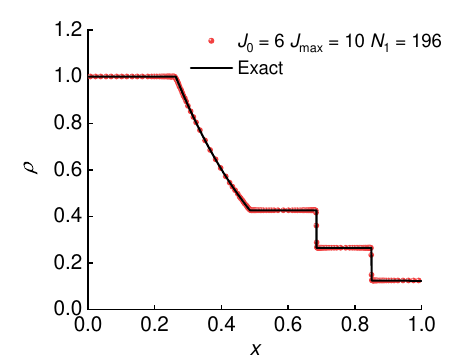}
        \caption{$J_{\max}=10$.}
        \label{fig:soddiffJmax10}
    \end{subfigure}

    \vspace{0.3cm}

    \begin{subfigure}[b]{0.45\textwidth}
        \centering
        \includegraphics[width=\linewidth]{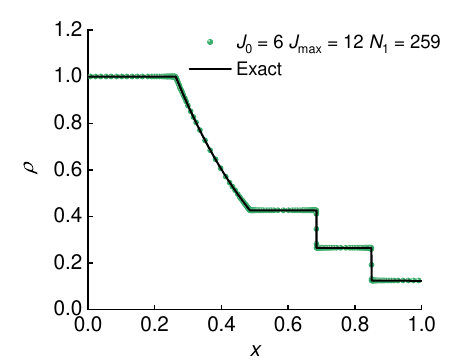}
        \caption{$J_{\max}=12$.}
        \label{fig:soddiffJmax12}
    \end{subfigure}
    \hfill
    \begin{subfigure}[b]{0.45\textwidth}
        \centering
        \includegraphics[width=\linewidth]{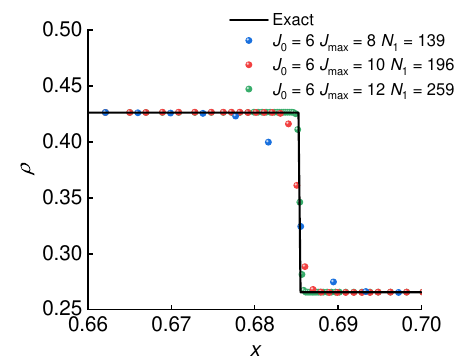}
        \caption{Zoomed-in view of the contact discontinuity.}
        \label{fig:soddiffJmaxZoom}
    \end{subfigure}
    
    \caption{Numerical solutions at $t=0.2$ of Sod's problem obtained by the adaptive $\mathrm{AMAIW_4T_2\mbox{-}BVD}$
    scheme with different finest resolution levels and $\epsilon_0=1.0 \times 10^{-4}$. The last panel shows a zoomed-in
    view near the contact discontinuity.}
    \label{fig:soddiffJmax}
\end{figure}

To evaluate the adaptive performance of the proposed adaptive multiresolution wavelet upwind scheme, the temporal
evolutions of the active-cell percentages are shown in Fig.~\ref{fig:soddiffJmaxpercentage}. The active leaf cell
distribution for $J_{\max}=12$ at $t=0.2$ is also presented in Fig.~\ref{fig:sodJm12activecell}. It can be observed that
the active-cell percentage increases during the early stage of the wave evolution and then reaches a nearly steady value
after the shock wave, contact discontinuity, and rarefaction wave are formed. The active-cell percentage decreases as
the finest resolution level increases. In particular, for $J_{\max}=12$, only about $6\%$ of the corresponding uniform
cells are activated. As shown in Fig.~\ref{fig:sodJm12activecell}, the finest-level leaf cells are mainly concentrated
in the vicinity of the strong discontinuities to resolve discontinuities with sharp transitions. In contrast, relatively
lower-level cells are sufficient to resolve the rarefaction wave. These results indicate that the proposed adaptive
scheme can efficiently identify, locate, and capture different flow features in the nonlinear Euler system through a
sparse representation of the solution.
\begin{figure}
    \centering

    \begin{subfigure}[b]{0.45\textwidth}
        \centering
        \includegraphics[width=\linewidth]{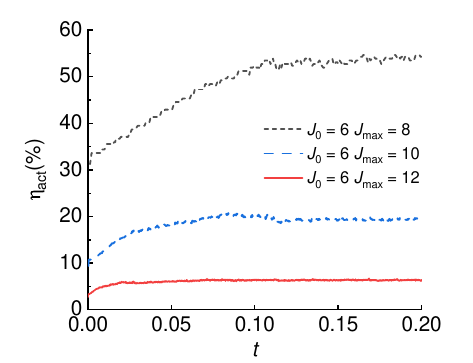}
        \caption{Active-cell percentage.}
        \label{fig:soddiffJmaxpercentage}
    \end{subfigure}
    \hfill
    \begin{subfigure}[b]{0.45\textwidth}
        \centering
        \includegraphics[width=\linewidth]{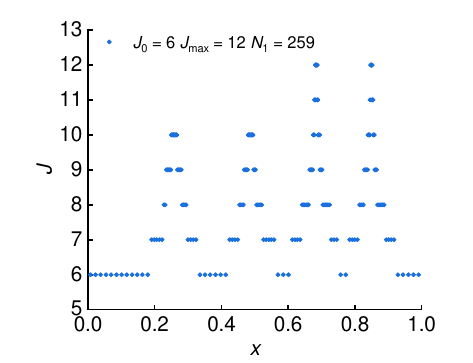}
        \caption{Active leaf cell distribution for $J_{\max}=12$ at $t=0.2$.}
        \label{fig:sodJm12activecell}
    \end{subfigure}
    
    \caption{Temporal evolutions of the active-cell percentages for different finest resolution levels and the active
    leaf cell distribution at $t=0.2$ with $J_{\max}=12$ obtained by the adaptive $\mathrm{AMAIW_4T_2\mbox{-}BVD}$
    scheme for Sod's problem with $\epsilon_0=1.0\times10^{-4}$.}
    \label{fig:sodactivecell}
\end{figure}

Finally, we verify the conservation property of the proposed adaptive average-interpolating wavelet upwind scheme for
the nonlinear Euler equations. Figure~\ref{fig:sodconservationError} compares the conservation errors of mass, momentum,
and total energy up to 1000 time steps obtained by the adaptive $\mathrm{AMAIW_4T_2\mbox{-}BVD}$ scheme with those of
the uniform-cell $\mathrm{AIW_4T_2\mbox{-}BVD}$ scheme at $J=J_{\max}$. It can be observed that the conservation errors
of mass, momentum, and total energy produced by the adaptive scheme are comparable to those of the corresponding
uniform-cell scheme, remaining below $10^{-13}$ and thus close to machine precision. These results confirm the strict
conservation property of the proposed adaptive wavelet upwind scheme for nonlinear Euler systems.
\begin{figure}
    \centering

    \begin{subfigure}[b]{0.45\textwidth}
        \centering
        \includegraphics[width=\linewidth]{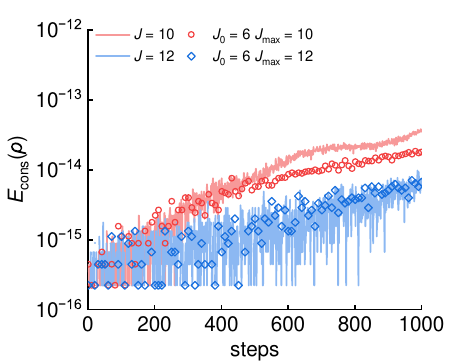}
        \caption{Mass.}
        \label{fig:sodconEmass}
    \end{subfigure}
    \hfill
    \begin{subfigure}[b]{0.45\textwidth}
        \centering
        \includegraphics[width=\linewidth]{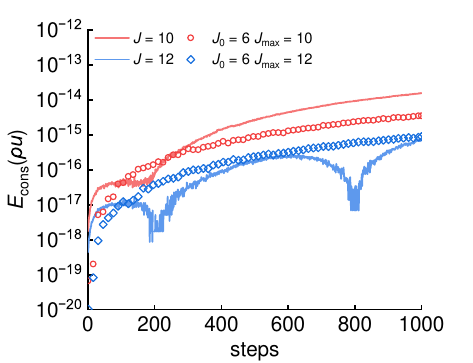}
        \caption{Momentum.}
        \label{fig:sodconEmomentum}
    \end{subfigure}

    \vspace{0.3cm}

    \begin{subfigure}[b]{0.45\textwidth}
        \centering
        \includegraphics[width=\linewidth]{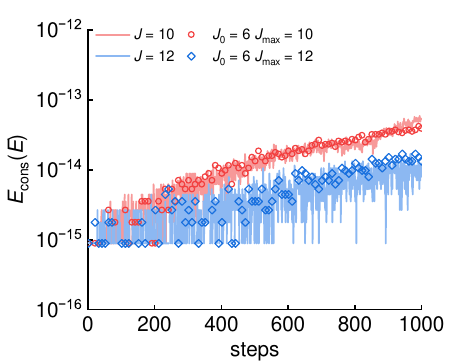}
        \caption{Total energy.}
        \label{fig:sodconEenergy}
    \end{subfigure}
    
    \caption{Conservation errors of mass, momentum, and total energy up to 1000 time steps for Sod's problem. The
adaptive $\mathrm{AMAIW_4T_2\mbox{-}BVD}$ scheme with $\epsilon_0=1.0\times10^{-4}$, and different finest resolution
levels is compared with the corresponding uniform-cell $\mathrm{AIW_4T_2\mbox{-}BVD}$ scheme at $J=J_{\max}$.}
    \label{fig:sodconservationError}
\end{figure}
\FloatBarrier

\subsubsection{Lax's problem}
To further evaluate the ability of the proposed adaptive wavelet upwind scheme to capture stronger shock, the Lax's
problem is solved with the following initial condition:
\begin{equation}
(\rho_0,u_0,p_0)=
\begin{cases}
(0.445,0.698,3.528), & 0 \leq x \leq 0.5,\\
(0.5,0.0,0.571), & 0.5 < x \leq 1 .
\end{cases}
\end{equation}

The numerical solutions at $t=0.13$ are computed using the proposed adaptive $\mathrm{AMAIW_4T_2\mbox{-}BVD}$ scheme
with different finest resolution levels. For comparison, the WENO5-Z scheme~\cite{Borges2008,Don2013} is also employed,
where the number of uniform cells is adjusted to obtain a shock resolution comparable to that of the proposed adaptive
wavelet upwind scheme. The computed density profiles and zoomed-in views near the discontinuities are shown in
Fig.~\ref{fig:laxdiffJmax}. It can be observed that the proposed adaptive $\mathrm{AMAIW_4T_2\mbox{-}BVD}$ scheme
produces satisfactory non-oscillatory solutions for all considered values of $J_{\max}$. As shown in the zoomed-in
views, the contact discontinuity and shock wave are resolved more sharply as the finest resolution level increases, and
the computed discontinuities gradually approach the exact solution without noticeable overshoots or spurious
oscillations.

Compared with the WENO5-Z scheme, the proposed adaptive wavelet scheme requires significantly fewer active leaf cells to
achieve comparable shock resolution. Specifically, the adaptive scheme uses only $N_1=121$, $198$, and $256$ active leaf
cells for $J_{\max}=8$, $10$, and $12$, respectively, whereas the WENO5-Z scheme uses $400$, $2000$, and $4000$ uniform
cells in the corresponding comparisons. Under comparable shock resolution, the adaptive wavelet upwind scheme captures
sharper contact discontinuities than the WENO5-Z scheme, demonstrating its strong capability in resolving contact
discontinuities.
\begin{figure}[!htbp]
    \centering

    \begin{subfigure}[b]{0.45\textwidth}
        \centering
        \includegraphics[width=\linewidth]{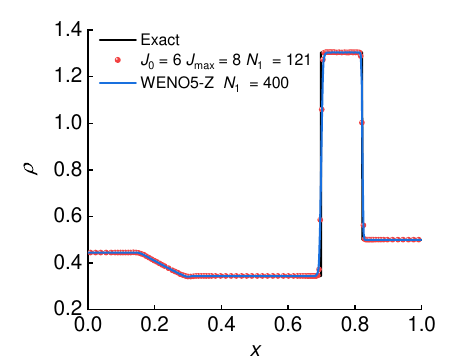}
        \caption{$J_{\max}=8$.}
        \label{fig:laxdiffJmax8}
    \end{subfigure}
    \hfill
    \begin{subfigure}[b]{0.45\textwidth}
        \centering
        \includegraphics[width=\linewidth]{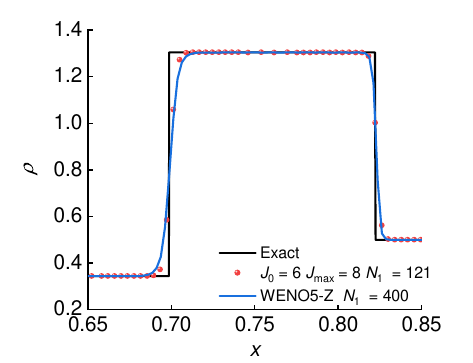}
        \caption{Zoomed-in view for $J_{\max}=8$.}
        \label{fig:laxdiffJmax8local}
    \end{subfigure}

    \vspace{0.3cm}

    \begin{subfigure}[b]{0.45\textwidth}
        \centering
        \includegraphics[width=\linewidth]{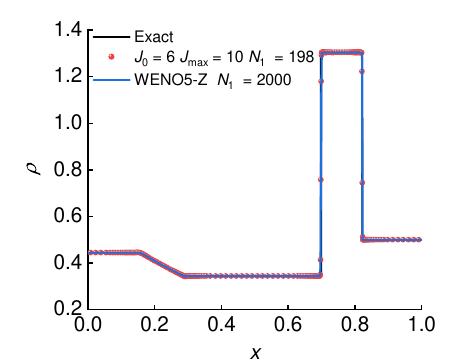}
        \caption{$J_{\max}=10$.}
        \label{fig:laxdiffJmax10}
    \end{subfigure}
    \hfill
    \begin{subfigure}[b]{0.45\textwidth}
        \centering
        \includegraphics[width=\linewidth]{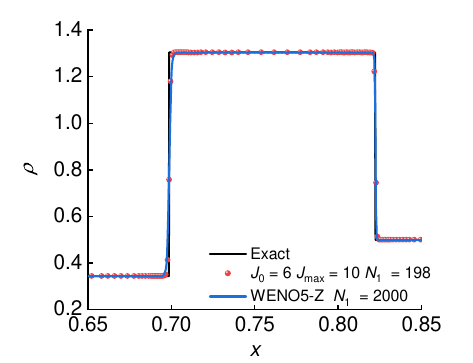}
        \caption{Zoomed-in view for $J_{\max}=10$.}
        \label{fig:laxdiffJmax10local}
    \end{subfigure}

    \vspace{0.3cm}

    \begin{subfigure}[b]{0.45\textwidth}
        \centering
        \includegraphics[width=\linewidth]{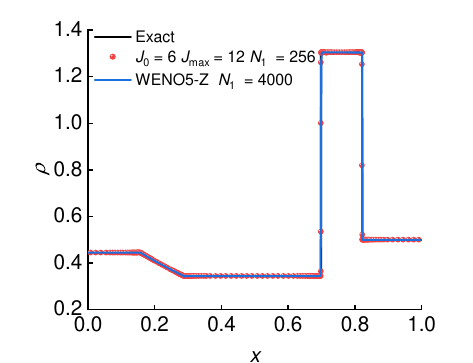}
        \caption{$J_{\max}=12$.}
        \label{fig:laxdiffJmax12}
    \end{subfigure}
    \hfill
    \begin{subfigure}[b]{0.45\textwidth}
        \centering
        \includegraphics[width=\linewidth]{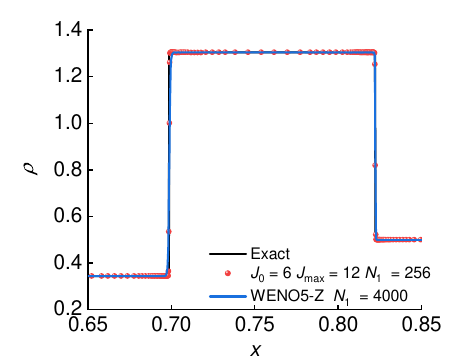}
        \caption{Zoomed-in view for $J_{\max}=12$.}
        \label{fig:laxdiffJmax12Local}
    \end{subfigure}
    
    \caption{Numerical solutions at $t=0.13$ of the Lax's problem obtained by the adaptive
    $\mathrm{AMAIW_4T_2\mbox{-}BVD}$ scheme with $\epsilon_0=1.0\times10^{-4}$ and different finest resolution levels,
    and WENO5-Z scheme. The right panels show the corresponding zoomed-in views near the strong discontinuities.}
\label{fig:laxdiffJmax}
\end{figure}

To illustrate how the multiresolution cells are distributed for different localized features, the active leaf cell
distributions of the proposed adaptive $\mathrm{AMAIW_4T_2\mbox{-}BVD}$ scheme at $t=0.13$ are shown in
Fig.~\ref{fig:laxcelldistribution}. The high-resolution leaf cells are mainly concentrated near the shock wave and the
contact discontinuity to capture strong jump discontinuities, while relatively lower-level cells are activated near the
head and tail of the rarefaction wave. In smooth regions, only coarse cells are retained. These results indicate that
the proposed adaptive scheme can effectively identify, locate, and resolve different flow features through a sparse cell
representation.
\begin{figure}[!htbp]
    \centering

    \begin{subfigure}[b]{0.45\textwidth}
        \centering
        \includegraphics[width=\linewidth]{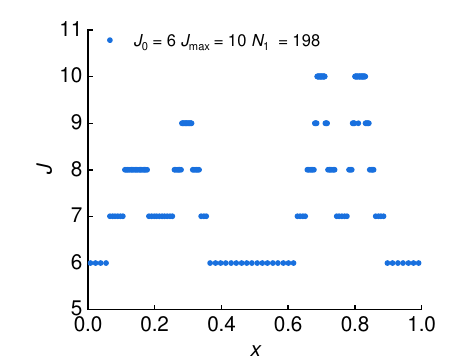}
        \caption{$J_{\max}=10$.}
        \label{fig:laxcelldistributionJmax10}
    \end{subfigure}
    \hfill
    \begin{subfigure}[b]{0.45\textwidth}
        \centering
        \includegraphics[width=\linewidth]{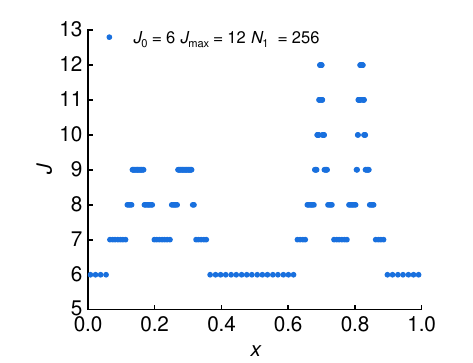}
        \caption{$J_{\max}=12$.}
        \label{fig:laxcelldistributionJmax12}
    \end{subfigure}

    \caption{Active leaf cell distributions at $t=0.13$ obtained by the adaptive $\mathrm{AMAIW_4T_2\mbox{-}BVD}$ scheme
    for the Lax's problem with $\epsilon_0=1.0\times10^{-4}$, and different finest resolution levels.}
    \label{fig:laxcelldistribution}
\end{figure}

To evaluate the computational efficiency of the proposed adaptive wavelet upwind scheme, we compare the CPU time of the
adaptive $\mathrm{AMAIW_4T_2\mbox{-}BVD}$ scheme with those of the corresponding uniform $\mathrm{AIW_4T_2\mbox{-}BVD}$
scheme and WENO5-Z for the Lax problem, as shown in Table~\ref{tab:cpu_time_lax}. As $J_{\max}$ increases, the
active-cell percentage $\overline{\eta}_{\mathrm{act}}$ decreases from $45.5\%$ to $5.9\%$, indicating that the adaptive
multiresolution representation becomes increasingly sparse at higher finest resolution levels. Compared with the uniform
$\mathrm{AIW_4T_2\mbox{-}BVD}$ computation at the same finest resolution level, the adaptive scheme reduces the CPU time
by factors of $1.697$, $3.903$, and $8.054$, respectively. Under comparable shock resolution, the CPU time ratios
$T_{\mathrm{WENO}}/T_{\mathrm{adap}}$ are $8.718$, $25.496$, and $12.739$, respectively. These results demonstrate that
the proposed adaptive wavelet scheme effectively reduces both the computational cost and the number of active degrees of
freedom by exploiting the sparsity of the multiresolution representation, while maintaining high-resolution
shock-capturing capability.
\begin{table}[!htbp]
\centering
\caption{CPU time comparison among the adaptive $\mathrm{AMAIW_4T_2\mbox{-}BVD}$ scheme, the corresponding uniform
$\mathrm{AIW_4T_2\mbox{-}BVD}$ scheme, and WENO5--Z scheme for the Lax's problem.}
\label{tab:cpu_time_lax}
\begin{tabular*}{0.98\textwidth}{@{\extracolsep{\fill}}ccccccc}
\toprule
$(J_0,J_{\max})$ & $\overline{\eta}_{\mathrm{act}}$ & $T_{\mathrm{adap}}/s$ & $T_{\mathrm{unif}}/s$ &
$T_{\mathrm{unif}}/T_{\mathrm{adap}}$ & $T_{\mathrm{WENO}}/s$ & $T_{\mathrm{WENO}}/T_{\mathrm{adap}}$ \\
\midrule
$(6,8)$ & 0.455 & 0.039 & $0.066\;(N_1=256)$ & 1.697 & $0.340\;(N_1=400)$ & 8.718 \\
$(6,10)$ & 0.173 & 0.278 & $1.084\;(N_1=1024)$ & 3.903 & $7.088\;(N_1=2000)$ & 25.496 \\
$(6,12)$ & 0.059 & 2.161 & $17.406\;(N_1=4096)$ & 8.054 & $27.528\;(N_1=4000)$ & 12.739 \\
\bottomrule
\end{tabular*}
\end{table}

\subsubsection{Shock-turbulence interaction}
High-order numerical schemes for compressible flows should be able to resolve both shock waves and complex smooth
structures. Here, we consider the shock--turbulence interaction problem proposed by Shu and Osher~\cite{Shu1989}, which
is widely used to assess the above capability of numerical schemes. In this test, a Mach 3 shock wave interacts with a
density disturbance, generating a flow field that contains both strong discontinuities and highly oscillatory smooth
structures. The initial condition is given by
\begin{equation}
(\rho_0,u_0,p_0)=
\begin{cases}
(3.857148,\,2.629369,\,10.33333), & -5 \leq x \leq -4,\\
(1+0.2\sin 5x,\,0,\,1), & \mathrm{otherwise},
\end{cases}
\qquad -5 \leq x \leq 5 .
\end{equation}
The final simulation time is $t=1.8$.

For this problem, the reference solution is computed by the WENO5-JS scheme~\cite{JiangShu1996} with 10240 uniform
cells. The numerical solutions at $t=1.8$ obtained by the proposed adaptive $\mathrm{AMAIW_4T_2\mbox{-}BVD}$ scheme are
shown in Fig.~\ref{fig:shockTurbulencediffJmax}, and the corresponding active leaf cell distributions are presented in
Fig.~\ref{fig:shockTurbulencecell}. It can be observed from Fig.~\ref{fig:shockTurbulencediffJmax} that the proposed
adaptive wavelet scheme captures the shock wave with a sharp transition while resolving the high-frequency smooth
structures generated behind the shock. As the finest resolution level increases from $J_{\max}=6$ to $J_{\max}=8$, the
oscillatory smooth structures are resolved more accurately and the numerical solution agrees better with the reference
solution.

The corresponding active leaf cell distributions in Fig.~\ref{fig:shockTurbulencecell} further illustrate the adaptive
resolution allocation. The highest-resolution leaf cells are concentrated in narrow regions around the shock wave and
steep transition zones. Meanwhile, relatively high-level cells are also allocated in the oscillatory region behind the
shock to resolve the small-scale smooth structures. These results demonstrate that the proposed adaptive multiresolution
wavelet upwind scheme can effectively identify and resolve both discontinuities and multiscale smooth structures through
a sparse representation of the solution. 
\begin{figure}[!htbp]
    \centering

    \begin{subfigure}[b]{0.45\textwidth}
        \centering
        \includegraphics[width=\linewidth]{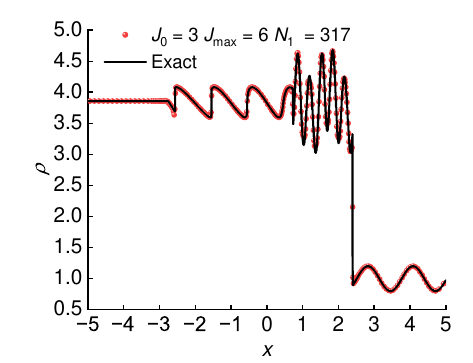}
        \caption{$J_{\max}=6$.}
        \label{fig:shockTurbulencediffJmax6}
    \end{subfigure}
    \hfill
    \begin{subfigure}[b]{0.45\textwidth}
        \centering
        \includegraphics[width=\linewidth]{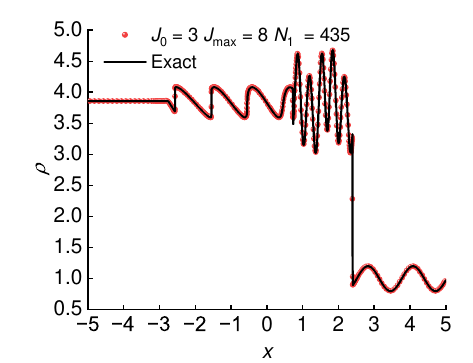}
        \caption{$J_{\max}=8$.}
        \label{fig:shockTurbulencediffJmax8}
    \end{subfigure}

    \caption{Numerical solutions at $t=1.8$ obtained by the adaptive $\mathrm{AMAIW_4T_2\mbox{-}BVD}$ scheme for the
    shock--turbulence interaction problem with $\epsilon_0=1.0\times10^{-3}$ and different finest resolution levels.}
    \label{fig:shockTurbulencediffJmax}
\end{figure}

\begin{figure}[!htbp]
    \centering

    \begin{subfigure}[b]{0.45\textwidth}
        \centering
        \includegraphics[width=\linewidth]{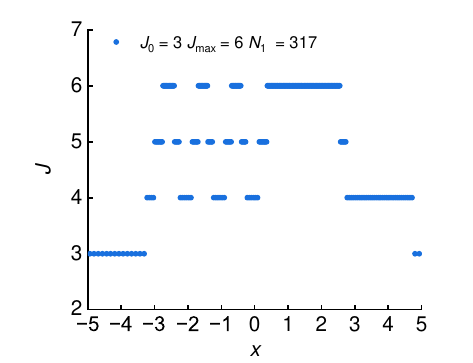}
        \caption{$J_{\max}=6$.}
        \label{fig:shockTurbulencecellJmax6}
    \end{subfigure}
    \hfill
    \begin{subfigure}[b]{0.45\textwidth}
        \centering
        \includegraphics[width=\linewidth]{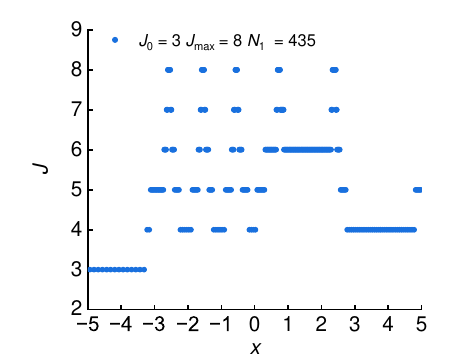}
        \caption{$J_{\max}=8$.}
        \label{fig:shockTurbulencecellJmax8}
    \end{subfigure}

    \caption{Active leaf cell distributions at $t=1.8$ obtained by the adaptive $\mathrm{AMAIW_4T_2\mbox{-}BVD}$ scheme
    for the shock--turbulence interaction problem with $\epsilon_0=1.0\times10^{-3}$ and different finest resolution
    levels.}
    \label{fig:shockTurbulencecell}
\end{figure}

\subsubsection{Two interacting blast waves}
The two interacting blast waves problem, introduced by Woodward and Colella~\cite{Woodward1984}, is a challenging
benchmark for numerical schemes due to the complex interactions between strong shock waves, contact discontinuities, and
rarefaction waves. The initial condition is given by
\begin{equation}
(\rho_0,u_0,p_0)=
\begin{cases}
(1,\,0,\,1000), & 0 \leq x < 0.1,\\
(1,\,0,\,0.01), & 0.1 \leq x < 0.9,\\
(1,\,0,\,100), & 0.9 \leq x < 1 .
\end{cases}
\end{equation}

Reflective boundary conditions are imposed at $x=0$ and $x=1$, following Ref.~\cite{HARTEN1987}. In the adaptive
framework, special care should be taken when constructing the exterior extension cell set near the physical boundary.
Specifically, the exterior cells must be generated as exact mirror images of the corresponding interior cells with
respect to the boundary. This mirror-symmetric extension ensures that the reconstructed boundary states and numerical
fluxes remain consistent with the prescribed reflective boundary condition, thereby preserving the local discrete
conservation balance near the boundary.

The reference solution at $t=0.038$ is computed by the WENO5-JS scheme~\cite{JiangShu1996} with 16384 uniform cells. The
density profiles obtained by the uniform-cell $\mathrm{AIW_4T_2\mbox{-}BVD}$ scheme and the adaptive
$\mathrm{AMAIW_4T_2\mbox{-}BVD}$ scheme are shown in Fig.~\ref{fig:tbdiffscheme}. Both schemes give satisfactory results
and sharply resolve the complex discontinuity structures without noticeable spurious oscillations. It should be noted
that, for this problem, polynomial-based shock-capturing schemes with BVD reconstruction have been reported to generate
obvious overshoots~\cite{Deng2019,Deng2020}. In contrast, the proposed wavelet upwind scheme combined with BVD
reconstruction yields non-oscillatory solutions without overshoots. This behavior benefits from the inherent low-pass
filtering property of the scaling functions, which enables the scheme to suppress high-frequency numerical perturbations
near strong discontinuities.

Moreover, the adaptive $\mathrm{AMAIW_4T_2\mbox{-}BVD}$ scheme gives results that are in agreement with those of the
corresponding uniform-cell $\mathrm{AIW_4T_2\mbox{-}BVD}$ scheme with $J=J_{\max}$. This agreement verifies the
effectiveness of performing the BVD reconstruction on the finest-level leaf cells in the proposed adaptive framework.
Meanwhile, the adaptive scheme achieves this high-resolution shock-capturing performance using a much smaller number of
active leaf cells, demonstrating its ability to resolve complex wave structures efficiently in a sparse form. These
results indicate that the proposed wavelet upwind schemes can robustly solve complex wave-interaction problems involving
strong shocks, and the adaptive version provides an efficient and stable approach for such challenging problems.

Finally, the numerical results at $t=0.038$ obtained by the proposed adaptive scheme are compared with those of the
WENO5-Z scheme~\cite{Borges2008,Don2013}, as shown in Fig.~\ref{fig:tbwWENOdiffJmax}. The number of uniform cells used
in the WENO5-Z scheme is selected so that the main wave structures, including the peak and valley profiles, are
reasonably comparable to those of the proposed adaptive wavelet scheme. Both methods capture the overall wave pattern
without visible spurious oscillations or overshoots. However, the adaptive $\mathrm{AMAIW_4T_2\mbox{-}BVD}$ scheme
resolves much sharper contact discontinuities around $x\approx0.59$ and $x\approx0.80$, requiring only about three
active leaf cells across the sharp transition and significantly fewer degrees of freedom over the whole domain. These
results demonstrates the superior contact discontinuity capturing capability of the proposed adaptive multiresolution
wavelet upwind scheme through a sparse representation.

\begin{figure}
    \centering

    \begin{subfigure}[b]{0.45\textwidth}
        \centering
        \includegraphics[width=\linewidth]{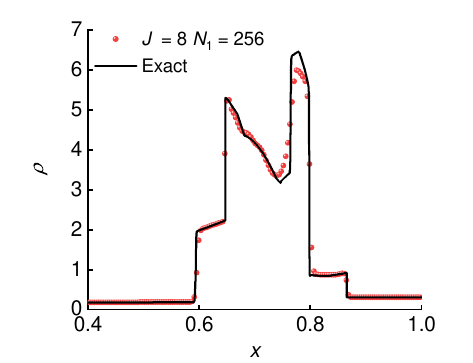}
        \caption{Uniform-cell scheme with $J=8$.}
        \label{fig:tbdiffJ8}
    \end{subfigure}
    \hfill
    \begin{subfigure}[b]{0.45\textwidth}
        \centering
        \includegraphics[width=\linewidth]{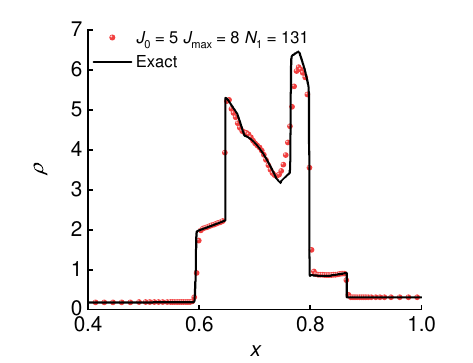}
        \caption{Adaptive scheme with $J_{\max}=8$.}
        \label{fig:tbdiffJmax8}
    \end{subfigure}

    \vspace{0.3cm}

    \begin{subfigure}[b]{0.45\textwidth}
        \centering
        \includegraphics[width=\linewidth]{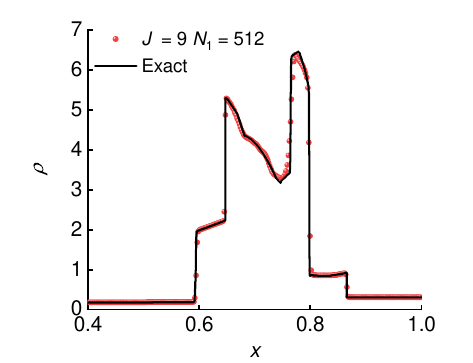}
        \caption{Uniform-cell scheme with $J=9$.}
        \label{fig:tbdiffJ9}
    \end{subfigure}
    \hfill
    \begin{subfigure}[b]{0.45\textwidth}
        \centering
        \includegraphics[width=\linewidth]{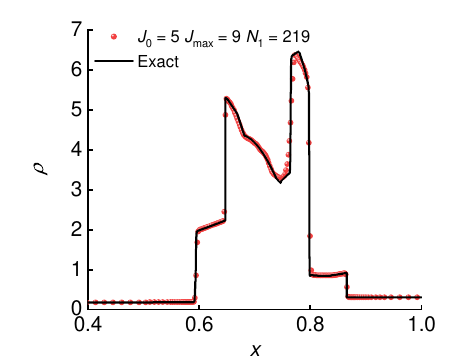}
        \caption{Adaptive scheme with $J_{\max}=9$.}
        \label{fig:tbdiffJmax9}
    \end{subfigure}

    \vspace{0.3cm}

    \begin{subfigure}[b]{0.45\textwidth}
        \centering
        \includegraphics[width=\linewidth]{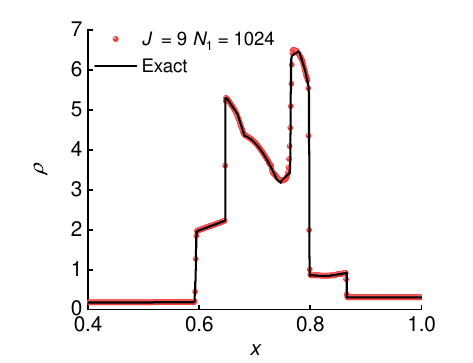}
        \caption{Uniform-cell scheme with $J=10$.}
        \label{fig:tbdiffJ10}
    \end{subfigure}
    \hfill
    \begin{subfigure}[b]{0.45\textwidth}
        \centering
        \includegraphics[width=\linewidth]{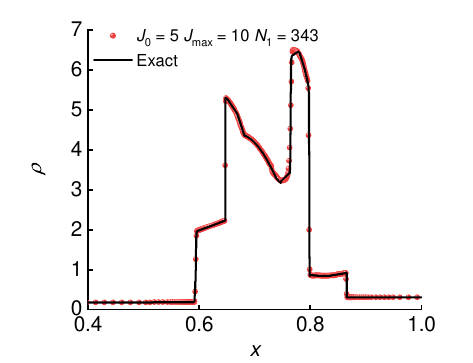}
        \caption{Adaptive scheme with $J_{\max}=10$.}
        \label{fig:tbdiffJmax10}
    \end{subfigure}
    
    \caption{Numerical solutions at $t=0.038$ of the two interacting blast waves problem obtained by the uniform-cell
$\mathrm{AIW_4T_2\mbox{-}BVD}$ scheme at $J=8,9,10$ and the adaptive $\mathrm{AMAIW_4T_2\mbox{-}BVD}$ scheme with
$\epsilon_0=1.0\times10^{-4}$, and $J_{\max}=8,9,10$.}
    \label{fig:tbdiffscheme}
\end{figure}

\begin{figure}
    \centering

    \begin{subfigure}[b]{0.45\textwidth}
        \centering
        \includegraphics[width=\linewidth]{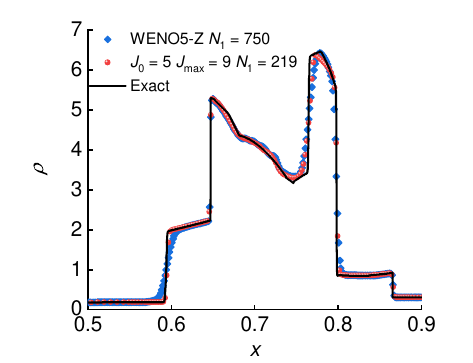}
        \caption{$J_{\max}=9$.}
        \label{fig:tbwWENOdiffJmax9}
    \end{subfigure}
    \hfill
    \begin{subfigure}[b]{0.45\textwidth}
        \centering
        \includegraphics[width=\linewidth]{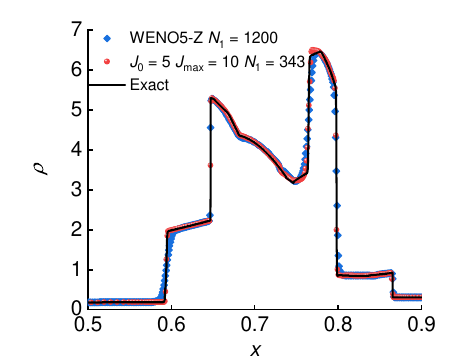}
        \caption{$J_{\max}=10$.}
        \label{fig:tbwWENOdiffJmax10}
    \end{subfigure}

    \caption{Comparison of the numerical solutions at $t=0.038$ of the two interacting blast waves problem obtained by
    the adaptive $\mathrm{AMAIW_4T_2\mbox{-}BVD}$ scheme with $\epsilon_0=1.0\times10^{-4}$ , and the WENO5-Z scheme.}
    \label{fig:tbwWENOdiffJmax}
\end{figure}
\FloatBarrier

\section{Conclusion}
\label{Sec:Conclusion}
In this study, a fourth-order conservative adaptive multiresolution average-interpolating wavelet upwind scheme with BVD
reconstruction has been developed within the finite volume framework. Asymmetric average-interpolating wavelets with
upwind properties are employed to construct the conservative discretization of the governing equations, while symmetric
average-interpolating wavelets are used for multiresolution decomposition and reconstruction in the adaptive procedure.
This formulation fully exploits the upwind stability of asymmetric wavelets and the data-compression efficiency of
symmetric wavelets within a unified cell-average wavelet framework.

A distinctive feature of the present method is that both the conservative discretization and the adaptive
multiresolution representation are constructed in terms of cell-average quantities. As a result, strict conservation can
be maintained during both numerical evolution and adaptive cell redistribution. Moreover, unlike hybrid adaptive wavelet
schemes, the present formulation avoids additional ghost-cell marking and reconstruction near coarse--fine cell
interfaces by evaluating numerical fluxes directly based on the wavelet multiresolution approximation. BVD
reconstruction is applied only at the finest resolution level, enabling effective non-oscillatory shock capturing while
avoiding unnecessary reconstruction in smooth regions.

Representative benchmark tests have demonstrated the accuracy, conservativity, error-control capability, computational
efficiency, and shock-capturing performance of the adaptive wavelet upwind scheme. The numerical results show that
shocks and contact discontinuities are sharply resolved without visible spurious oscillations or overshoots. Compared
with uniform-cell computations, the adaptive scheme significantly reduces the number of active degrees of freedom while
maintaining sufficient resolution in regions containing localized flow features. These results indicate that the
proposed conservative adaptive multiresolution average-interpolating wavelet upwind scheme provides an efficient,
robust, and reliable adaptive approach for high-resolution simulations of compressible flows. Future work will focus on
extending the present conservative adaptive wavelet upwind framework to multidimensional compressible flows.

\printcredits

\section*{Acknowledgements}
This work was supported by the National Natural Science Foundation of China (Key Program, Grant No. 12532010) and
Fundamental and Interdisciplinary Disciplines Breakthrough Plan of the Ministry of Education of China (Grant No.
JYB2025XDXM105).

\bibliographystyle{cas-model2-nums}

\bibliography{cas-refs}

\end{document}